\documentclass{amsart}[12pt]
\usepackage{amsthm}
\usepackage{mathrsfs}
\usepackage{amsfonts}
\usepackage{comment}
\usepackage{amssymb}
\usepackage{amsmath}
\usepackage{xcolor}
\usepackage{eufrak}
\usepackage{comment}
\usepackage[colorlinks = true, citecolor = blue]{hyperref}
\usepackage{mathtools}
\usepackage{eucal}
\usepackage{tikz-cd}
\usepackage{pdflscape}
\usepackage{multirow}
\usepackage[margin=1.25in]{geometry}
\usepackage{cite}
\usepackage{tikz}
\usepackage{float}
\usepackage{ltablex}

\numberwithin{equation}{section}
\numberwithin{figure}{section}
\numberwithin{table}{section}
\usepackage{breqn}
\allowdisplaybreaks
\usepackage[OT2,T1]{fontenc}
\usepackage{cleveref}
\usepackage[shortlabels]{enumitem}
\usepackage{stmaryrd}
\newtheorem{theorem}{Theorem}

\newtheorem{Theorem}{Theorem}[section]

\newtheorem{corollary}[Theorem]{Corollary}

\newtheorem{lemma}[Theorem]{Lemma}
\newtheorem{proposition}[Theorem]{Proposition}

\theoremstyle{definition}
\newtheorem{Example}[Theorem]{Example}

\newtheorem*{definition}{Definition}
\theoremstyle{remark}
\newtheorem{remark}[Theorem]{Remark}

\DeclareMathOperator{\Aut}{Aut}

\DeclareMathOperator{\sm}{sm}
\DeclareMathOperator{\rank}{rank}
\DeclareMathOperator{\wee}{sm}
\DeclareMathOperator{\med}{med}
\DeclareMathOperator{\giant}{big}
\DeclareMathOperator{\Spec}{Spec}
\DeclareMathOperator{\ev}{ev}
\DeclareMathOperator{\aff}{aff}
\DeclareMathOperator{\irr}{irr}

\newcommand{\Q}{\mathbb{Q}}

\newcommand{\C}{\mathbb{C}}
\newcommand{\Z}{\mathbb{Z}}
\newcommand{\R}{\mathbb{R}}
\newcommand{\F}{\mathbb{F}}
\newcommand{\A}{\mathbb{A}}

\renewcommand{\P}{\mathbb{P}}
\renewcommand{\a}{\textbf{a}}

\renewcommand{\c}{\textbf{c}}

\newcommand{\wt}[1]{\widetilde{#1}}
\newcolumntype{Y}{>{\centering\arraybackslash}X}
\newcommand{\denomtimes}{\\ \times}

\newcommand\longvdots[1]{\raisebox{1em}{\rotatebox{-90}{\hbox to #1 {\tiny\dotfill}}}}

\title{On the proportion of locally soluble superelliptic curves}

\author{Lea Beneish}
\address{Lea Beneish\\
Department of Mathematics\\
University of California, Berkeley\\
970 Evans Hall\\
Berkeley, CA 94720\\
United States}
\email{leabeneish@math.berkeley.edu}

\author{Christopher Keyes} 
\address{Christopher Keyes \\
	Department of Mathematics\\
	Emory University\\
	400 Dowman Drive\\
    Atlanta, GA 30322\\
    United States}
\email{christopher.keyes@emory.edu}

\begin{document}

\begin{abstract} We investigate the proportion of superelliptic curves that have a $\mathbb{Q}_p$ point for every place $p$ of $\mathbb{Q}$. We show that this proportion is positive and given by the product of local densities, we provide lower bounds for this proportion in general, and for superelliptic curves of the form $y^3 = f(x,z)$ for an integral  binary form $f$ of degree 6, we determine this proportion to be 96.94\%. More precisely, we give explicit rational functions in $p$ for the proportion of such curves over $\Z_p$ having a $\mathbb{Q}_p$-point.
\end{abstract}

\maketitle 

\section{Introduction}

In 1983, Faltings proved that if $C$ is a curve of genus $g > 1$ over $\mathbb{Q}$, then the set of $\mathbb{Q}$-rational points of $C$, $C(\mathbb{Q})$, is finite \cite{faltings}. The questions of counting and classifying the $\mathbb{Q}$-rational points of a given curve and the study of how $C(\mathbb{Q})$ varies as $C$ varies in families are areas of active work. For example, there has been recent work on sparsity of rational points on hyperelliptic and superelliptic curves by Ellenberg--Hast, Poonen--Stoll, Shankar--Wang, and Stoll \cite{ellenberghast,poonenstollhec,shankarwang,stolluniformbounds}. See also \cite{Chabauty, Coleman, Kim_2005, Kim_2009, McPoonen} for work on the Chabauty--Coleman method and its generalization, the Chabauty--Kim method. In studying the $\mathbb{Q}$-rational points of a curve it is often useful to examine the $\mathbb{Q}_{p}$-rational points of the curve, for $p$ a place of $\mathbb{Q}$.

In particular, one can ask when a curve is \textit{locally soluble}, that is, if the curve has a point over $\mathbb{Q}_{p}$ for every place $p$ of $\mathbb{Q}$ (including the infinite place, $\Q_{\infty} = \R$). Poonen--Stoll \cite{poonenstoll_short, poonenstoll_long} using the sieve of Ekedahl \cite{Ekedahl} have shown that this proportion is positive in the case of hyperelliptic curves. Bright--Browning--Loughran \cite{BBL} have generalized this method to certain families of varieties over number fields. Bhargava--Cremona--Fisher \cite{BCF_plane_cubic, BCF_genus_one} determined the proportion of locally soluble plane cubics ($\approx75.96\%$) and genus one curves ($\approx 97.3\%$) by expressing the local densities as rational functions of $p$ (in forthcoming work, they compute this quantity for hyperelliptic curves of genus $g>1$). In higher dimensions, Bright--Browning--Loughran, Fisher--Ho--Park and Poonen--Voloch have studied local solubility of various hypersurfaces \cite{BBL, fisherhopark, poonenvoloch}. Further, Browning \cite{Browning} studied certain cubic hypersurfaces in $\P^3$, giving explicit rational functions for the local densities to show that nearly all ($\approx 99\%$) are locally soluble, and moreover proving that a positive proportion of such surfaces have global points. 

In recent years, there have been several works studying the arithmetic of superelliptic curves such as \cite{Vishal, Vishalthesis, ellenberghast, beneishkeyesfields}, including work of Watson \cite{watson} on the failure of the Hasse principle in twist families of superelliptic curves. In this article, we study the proportion of locally soluble superelliptic curves. 

For a positive integer $m \geq 2$, a \textit{superelliptic curve $C_f/\Q$ of exponent $m$} is a projective curve with affine equation
\begin{equation}\label{eq:sec}
	C_f \colon y^m = f(x) = \sum_{i=0}^d c_ix^i.
\end{equation}
Such a curve possesses a cyclic degree $m$ map to the projective line $\P^1$ defined over $\Q$, sending a point $(x,y) \mapsto x$. The genus of such a curve $C_f$ over an algebraically closed field $k$ is given by 
\begin{equation}
\label{eq:genus}
	 g(C_f)=\frac{1}{2} \left( m(|B|-2)- \sum\limits_{\alpha \in B} (m, r_{\alpha})\right)+1, 
\end{equation}
where we denote by $B$ the set of branch points of the map to $\mathbb{P}^1$ and we denote by $r_{\alpha}$ the order of $\alpha$ as a root of $f(x)$. The value $r_{\infty}$ is analogously defined and we use that $(m, r_\infty) = (m, \deg(f))$. When $m \mid d$, equivalent to the superelliptic map $C_f \to \P^1$ being unramified at infinity, $C_f$ embeds as a closed subsvariety into the weighted projective space $\mathbb{P}\left(1,\frac{d}{m},1\right)$ as  the vanishing set of
\begin{equation}
\label{eq:projectiveeqn} 
	F(x,y,z) =  y^m-f(x,z)= y^m-\sum_{i=0}^d c_ix^iz^{d-i}.
\end{equation}

We study the local solubility of these curves in three ways. First, by showing that the proportion of locally soluble superelliptic curves is positive and given by a product of local factors. Next, we prove explicit lower bounds for the proportion of locally soluble superelliptic curves in terms of an Euler product depending on $m$. Finally, we give an explicit rational function for the local factors, in the case of superelliptic curves with $m=3$ and $d=6$.

We now define the proportion of locally soluble superelliptic curves (i.e.\ the probability that a superelliptic curve is locally soluble). Given a superelliptic curve $C_f$ of exponent $m$ and degree $d$ divisible by $m$ as in \eqref{eq:projectiveeqn}, we define its height $h(C_f)$ to be the height of its defining polynomial $h(f):=\max \{ |c_0|, \dots, |c_d| \} $, the maximum of the absolute values of the coefficients. Then we define the proportion as 
\begin{equation}\label{eq:def_density} 
\rho_{m,d}=\lim\limits_{B\to \infty} \dfrac{\#\{ C_f\mid C_f \text{ is locally soluble and  }h(C_f)\leq B \} }{\#\{ C_f\mid h(C_f) \leq B \} }, 
\end{equation}
where here $C_f$ ranges over superelliptic curves of exponent $m$ and degree $d$. We refer to $\rho_{m,d}$ as the \textit{adelic density} of equations of locally soluble such curves.

 To define the corresponding local densities, let $p$ be a prime and $\mu_p$ be a Haar measure on the additive group $\Z_p^{d+1}$, normalized such that $\mu_p \left(\Z_p^{d+1} \right) = 1$.  We define
\begin{equation}
\label{eq:def_local_density}
	\rho_{m,d}(p) = \mu_p \left( \left\{ (c_0, \ldots, c_d) \in \Z_p^{d+1} \mid y^m = c_dx^d + \cdots + c_0z^d \text{ has a } \Q_p\text{-point} \right\}\right).
\end{equation}
For the place at infinity, we let $\mu_\infty$ denote the usual Euclidean measure on $\R^{d+1}$ and set 
\begin{equation}
\label{eq:def_local_density_infty}
	\rho_{m,d}(\infty) = \frac{1}{2^{d+1}} \mu_\infty \left( \left\{ (c_0, \ldots, c_d) \in [-1,1]^{d+1} \mid y^m = c_dx^d + \cdots + c_0z^d \text{ has an } \R\text{-point} \right\}\right).
\end{equation}

Our first result shows that the proportion of locally soluble superelliptic curves of exponent $m$ and degree $d$, $\rho_{m,d}$ is positive using the methods of Poonen--Stoll and Bright--Browning--Loughran \cite{BBL,poonenstoll_short, poonenstoll_long} and further that the adelic density $\rho_{m,d}$ can be computed as a product of the local densities.

\begin{theorem}\label{thm:pos_prop}
	Fix integers $m \geq 2$ and $d$ divisible by $m$, such that $(m,d) \neq (2,2)$. Then we have $\rho_{m,d} > 0$, i.e.\ a positive proportion of superelliptic curves over $\mathbb{Q}$ of the form \eqref{eq:projectiveeqn} are everywhere locally soluble.
	
	Moreover, the adelic density may be computed as the product of local densities,
	\[\rho_{m,d} = \rho_{m,d}(\infty) \prod_{p} \rho_{m,d}(p).\]
\end{theorem}

\begin{remark} A version of this theorem holds over number fields as well as over $\mathbb{Q}$; for the statement (and proof) of this theorem in full generality see Corollary \ref{cor:pos_prop_NF}.
\end{remark}

For the remainder of the paper we focus on bounding and computing the local densities $\rho_{m,d}(p)$ for the finite primes $p$ of $\Q$, using these to compute or bound the adelic densities $\rho_{m,d}$. Given $m,d$, we find explicit lower bounds for $\rho_{m,d}(p)$ in Propositions \ref{prop:bound_p=1_large}, \ref{prop:bound_pneq1}, \ref{prop:bound_matrix}, and \ref{prop:bound_p|m}, and an upper bound for $\rho_{m,d}(2)$ in Lemma \ref{lem:upper_bound_p=2}. By Theorem \ref{thm:pos_prop}, this yields upper and lower bounds for $\rho_{m,d}$; see Corollary \ref{cor:lower_bound} and Examples \ref{ex:m=3} and \ref{ex:m=5}. 

These bounds are sufficient to then bound the limiting behavior of $\rho_{m,d}$ for fixed $m$ as $d \to \infty$; see Corollaries \ref{cor:liminf_uniform}, \ref{cor:limit_rho}, and \ref{cor:limsup_uniform}. We summarize these below for prime $m$.

\begin{theorem}\label{thm:lower_bound_limit}
	Fix a prime $m$ and suppose $m \mid d$. The limiting behavior of $\rho_{m,d}$ as $d\to\infty$ may be described by 
	\[\liminf_{d \to \infty} \rho_{m,d} \geq \left(1 - \frac{1}{m^{m+1}}\right) \prod_{p \equiv 1 (m)} \left( 1 - \left(1-\frac{p-1}{mp}\right)^{p+1} \right) \prod_{p \not\equiv 0,1 (m)} \left(1 - \frac{1}{p^{2(p+1)}}\right).\]		
	When $m > 2$, we have the following numerical estimates, uniform in $m$:
		\[\liminf_{d \to \infty} \rho_{m,d} \geq 0.83511\]
		and
		\[\limsup_{d \to \infty} \rho_{m,d} \leq 1 - \frac{1}{2^9} \approx 0.99804.\]
\end{theorem}

\begin{remark}\label{rem:composites}
	When the exponent $m$ is composite, the methods of bounding $\rho_{m,d}(p)$ from below discussed in Section \ref{sec:lower_bounds} still apply, but the resulting expressions for a lower bound of $\rho_{m,d}$ may be less compact. See Example \ref{ex:m=4} for a discussion in the case of $m=4$. 
\end{remark}

After giving several examples of lower bounds for various pairs $m,d$, we employ methods similar to those of Bhargava--Cremona--Fisher to compute exact formulas for the local densities $\rho_{3,6}(p)$ for $p$ sufficiently large, the first case of superelliptic curves not already addressed by \cite{BCF_plane_cubic}, \cite{BCF_genus_one}, or their forthcoming paper on hyperelliptic curves of higher genus.

 We give this local density as a rational function in $p$ depending on the residue class of $p$ modulo $3$ (assuming $p$ is sufficiently large). We compute lower bounds for the solubility when $p$ is small using a brute force search, allowing us to give an approximate proportion of locally soluble $m=3$, $d=6$ superelliptic curves.

\begin{theorem}\label{thm:exact_rho_3_6}
	For superelliptic curves of the form \eqref{eq:projectiveeqn} with $m=3$ and $d=6$, the exact value of $\rho_{3,6}$ is about 96.94\%.
	
	Moreover, there exist rational functions $R_1(t)$ and $R_2(t)$ such that the local density $\rho_{3,6}(p)$ is given by
	\[\rho_{3,6}(p) = \begin{cases} R_1(p), & p \equiv 1 \pmod{3} \text{ and } p > 43\\
	R_2(p), & p \equiv 2 \pmod{3} \text{ and } p > 2.\end{cases}\]
	The explicit formula is given in \eqref{eq:rho}. The asymptotic behavior of $R_1(t)$ and $R_2(t)$ is described by
	\begin{align*}
		1- R_1(t) & \sim \frac{2}{3}t^{-4},\\
		1- R_2(t) & \sim \frac{53}{144} t^{-7}.
	\end{align*}
	
\end{theorem}

\begin{remark} The proof, given in Section \ref{sec:m=3_d=6}, involves relating $\rho_{3,6}$ to several quantities. These relations were implemented in Sage \cite{sagemath} to solve for the explicit formula. The Sage Notebook used for these calculations can be \href{https://github.com/c-keyes/Density-of-locally-soluble-SECs/blob/bd6a8b39ea8c63bf8e7a847063c70998d01ee8aa/SEC_rho36_23Aug21.ipynb}{found here}.
\end{remark}

\begin{remark} In contrast to the work of Bhargava--Cremona--Fisher, where the local density at a prime $p$ for a genus one curve is given by a degree-9 rational function of $p$ \cite{BCF_genus_one} and the that of a plane cubic curve is given by a degree-12 rational function of $p$ \cite{BCF_plane_cubic}, the local density of our superelliptic curves $\rho_{3,6}(p)$ is a degree-57 rational function of $p$. This situation produced considerably more cases to check. Moreover, one can see that the number of such cases increases quickly in both $m$ and $d$, and certain independence arguments we make do not hold for $d\geq 8$; see Remark \ref{rem:lifting_double_roots}. For this reason, we restricted our attention to $m=3, d=6$ superelliptic curves for the exact expression.
\end{remark}

This paper is organized as follows. Section \ref{sec:pos_prop} contains the proof of a more general version of Theorem \ref{thm:pos_prop}, that the proportion of locally soluble superelliptic curves over any number field $k$ is positive. Section \ref{sec:lower_bounds} contains the proof of Theorem \ref{thm:lower_bound_limit},  the lower bounds for the proportion of locally soluble superelliptic curves with exponent $m$ and degree $d$ with $m\mid d$ and several examples of lower bounds for $\rho_{m,d}(p)$ for specific pairs $m,d$. This is contrasted in Section \ref{sec:upper_bounds} with a discussion of upper bounds for the local densities, leading to a general upper bound for $\rho_{m,d}$. Section \ref{sec:m=3_d=6} contains the proof of the exact formula for the local densities $\rho_{3,6}(p)$. Appendix \ref{appx:code} contains an explanation of a computational approach to bounding the local densities $\rho_{m,d}(p)$ for small primes. Appendix \ref{appx:counting_forms} contains the number of degree $2\leq d\leq 6$ binary forms $f(x,z)$ over $\mathbb{F}_p$ having the different possible factorization types. Appendix \ref{appx:explicit_formulas} contains the explicit expressions for the rational functions from Section \ref{sec:m=3_d=6}. 

\subsection*{Acknowledgments} The authors are grateful to Manjul Bhargava, Henri Darmon, Jackson Morrow, and David Zureick-Brown for helpful conversations. The authors would also like to thank Tim Browning, Jackson Morrow, Jeremy Rouse, Lori Watson, and David Zureick-Brown for their valuable feedback on an earlier draft. In addition, we thank an anonymous referee for their thorough reading of the manuscript and thoughtful suggestions that improved the paper.

\section{The proportion is positive}
\label{sec:pos_prop}

In proving Theorem \ref{thm:pos_prop}, we can in fact produce a more general statement about superelliptic curves over number fields. For the remainder of this section, let $m \geq 2$ be an integer and $d$ be a multiple of $m$. Let $k/\Q$ be an algebraic number field, with $\mathcal{O}_k$ denoting the ring of integers, $k_v$ denoting the $v$-adic completion at a place $v$, and $\mathbf{A}_k$ denoting the ring of adeles.

\begin{definition}\label{def:locally_soluble}
	A scheme $X/k$ is \textbf{everywhere locally soluble} if $X(k_v) \neq \emptyset$ for all places $v$ of $k$. 
\end{definition}

If $X$ is proper over $k$, then the adelic points of $X$ are the product of the $k_v$-points,
\[X(\mathbf{A}_k) = \prod_{v} X(k_v).\]
In this case we have that $X$ is everywhere locally soluble if and only if $X(\mathbf{A}_k) \neq \emptyset$. Note that a superelliptic curve $C_f/k$ is projective, and therefore proper over $k$. 

To define the density of superelliptic curves $C_f$ of the form \eqref{eq:projectiveeqn} with integral coefficients $(c_i)_{i=0}^d = \c \in \mathcal{O}_k^{d+1}$ which are locally soluble, we will need a suitable way to take limits, which specializes to the usual density over $\Q$. Let $k_\infty = \mathcal{O}_k \otimes_\Z \R$ and take $\Psi \subset k_\infty^{d+1}$ to be a bounded subset of positive measure whose boundary has measure zero, $\mu_\infty(\partial \Psi) = 0$. One can then take a limit to define the density
\begin{equation}\label{eq:density_number_field}
	\rho_{m,d,k, \Psi} = \lim_{B \to \infty} \frac{\# \left \lbrace \c \in \mathcal{O}_k^{d+1} \cap B \Psi^{d+1} \mid C_f(\mathbf{A}_k) \neq \emptyset \right \rbrace}{\# \left \lbrace \c \in \mathcal{O}_k^{d+1} \cap B \Psi^{d+1} \right \rbrace}.
\end{equation}
Note that in the case of $k = \Q$, we have $k_\infty = \R$ and may choose $\Psi = [-1,1]$, so that $\rho_{m,d,\Q, \Psi}$ takes the form
\[\rho_{m,d,\Q, [-1,1]} = \lim_{B \to \infty} \frac{ \# \left \lbrace \c \in \Z^{d+1} \cap [-B,B]^{d+1} \mid C_f(\mathbf{A}_k) \neq \emptyset \right \rbrace}{\# \left\lbrace \c \in \Z^{d+1} \cap [-B, B]^{d+1} \right \rbrace},\]
which agrees with \eqref{eq:def_density} upon observing $h(C_f) \leq B$ precisely when $\c \in [-B, B]^{d+1}$. Note also that this definition depends on the choice of $\Psi$; for example taking $k=\R$ and $\Psi = [0,1]^{d+1}$ instead would produce a different answer. Asking for $\Psi$ to be convex and symmetric in $k_\infty$ is likely desirable.

To extend the definitions of the local densities, for a finite place $v \nmid \infty$, we have
\[\rho_{m,d,k}(v) = \mu_v \left( \left\{ \c \in (\mathcal{O}_k)_v^{d+1} \mid y^m = c_dx^d + \cdots + c_0z^d \text{ has a } k_v\text{-point} \right\} \right),\]
where $\mu_v$ is a normalized Haar measure on $(\mathcal{O}_k)_v^{d+1}$, thus extending \eqref{eq:def_local_density}. At the infinite places, we take
\begin{align*}
	\rho_{m,d,k,\Psi}(\infty) &= \frac{\mu_\infty\left( \left\{ \c \in \Psi \mid y^m = c_dx^d + \cdots + c_0z^d \text{ has a } k_\infty\text{-point} \right\}\right)}{\mu_\infty(\Psi)}
\end{align*}
This could further be broken down into a product of local densities for $v \mid \infty$, but it is not necessary for our analysis.

We can use ideas of Bright--Browning--Loughran \cite{BBL} to show that $\rho_{m,d,k, \Psi}$ exists, is nonzero, and is computable via a product of local densities. This was already known for hyperelliptic curves over $\Q$ (i.e.\ the $\rho_{2,d,\Q, [-1,1]}$ case) by work of Poonen and Stoll \cite{poonenstoll_short}. In particular, we will need the following result, which is a slight weakening of \cite[Theorem 1.4]{BBL}.

\begin{Theorem}[see {\cite[Theorem 1.4]{BBL}}]
\label{thm:BBL}
	Let $k$ be a number field and $\pi \colon X \to \A^n$ a dominant quasiprojective $k$-morphism with geometrically integral generic fiber, and let $X_P$ denote the fiber of $\pi$ over a point $P \in \A^n_k$. Assume further that
	\begin{enumerate}[label = (\roman*)]
		\item the fiber of $\pi$ above each codimension 1 point of $\A^n$ is geometrically integral,
		\item $X(\mathbf{A}_k) \neq \emptyset$,
		\item for each real place $v$ of $k$, we have $B \pi(X(k_v)) \subseteq \pi(X(k_v))$ for all $B \geq 1$.
	\end{enumerate}
	Let $\Psi' \subset k_\infty^n$ be a bounded subset of positive measure lying in $\pi(X(k_\infty))$ whose boundary has measure zero. Then the limit
	\begin{equation}\label{eq:bbl_limit}
		\lim_{B \to \infty} \frac{ \# \left \lbrace P \in \mathcal{O}_k^n \cap B\Psi' \mid X_P(\mathbf{A}_k) \neq \emptyset \right\rbrace}{ \#\left \lbrace P \in \mathcal{O}_k^n \cap B\Psi' \right\rbrace} \end{equation}
	exists, is nonzero, and is equal to a product of local densities, $\prod_{v \nmid \infty} \mu_v\left( \left\{ P \in (\mathcal{O}_k)_v^n \mid X_P(\Q_p) \neq \emptyset \right\} \right)$.
\end{Theorem}

We now translate our problem into this language. Consider the affine space $\A^{d+1}_k = \Spec k[c_0, \ldots, c_d]$ and let $\mathcal{P}_k = \P_k\left(1,\frac{d}{m},1\right)$ be the weighted projective space into which curves of the form $C_f$ naturally embed, with coordinates $[x : y : z]$. Thus the vanishing set of $F(x,y,z) = \eqref{eq:projectiveeqn}$ gives a variety $X  \subset \A^{d+1}_k \times \mathcal{P}_k$. We have a natural map $\pi \colon X \to \A^{d+1}_k$, where the fiber $X_P$ over a $k$-point $P \in \A^{d+1}_k(k)$ corresponds to a specialization $C_f$, where the coefficients of $f$ are given by the coordinates of $P$.

If $\Psi \subset k_\infty^{d+1}$ is a bounded subset of positive measure with $\mu_\infty(\partial\Psi) =0$, we take $\Psi' = \Psi \cap \pi(X(k_\infty))$; that is, the polynomials $f$ with coefficients given in $\Psi$ such that $C_f$ has points over all archimedean completions of $k$. Thus if Theorem \ref{thm:BBL} holds, we have
\begin{align*}
	\rho_{m,d,k,\Psi} &= \lim_{B \to \infty} \frac{ \# \left \lbrace P \in \mathcal{O}_k^{d+1} \cap B\Psi' \mid X_P(\mathbf{A}_k) \neq \emptyset \right\rbrace}{ \#\left \lbrace P \in \mathcal{O}_k^{d+1} \cap B\Psi' \right\rbrace} \cdot \frac{\#\left \lbrace P \in \mathcal{O}_k^{d+1} \cap B\Psi' \right\rbrace}{\#\left \lbrace P \in \mathcal{O}_k^{d+1} \cap B\Psi \right\rbrace}\\
	&= \rho_{m,d,k,\Psi}(\infty) \prod_{v \nmid \infty} \rho_{m,d,k}(v),
\end{align*}
as the ratio of the lattice points contained in $B\Psi'$ to $B\Psi$ approaches $\mu_\infty(\Psi')/\mu_\infty(\Psi)$ as $B \to \infty$. To apply Theorem \ref{thm:BBL}, we need to prove that $\pi$ and $\Psi'$ has the desired properties and verify (i), (ii), and (iii). We begin by characterizing when varieties cut out by equations of the form \eqref{eq:projectiveeqn} are geometrically integral.

\begin{lemma}
\label{lem:m-th_power_lemma}
    Fix $m \geq 2$ and $d$ divisible by $m$. Let $k$ be any base field. Suppose $f(x,z) \in k[x,z]$ is homogeneous of degree $d$ and and take $C_f$ the closed subvariety of $\P_k\left(1, \frac{d}{m}, 1\right)$ cut out by \eqref{eq:projectiveeqn}. The following are equivalent.
    \begin{enumerate}[label = (\alph*)]
        \item $f \neq ah^q$ for all prime divisors $q \mid m$, $a \in k$, and homogeneous degree $d/q$ polynomials $h(x,z) \in k[x,z]$.
        \item $f \neq g^q$ for all prime divisors $q \mid m$ and homogeneous degree $d/q$ polynomials $g(x,z) \in \overline{k}[x,z]$.
        \item $C_f$ is geometrically integral.
    \end{enumerate}
\end{lemma}

\begin{proof}
    For the (a) $\implies$ (b) direction, we prove the contrapositive. Suppose $f = g^q$ for some prime divisor $q \mid m$ and $g \in \overline{k}[x,z]$. We will find $a \in k$ and $h \in k[x,z]$ such that $f = ah^q$. For the moment, let us further assume that the characteristic of $k$ is prime to $q$.
    
    Write $g = a_0g_0$ where $a_0 \in \overline{k}$ and $g_0$ has leading coefficient 1 (i.e.\ the highest power of $x$ appears with coefficient 1). Then we have
	\[f = a_0^qg_0^q = a_0^q\left(b_{d/q}x^{d/q} + b_{d/q-1}x^{d/q-1}z + \cdots + b_0z^{d/q}\right)^q.\]
	Assume for convenience that $b_{d/q} = 1$; if not, it must be zero by our construction, and the proof proceeds identically, starting with the first nonzero value of $b_{d/q-i}$. The leading term of $f$ is $a_0^qx^d$, so $a_0^q \in k$. Set $a = a_0^q$. Now we examine the $x^{d-1}z$ term: it is $aqb_{d/q-1}x^{d-1}z$, so using the fact that $a$ and $q$ are units in $k$ (here we use $\mathrm{char}(k) \nmid q$), we see that $b_{d/q-1} \in k$. 
	
	Proceeding inductively, we show $b_{d/q-i} \in k$ for all $0 \leq i \leq d/q$. The $x^{d-i}z^i$ term of $f$ looks like 
	\[a(\ldots + qb_{d/q-i})x^{d-i}z^i,\]
	where the omitted terms consist only of $b_{d/q-j}$ for $j < i$, and hence are already known to be in $k$ by the induction hypothesis. Thus we conclude $b_{d/q-i} \in k$, showing that $g_0 \in k[x,z]$. Setting $h=g_0$, we have written $f = ah^q$ for $h$ defined over $k$.
	
	Suppose now that $\mathrm{char}(k) = q$ and write $g = b_{d/q}'x^{d/q} + \cdots + b_0'z^{d/q}$. We have
	\[f=g^q = (b_{d/q}')^qx^{d} + \cdots + (b_0')^qz^{d}.\]
	Hence $(b_i')^q \in k$ for all $i$. Moreover, the $q$-th power map $k^\times \to k^\times$ is an isomorphism, so we take $a = 1$ and $h = b_{d/q}x^{d/q} + \cdots + b_0z^{d/q}$ for the unique $b_i \in k$ such that $(b_i')^q = b_i^q$, finding that $f = ah^q$.
    
    To prove (b) $\implies$ (c), we claim it suffices to show that the standard open affine pieces $U\colon y^m = f(x,1)$ and $U' \colon y^m = f(1,z)$, obtained by pulling back the map $C_f \to \P^1$ along the standard affine patches, are geometrically integral. In particular, this implies that the stalks of the structure sheaf of $C_f$ (over $\overline{k}$) are integral domains. A straightforward computation shows that this sheaf has only the constants as its global sections, hence $C_f$ is geometrically connected. Since $C_f$ is Noetherian and nonempty, geometric connectedness and integral stalks suffice to ensure $C_f$ is geometrically integral (see e.g.\ \cite[Exercise 5.3.C]{vakil}).
    
    We now argue that $U = \Spec \overline{k}[x,y]/(y^m - f(x,1))$ is geometrically integral by exploiting its map to $\A^1 = \Spec \overline{k}[x]$. The same argument applies to $U'$. We compute the generic fiber of this map to be the spectrum of $\overline{k}(x) \otimes_{\overline{k}[x]} \overline{k}[x,y]/(y^m - f(x,1)) \simeq \overline{k}(x)[y]/(y^m - f(x,1))$. This is a field by our hypothesis (b); if not, $f(x,1) = g_0(x)^q$ for some prime $q \mid m$, and we have 
    \[f(x,z) = z^df(x/z,1) = z^dg_0(x/z)^q = \left(z^{d/q}g_0(x/z)\right)^q,\]
    violating (b).
    
    Finally, we verify that the natural map of rings $\overline{k}[x,y]/(y^m - f(x,1)) \to \overline{k}(x)[y]/(y^m - f(x,1))$ is injective. Taking $g(x,y)$ in the kernel of this map and assuming its degree in $y$ is less than $m$, we have
    \[g(x,y) = \sum_{0 \leq i < m} g_i(x) y^i = \left(\sum_{j \geq 0} h_j(x)y^j\right)(y^m - f(x,1)),\]
    where $h_i \in \overline{k}(x)$ for all $i$. Expanding the right hand side, we see $g_i(x) = -h_i(x)f(x,1)$ for $0 \leq i < m$ and $h_j(x) = h_{j+m}(x)f(x,1)$ for all $j$. Since $h_j(x) = 0$ for all $j \gg 0$, we must have $h_j = 0$ for all $j$, hence $g(x,y) = 0$. Thus $\overline{k}[x,y]/(y^m - f(x,1))$ injects into a field, and therefore must be an integral domain, making $U$ geometrically integral.
    
    The (c) $\implies$ (a) direction follows from the observation that if $f = ah^q$ then over $\overline{k}$ we have the nontrivial factorization of \eqref{eq:projectiveeqn}
    \[y^m - f(x,z) = \prod_{1 \leq i \leq q} \left(y^{m/q} - \zeta^i \alpha h(x,z)\right),\]
    for $\zeta$ a primitive $q$-th root of unity and $\alpha \in \overline{k}$ satisfying $\alpha^q = a$.
\end{proof}

\begin{lemma}\label{lem:properties_of_pi}
	Let $\pi \colon X \to \A^{d+1}$, as above, be considered as a morphism of $k$-varieties. Then $\pi$ is dominant, projective, and has geometrically integral generic fiber.
\end{lemma}

\begin{proof}
	The generic fiber of $\pi$ is the curve given by \eqref{eq:projectiveeqn}, but viewed as a closed subscheme of $\mathcal{P}_{k(c_0, \ldots, c_d)}$. This is geometrically integral by an application of Lemma \ref{lem:m-th_power_lemma} over the base field $k(c_0, \ldots, c_d)$ because the generic degree $d$ polynomial is not an $q$-th power over the algebraic closure of $k(c_0, \ldots, c_d)$ for any prime divisor $q \mid m$ (in fact it is squarefree, since the discriminant of such a polynomial is nonzero in $k(c_0, \ldots, c_d)$). Interpreting this (or rather $\Spec$ of the function field) as the generic point of $X$, we see that the generic point of $X$ maps to that of $\A^{d+1}_k$, so $\pi$ is dominant . For quasiprojectivity, we note that since $\mathcal{P}_k \to \Spec k$ is projective and projectivity is preserved under base change, we have $\A^{d+1}_k \times \mathcal{P}_k \to \A^{d+1}_k$ is projective. Closed embeddings are projective and projectivity is closed under composition, giving that $X \to \A^{d+1}_k \times \mathcal{P}_k \to \A^{d+1}_k$ is projective.
\end{proof}

In order to justify (i) of Theorem \ref{thm:BBL}, we first need to understand the $q$-th power map on polynomials for a prime $q$. Viewing $\A^{d+1}$ as the space of polynomials of degree up to $d$, we define a map
\begin{align*}
	\phi \colon \A^{d/q + 1} &\to \A^{d+1}\\
	g(x) &\mapsto g(x)^q
\end{align*}
for $q$ any prime dividing $d$. By identifying a polynomial $g(x) = a_{d/q}x^{d/q} + \cdots + a_0$ with the prime ideal $(a_0 - t_0, \ldots, a_{d/q} - t_{d/q}) \in \Spec \overline{k}[t_0, \ldots, t_{d/q}]$ and studying the coefficients of $g(x)^q$, one can produce the equations for $\phi$.

What we need from this is a lower bound on the codimension of the image of this map, a fact which is proven for the $q=2$ case in \cite[Lemma 3]{poonenstoll_short}, so long as $d > 2$. 

\begin{lemma}
\label{lem:pol_mth_power_codim}
	For a positive integer $d$ and a prime $q \mid d$, let $\phi \colon \A^{d/q+1}_k \to \A^{d+1}_k$ be the $q$-th power map described above. Let $V \subseteq \A^{d+1}_k$ be the scheme theoretic image of $\phi$. So long as $(d,q) \neq (2,2)$, $V$ has codimension at least 2. 
\end{lemma}

\begin{proof}
    Let $I$ be the ideal of $V$ and $\Delta$ the discriminant of a degree $d$ polynomial (viewed as a function on $\A^{d+1}_k$). Since $I$ corresponds to the functions vanishing on $V$, and an element of the image of the $q$-th power map is necessarily not separable, we have $(\Delta) \subseteq I$. Moreover, only in the case $q=2$ and $d=2$ is $(\Delta) = I$. A degree two polynomial is a perfect square if and only if it is nonseparable, but for all $(d,q) \neq (2,2)$, there exist degree $d$ polynomials which are neither separable nor perfect $q$-th powers. As $(\Delta)$ is prime (see e.g.\ \cite[Example 1.4]{GKZ_discriminants}), we have the chain $(0) \subsetneq (\Delta) \subsetneq I$, making the codimension of $V$ at least 2.
\end{proof}

\begin{corollary}\label{cor:pos_prop_NF}
	Suppose $d>2$. Let $k$ be a number field and $\pi \colon X \to \A^{d+1}_k$, as above, be considered as a morphism of $k$-varieties. Suppose $\Psi \subseteq k_\infty^{d+1}$ is a bounded subset such that $\Psi' = \Psi \cap \pi(X(k_\infty))$ has positive measure with $\mu_\infty(\partial \Psi') = 0$. Then $\rho_{m,d,k,\Psi}$, as defined in \eqref{eq:density_number_field}, exists, is nonzero, and is equal to a product of local densities. 
\end{corollary}

\begin{proof}
	By Lemma \ref{lem:properties_of_pi}, we have that $\pi$ is dominant, projective, and has geometrically integral generic fiber. It remains to show that the hypotheses (i), (ii), and (iii) of Theorem \ref{thm:BBL} apply. 
	
	Let $P \in \A^{d+1}_k$ be a codimension one point. By Lemma \ref{lem:m-th_power_lemma}, the fiber $X_P$ is geometrically integral precisely when $P \notin V_q$ for some prime $q \mid m$, where $V_q$ is the scheme-theoretic image of the $q$-th power map described above. By Lemma \ref{lem:pol_mth_power_codim}, each such $V_q$ has codimension at least 2 and thus cannot contain $P$. Put informally, it takes more than just one algebraic relation on the coefficients to force $f(x,z)$ to be an $q$-th power for some $q \mid m$. Thus (i) is satisfied.
		
	For (ii), we have $X(\mathbf{A}_k) \neq \emptyset$ because $X(k) \neq \emptyset$. That is, there exist superelliptic curves $C_f$ with $k$-rational points. For example, one can take
	\[C_f\colon y^m = x^d + xz^{d-1},\]
	which has a $k$-rational point at $[0:0:1]$.
	
	Finally, to see that for real places $v$, $\pi(X(k_v))$ is closed under the action of $\R_{\geq 1}$, simply note that positive $m$-th roots are in $\R$, so if $C_f$ has a $k_v$ point $[x:y:z]$ then $C_{Bf}$ has the $k_v$ point $[x:\sqrt[m]{B} y : z]$. 
\end{proof}

To prove Theorem \ref{thm:pos_prop}, we need only specialize $k=\Q$ and find an appropriate $\Psi \subseteq \pi(X(\R))$ satisfying the desired properties in Corollary \ref{cor:pos_prop_NF}, such that the limit \eqref{eq:bbl_limit} computes the limit in Theorem \ref{thm:pos_prop}.

\begin{proof}[Proof of Theorem \ref{thm:pos_prop}]
	When $k = \Q$ we have $k_\infty = \R$, so we set $\Psi = [-1,1]^{d+1} \cap \pi(X(\R))$, which may be viewed as the set of homogeneous polynomials $f(x,z)$ of degree $d$ with real coefficients of absolute value at most 1 such that $C_f$ has a real point. This subset is clearly bounded, and has positive measure since it contains the set $\left\lbrace \c \in [-1,1]^{d+1} \mid 0 \leq c_0 \leq 1\right\rbrace$, whose measure is half that of the unit cube. To see why, we merely recognize that if $c_0 \geq 0$ then $C_f$ has an $\R$-point $[0:\sqrt[m]{c_0}:1]$. 
		
	To check $\mu_\infty(\partial\Psi) = 0$, with respect to the Euclidean measure $\mu_\infty$ on $[-1,1]^{d+1}$, we use the evaluation map, $\ev_{[x:z]}$ for a point $[x:z] \in \P^1_\R$. This map takes $\ev_{[x:z]}(\c) = f(x,z)$, where $f$ is the degree $d$ binary form in $\R[x,z]$ defined by $\c$. We observe
	\[\Psi = \left( \bigcup_{[x:z] \in \P^1_\R} \ev_{[x:z]}^{-1}\left( (0, \infty) \right) \right) \cup \left \lbrace \c \in [-1,1] \mid f(x,z) = 0 \text{ for some } [x:z] \in \P^1_\R\right \rbrace.\]
	As $\ev_{[x:z]}$ is continuous (in fact it is linear), the union $\cup_{[x:z] \in \P^1_\R} \ev_{[x:z]}^{-1}\left( (0, \infty) \right) $ is open, and hence $\partial \Psi$ is contained in the set
	\[\left \lbrace \c \in [-1,1] \mid f(x,z) \leq 0 \text{ for all } [x:z] \in \P^1_\R \text{ and } f(x,z) = 0 \text{ for some}\right \rbrace.\]
	To be in this set, it is necessary for each such root of $f(x,z) = 0$ to have even multiplicity, and in particular $\c$ is contained in the vanishing set of the discriminant polynomial, which has measure zero.
	
	Thus the limit \eqref{eq:bbl_limit} computes $\frac{\rho_{m,d}}{\rho_{m,d}(\infty)} = \prod_{p \nmid \infty} \rho_{m,d}(p)$ as a product of local densities $\rho_{m,d}(p) = \mu_p\left( \pi(X(\Q_p)) \right)$, completing the proof of Theorem \ref{thm:pos_prop}.
\end{proof}

We conclude this section by making some observations about $\rho_{m,d,k}(\infty)$. If $k$ is totally complex, i.e.\ it has no real places and only complex places, then we have that $\rho(\infty) = 1$. Using the fact that $\C$ is algebraically closed, any choice of $f(x,z)$ with coefficients $k_\infty$ will have a solution $[x:y:z]$ for any choice of $[x:z] \in \P^1_{k_\infty}$.

Whenever $m$ is odd, we have $\rho_{m,d,k}(\infty) = 1$, because real numbers always have an $m$-th root in this case. Geometrically, we observe that $\pi$ is surjective on $k_\infty$-points, $\pi(X(k_\infty)) = \A^{d+1}(k_\infty)$. Conversely, if $m$ is even, $\rho_{m,d,k}(\infty)$ depends only on $d$, and not on $m$. In particular, $\rho_{m,d,\Q}(\infty)$ is equal to the proportion of (real) polynomials $f(x,z)$ which take a positive value somewhere. 

One can easily determine
\[\rho_{m,d,\Q,[-1,1]}(\infty) \geq \frac{3}{4}\]
by observing that $c_d \geq 0$ or $c_0 \geq 0$ is sufficient to ensure the existence of a real point, but at present no analytic approach for computing it is known. Bhargava--Cremona--Fisher \cite[Proposition 3.1]{BCF_genus_one} determined
\[0.873914 \leq \rho_{m,4,\Q}(\infty) \leq 0.874196\]
for even $m$ using a rigorous numerical approach. As observed in \cite{BCF_genus_one}, one could also use a Monte Carlo method to sample the coefficient space to estimate $\rho(\infty)$. See Example \ref{ex:m=4} for such approximations of $\rho_{4,d}(\infty)$ when $4 \leq d \leq 20$.

\section{Lower bounds for the proportion}\label{sec:lower_bounds}

In this section, we give a closed form lower bound for the density $\rho_{m,d}$, albeit one containing infinite products over primes, using only a na\"ive form Hensel's lemma, which allows us to lift roots of equations over $\F_p$ to ones over $\Z_p$. We state what we will need below.

\begin{Theorem}[Hensel's lemma]
\label{thm:hensel}
	Let $F(t) \in \Z_p[t]$ be a polynomial and $\overline{F}(t) \in \F_p[t]$ its reduction modulo $p$. Use $\overline{F}'(t)$ to denote the formal derivative with respect to $t$. If there exists $\overline{t_0} \in \F_p$ such that
	\[\overline{F}(\overline{t_0}) = 0 \quad \text{and} \quad \overline{F}'(\overline{t_0}) \neq 0,\]
	then there exists a lift $t_0 \in \Z_p$ such that $F(t_0) = 0$ and the reduction of $t_0$ modulo $p$ is $\overline{t_0}$.
	
	More generally, if there exists $t_1 \in \Z_p$ such that
	\begin{equation}\label{eq:Hensel_improved}
		v\left(F(t_1)\right) > 2 v\left(F'(t_1)\right),\end{equation}
	where $v$ denotes the $p$-adic valuation, then there exists $t_0 \in \Z_p$ such that $F(t_0) = 0$ and $v(t_0 - t_1) \geq v(F(t_1)) - 2v(F'(t_1))$.
\end{Theorem}

Theorem \ref{thm:hensel} is a standard result in algebraic number theory and in fact holds for any complete discrete valuation ring in lieu of $\Z_p$; see e.g.\ \cite[\S II.2 Proposition 2]{lang} or \cite[Proposition 7.31, Theorems 7.32, 7.33]{milneANT}.

	Let $S$ be a subset of the set of binary degree $d$ forms over $\F_p$. The translation invariance of the Haar measure $\mu_p$ implies that the measure of the set of degree $d$ forms over $\Z_p$ which reduce modulo $p$ to an element of $S$ is equal to the ratio of $\#S$ to the (finite) number of binary degree $d$ forms over $\F_p$,
	\[\mu_p \left( \left\{ f(x,z) \in \Z_p[x,z] \mid f \text{ deg. } d \text{ homog.},\ \overline{f} \in S \right\}\right) = \frac{\#S}{\#\left\{ \overline{f}(x,z) \in \F_p[x,z] \mid \overline{f} \text{ deg. } d \text{ homog.}\right\}}.\]
	In particular, this means that when conditions on the reduction $\overline{f}(x,z)$ guarantee $C_f$ to have a $\Q_p$-point, we can count the number of forms over $\F_p$ satisfying these conditions to give a lower estimate of $\rho_{m,d}(p)$.	
	
	In the case of this section, we use the first statement of Hensel's lemma to give sufficient conditions on $\overline{f}(x,z)$ for $C_f$ to have a $\Q_p$-point, with $F(x,y,z) = \eqref{eq:projectiveeqn}$, with one of $x$, $y$, or $z$ taking the place of the lifting variable $t$. This relatively simple approximation strategy yields lower bounds for $\rho_{m,d}$, as demonstrated for $(m,d) = (3,6)$ and $(5,5)$ in Examples \ref{ex:m=3} and \ref{ex:m=5}. Moreover, they give clues as to the limiting behavior of the density for fixed $m$ as $d \to \infty$; see Corollary \ref{cor:limit_rho} and Example \ref{ex:lim_infs}.

\subsection{Lower bounds for local densities \texorpdfstring{$\rho_{m,d}(p)$}{rho_md(p)}}

Fix a prime exponent $m$ and a degree $d$ divisible by $m$. Recall that the genus of a superelliptic curve $C_f \colon y^m = f(x,z)$ is given by \eqref{eq:genus}. If $f$ is separable of degree $d$, this becomes $g = \frac{(m-1)(d-2)}{2}$.

When $C$ is a smooth curve and $p$ is sufficiently large, the Weil conjectures imply that $C(\F_p) \to \infty$ as $p$ grows large among primes of good reduction. $C$ has bad reduction at only finitely many primes, so this together with Hensel's lemma shows that $C$ is soluble over $\Q_p$ for almost all primes $p$.

To make this effective, we can use the Hasse--Weil bound, which states that
\begin{equation}\label{eq:hasse-weil}
	\left|\#C(\F_p) - (p+1) \right| \leq 2g\sqrt{p}.
\end{equation}
This can be improved further,
\begin{equation}\label{eq:hasse-weil_improvement}
	\left|\#C(\F_p) - (p+1) \right| \leq g\lfloor2\sqrt{p}\rfloor
\end{equation}
see \cite[\S 4.7.2.2]{serre_NXp}. This implies that for the superelliptic curve $C_f$, if $f(x,z)$ is separable over $\F_p$ then whenever
\[p+1-g\lfloor 2\sqrt{p} \rfloor > 0,\]
we have $\#C_f(\F_p) > 0$. Taking $p > 4g^2 - 1 = (m-1)^2(d-2)^2 - 1$ is sufficient in the case that $m \mid d$. This leads us to the following proposition.

\begin{proposition}\label{prop:bound_p=1_large}
	Suppose $m$ is prime and $d$ is divisible by $m$. For all primes $p > (m-1)^2(d-2)^2-1$ we have the lower bound
	\[\rho_{m,d}(p) \geq 1 - p^{\frac{-d(m-1)}{m}}.\]
\end{proposition}

\begin{proof}
	Let $\overline{f}$ denote the reduction of $f$ modulo $p$. By Lemma \ref{lem:m-th_power_lemma}, If $\overline{f} \neq ah^m$ for $a \in \F_p$ and $h \in \F_p[x,z]$, then the curve over $\F_p$ given by $y^m = \overline{f}(x,z)$ is geometrically integral, hence the reduction of $C_f$ modulo $p$ is geometrically integral. A straightforward count shows that there are $p^{\frac{d}{m} + 1}$ homogeneous polynomials $f(x,z) = ah(x,z)^m \in \F_p[x,z]$ of degree $d$, or equivalently, that the fraction of $f$ which are \textit{not} $m$-th powers modulo $p$ is given in the statement.
	
	It remains to prove that if $\overline{f}$ is not an $m$-th power modulo $p$, then its reduction modulo $p$ has a smooth point. If $\overline{C_f}$ is smooth, then the size assumption on $p$ and the bound \eqref{eq:hasse-weil} ensure that $\overline{C_f}(\F_p) \neq \emptyset$, and the smoothness allows us to lift to a point in $C_f(\Q_p)$ via Hensel's lemma (Theorem \ref{thm:hensel}).
	
	If $\overline{C_f}$ is not smooth, then we must normalize. Denote the normalization by $\wt{C_f}$. The genus of $\wt{C_f}$ is $g - \sum_{P \text{ sing}} \frac{1}{2}r_P(r_P-1)$, where $r_P$ denotes the multiplicity at $P$ (see e.g.\ \cite[Ch.\ V,\ Example 3.9.2]{hartshorne}). When $P$ is signular we have $r_P \geq 2$, allowing us to compute
	\begin{align*}
		\#\overline{C_f}^{\sm}(\F_p) &= \#\wt{C_f} - \sum_{P \text{ sing}} r_P \\
		&\geq p+1 - 2\left(g - \sum_{P \text{ sing}} \frac{1}{2}r_P(r_P-1)\right)\sqrt{p} - \sum_{P \text{ sing}} r_P & (\text{by \eqref{eq:hasse-weil}})\\
		&= p+1 - 2g\sqrt{p} + \sum_{P \text{ sing}} \Big(r_P(r_P - 1)\sqrt{p} - r_P\Big) \\
		&>p+1 - 2g\sqrt{p} + \sum_{P \text{ sing}} r_P(r_P - 2) & (\sqrt{p} > 1) \\
		&\geq p+1 - 2g\sqrt{p}.
	\end{align*}
	This quantity is positive by our assumption on the size of $p$.
\end{proof}

The argument of Proposition \ref{prop:bound_p=1_large} can be extended to the case of $m$ composite by considering the prime divisors of $m$ and proceeding via inclusion-exclusion.

\begin{corollary}\label{prop:bound_large_m_comp}
	Fix positive integers $m,d$ such that $m \mid d$ and let $\omega$ denote the number of distinct prime divisors of $m$. Set $g_0$ to be the genus of $C_f$ when $f$ is separable of degree $d$, given by \eqref{eq:genus}. For all primes $p$ such that $p+1-g_0\lfloor 2\sqrt{p}\rfloor >0$, we have the lower bound
	\[\rho_{m,d}(p) \geq 1 - \sum_{q \mid m} p^{\frac{-d(q-1)}{q}} + \sum_{\substack{q_1, q_2 \mid m\\q_1 \neq q_2}} p^{\frac{-d(q_1q_2-1)}{q_1q_2}} - \cdots + (-1)^{\omega} \sum_{\substack{q_1, \ldots, q_\omega \mid m \\ q_i \text{ distinct}}} p^{\frac{-d((q_1 \cdots q_\omega)-1)}{q_1 \cdots q_\omega}}.\]
\end{corollary}

%\begin{proof}
%	In the proof of Proposition \ref{prop:bound_p=1_large}, the hypothesis that $m$ is prime was used to show that $\overline{C_f}$ was reducible for only $p^{\frac{d}{m} + 1}$ choices of $\overline{f}(x,z)$. If $m$ is composite, then $y^m = f(x,z)$ is reducible precisely when $f(x,z)$ is a $q$-th power for some divisor $q \mid \gcd(m,d)$. We count how often this occurs by summing over the prime divisors of $\gcd(m,d)$ and recognizing that we have double-counted forms $f(x,z)$ which are $q_1q_2$-th powers. Continuing via inclusion-exclusion produces the number of degree $d$ binary forms $\overline{f}(x,z)$ for which $\overline{C_f}$ is irreducible. 
%	
%	If $f(x,z)$ is separable, then the genus $g(C_f) = g_0$, and the bound \eqref{eq:hasse-weil_improvement} allows us to lift a smooth $\F_p$-point to a $\Q_p$-point. If $C_f$ is not smooth, the same steps as in the proof of Proposition \ref{prop:bound_p=1_large} reveal that so long as $p+1-g_0\lfloor 2\sqrt{p}\rfloor >0$ holds, such an $\F_p$-point exists.
%\end{proof}

Consider now the case of primes $p$ with $\gcd(p-1,m) = 1$. If $m$ is prime, this is equivalent to $p \not\equiv 1 \pmod{m}$. Here the $m^{\text{th}}$ power map $\F_p \to \F_p$ is an isomorphism of rings, and in particular is surjective. Suppose $[x_0: z_0] \in \P^1_{\F_p}$ such that $f(x_0, z_0) \not\equiv 0 \pmod{p}$. Then we may apply Hensel's lemma to the polynomial (in $y$) $y^m = f(x_0, z_0)$ for any lift of $x_0, z_0$ to $\Z_p$, giving rise to a local solution.

\begin{proposition}\label{prop:bound_pneq1}
	Fix integers $m$ and $d$ divisible by $m$ such that $\gcd(p-1,m) = 1$ and $p \nmid m$.
	\begin{enumerate}[label = (\alph*)]
		\item If $p > \frac{d-1}{2}$ we have the lower bound
		\[\rho_{m,d}(p) \geq 1 - \frac{1}{p^{d+1}}.\]
		\item If $p \leq \frac{d}{2} - 1$ we have the lower bound
		\[\rho_{m,d}(p) \geq 1 - \frac{1}{p^{2(p+1)}}.\]
	\end{enumerate}
\end{proposition}

\begin{proof}
	If $\overline{f}$ takes a nonzero value in $\F_p$ at any $[x_0 : z_0]$, then the above discussion can be used to lift an $\F_p$-solution to $y^m = f(x_0, z_0)$ to $\Q_p$. Also, if $\overline{f}$ has a simple root $(x_0, z_0)$ in $\F_p$ then we can use Hensel's lemma on one of $x$ or $z$ to lift it to a $\Q_p$-solution $y^m = f(x,z)$ with $p \mid y$. The only case not immediately dealt with by Hensel's lemma is if $\overline{f}$ has a double root at every $[x_0 : z_0]$.
	
	If $\overline{f}$ is nonzero modulo $p$ and $p > \frac{d}{2}-1$, then this cannot happen solely for degree reasons. Having a double root at each value is equivalent to the degree $2(p+1)$ polynomial $x^2(x-1)^2 \cdots (x-(p-1))^2z^2$ dividing the degree $d$ polynomial $\overline{f}$. This is not possible for $p > \frac{d}{2}-1$, so the only case not addressed by Hensel's lemma is if $f \equiv 0 \pmod{p}$, giving the lower bound in (a).
	
	To verify (b), we have
	\begin{equation}\label{eq:lots_of_double_roots}
		\overline{f}(x,z) = g(x,z)x^2(x-z)^2 \cdots (x-(p-1)z)^2z^2
	\end{equation}
	for some degree $d - 2(p+1)$ form $g(x,z)$. There are $p^{d - 2(p+1) + 1}$ such choices of $g$, so the proportion of forms $f$ for which $\overline{f}$ is not as in \eqref{eq:lots_of_double_roots} is given in (b).
\end{proof}	

We can also make use of some basic linear algebra to obtain lower estimates for $\rho_{m,d}(p)$, by viewing the evaluation of $\overline{f}$ as a matrix-vector product. This turns out to be useful, especially for primes $p \equiv 1 \pmod{m}$ that are too small for Proposition \ref{prop:bound_p=1_large} to apply.

As usual write $f(x,z) = c_dx^d + c_{d-1}x^{d-1}z + \cdots + c_1xz^{d-1} + c_0z^d$ and denote by $A$ the $(p+1) \times (d+1)$ matrix \[A = \begin{pmatrix}
	1 & 0 & \cdots & 0 & 0 \\
	1 & 1 & \cdots & 1 & 1 \\
	1 & 2 & \cdots & 2^{d-1} & 2^d \\
	&& \vdots && \\
	1 & (p-1) & \cdots & (p-1)^{d-1} & (p-1)^d \\
	0 & 0 & \cdots & 0 & 1
\end{pmatrix}\]
with entries in $\F_p$. We can simultaneously evaluate $f(x_0, z_0)$ at all $[x_0 : z_0]$ in $\P^1_{\F_p}$ by taking the product
\[A\begin{pmatrix}
	c_0 \\ c_1 \\ \vdots \\ c_{d-1} \\ c_d
\end{pmatrix} = \begin{pmatrix}
	\overline{f}(0,1) \\ \overline{f}(1,1) \\ \vdots \\ \overline{f}(p-1,1) \\ \overline{f}(1,0)
\end{pmatrix}.\]
We use this relationship to find lower bounds for the number of $\overline{f}(x,z) \in \F_p[x,z]$ of degree $d$ with at least one $m$-th power value in $\F_p$. 

\begin{lemma}\label{lem:A_is_invertible}
	For the matrix $A$ above, we have
	\[\dim_{\F_p} \ker A = \begin{cases} d-p, & p < d \\ 0, & p \geq d \end{cases}\]
	which is equivalent to
	\[\rank A = \begin{cases} p + 1, & p < d \\ d+1, & p \geq d. \end{cases} \]
\end{lemma}

\begin{proof} Suppose $\c \in \ker A$, where $\c$ is viewed as an element of $\F_p^{d+1}$. Then the corresponding degree $d$ binary form $\overline{f}(x,z)$ has roots at all $[x:z] \in \P^1_{\F_p}$, and equivalently the degree $p+1$ form $x(x-z) \cdots (x-(p-1)z) z$ divides $\overline{f}$. If $p \geq d$ this is a contradiction unless $\c = \mathbf{0}$. If $p < d$, then we write
\[\overline{f}(x,z) = g(x,z)x(x-z) \cdots (x-(p-1)z) z\]
for some degree $d-p-1$ form $g(x,z)$. There are $p^{d-p}$ choices of such $g$, hence in this case the kernel of $A$ has dimension $d-p$.
The rank of $A$ is given by $d+1 - \dim \ker A$.
\end{proof}

\begin{proposition}\label{prop:bound_matrix}
	Fix positive integers $m$ and $d$. For a prime $p \nmid m$ let $\Phi(p) = \#(\F_p^\times)^m / \# \F_p$ be the fraction of elements of $\F_p$ that are nonzero $m$-th powers. Let $r$ denote the rank of $A$. Then we have
	\[\rho(p) \geq 1 - (1 - \Phi(p))^r.\]
	By Lemma \ref{lem:A_is_invertible} we have
	\[\rho(p) \geq \begin{cases} 1 - (1 - \Phi(p))^{p+1},& p < d \\ 1 - (1 - \Phi(p))^{d+1},& p \geq d \end{cases}.\]
\end{proposition}

\begin{proof} 
	We may perform column reduction on $A$ by multiplying on the right by $UP$, where $U$ is an invertible upper triangular matrix and $P$ is a permutation matrix, if appropriate. This gives $A' = AUP$ of rank $r$. Note that the image of $A'$ coincides with that of $A$, so in particular an entry of $A \c$ is in $\left(\F_p^\times\right)^m$ whenever the corresponding entry of $A' (UP)^{-1}\c$ is.
	
	We argue that  the lower bound holds as follows. The first row of $A'$ thus reveals that the proportion of $f$ for which $f(0,1) = c_0$ is in $\left(\F_p^\times\right)^m$ is $\Phi(p)$. For a fixed $c_0$, we let $c_1$ vary and use the second row of $A'$ to see that $f(1,1)$ ranges over all $\F_p$, and the proportion for which $f(1,1) \in \left(\F_p^\times\right)^m$ is again $\Phi(p)$. 
	
	Continuing in this fashion, we see that for any fixed $c_0, \ldots, c_{i-1}$ with $i \leq r-1$, $f(i,1)$ (or $f(1,0)$ if $i=d$) is given by the $(i+1)$-th row of $A'$ containing a pivot (note that this may not coincide with the $(i+1)$-th row). Hence the $(i+1)$-th (pivot) entry is in $\left(\F_p^\times\right)^m$ for $\Phi(p)$ of the possible $c_i$. Thus letting $\c$ vary, the proportion for which at least one of the first $r$ entries of $A \c = A'\c$ is in $\left(\F_p^\times\right)^m$ is 
	\[\Phi(p) + \left( 1 - \Phi(p)\right)\Big( \Phi(p) + \left( 1 - \Phi(p)\right) \Big( \Phi(p) + \left( 1 - \Phi(p)\right) \Big( \cdots \Big) \Big) \Big) = 1 - (1 - \Phi(p) )^r.\] 
	Since $p \nmid m$, one such value $\overline{f}(x,z) \in \left(\F_p^\times \right)^m$ is sufficient to lift via Hensel's lemma to a $\Q_p$-point of $C_f$. This yields the result, and Lemma \ref{lem:A_is_invertible} may be used to determine the rank $r$.
\end{proof}

\begin{remark}\label{rem:crudeness}
	This bound is somewhat crude in the sense that we are ignoring a great deal of possible liftable points, especially when $p \geq d$ is very large. By only using $d+1$ of the $p+1$ points $\overline{f}(x,z)$, we are able to compute an explicit lower bound, but it is quite likely that one of the points we ignore also lifts. This will be sufficient for us to prove results about $\rho_{m,d}$ in the limit as $d \to \infty$, e.g.\ Corollary \ref{cor:limit_rho}. For any fixed $(m,d)$ one can use a brute force computer search to obtain much better estimates; see Examples \ref{ex:m=3} and \ref{ex:m=5}.
\end{remark}

\begin{remark}
\label{rem:no_simple_roots}
	Another way to refine the proof of Proposition \ref{prop:bound_matrix} is to consider points $[x:z]$ where $\overline{f}(x,z) = 0$. These lift whenever the partial derivative $\overline{f}_x(x,z)$ or $\overline{f}_z(x,z)$ is nonzero. By formulating a matrix similar to $A$ and applying the same column reduction used to obtain $A'$, one can give a lower estimate that improves on Proposition \ref{prop:bound_matrix} for $p \leq 7$. However, this method adds considerable effort and (what the authors believe is) unnecessary confusion, while providing marginal improvement for finitely many primes, so we elect to omit this.
\end{remark}

Finally, let us consider the case of primes $p$ dividing $m$. These require some special attention, because the strategy of lifting using Hensel's lemma on $y^m = f(x_0, z_0)$ as in Proposition \ref{prop:bound_pneq1} fails when $p \mid m$, as the partial derivative with respect to $y$ vanishes. 

However, when $m = p$ or more generally $\gcd(p-1,m) = 1$, the $m$-th power map is an isomorphism of $\F_p$. In this case, for any point $[x_0 : z_0]$, there exists $y_0 \in \F_p$ such that $y_0^m \equiv f(x_0, z_0) \pmod{p}$. This means that for each $[x_0 : z_0]$, we need only check whether or not Hensel's lemma applies to lifting via $x$ or $z$, allowing us to obtain a point of $C(\Q_p)$.

\begin{proposition}\label{prop:bound_p|m}
	Fix positive integers $m$ and $d$ divisible by $m$. Suppose $p$ is a prime dividing $m$ with $\gcd(p-1,m) = 1$.  Then we have the lower bound
	\[\rho_{m,d}(p) \geq \begin{cases} 1 - \frac{1}{p^{p-1}}, &d=p \\
									   1 - \frac{1}{p^p}, & d=2p \\
									   1 - \frac{1}{p^{p+1}} & d > 2p.\end{cases}\]
\end{proposition}

\begin{proof}
	We begin by by evaluating $f$ at the point at infinity, $f(1,0) = c_d$. We know that there exists $y_0 \in \F_p$ such that $y_0^m = c_d$, since the $m$-th power map is an isomorphism in this case, so the polynomial $f(1,z) - y_0^m = c_d + c_{d-1}z + \cdots + c_0z^d - y_0^m$ has a solution at $z=0$. The derivative with respect to $z$ is $c_{d-1} + 2c_{d-2}z + \cdots + dc_0z^{d-1}$, which is nonzero at $z=0$ if and only if $c_{d-1} \neq 0 \pmod{p}$. If this is the case, then Hensel's lemma applies and $f$ has a $\Q_p$-point. 
	
	We have  seen that $p \mid c_{d-1}$ is a necessary condition for the point at infinity not to lift, so we now study the affine points $[x : 1]$. Again, for any $x_0$, there exists $y_0$ solving $y_0^m \equiv f(x_0, 1) \pmod{p}$, and this solution lifts via Hensel's lemma if the derivative $f'(x_0) = dc_dx_0^{d-1} + \cdots + c_1 \not\equiv 0 \pmod{p}$. Thus for $\overline{f}$ not to have any liftable points we need $p \mid c_{d-1}$ and $x(x-1) \cdots (x-(p-1)) \mid f'(x)$.
	
	We have $\deg f'(x)= d-2$ since $p \mid d$ forces the $x^{d-1}$ term to vanish. Let $h(x) = \sum_{i=0}^{d-2-p} a_ix^i$ be a polynomial, such that $x(x-1) \cdots (x-(p-1))h(x)$ has degree $d-2$. We count the number of $h$ that produce 
	\[x(x-1) \cdots (x-(p-1))h(x) = f'(x).\]
	If $d=p$, then this is only possible if $h \equiv f' \equiv 0 \pmod{p}$.
	
	Notice that for all integers $0 < k < d/p$, the $x^{pk-1}$ term of $f'(x)$ vanishes modulo $p$. This determines $ d/p - 1$ independent linear conditions on the coefficients $a_i$, so there are at most $p^{d-2-p-d/p + 2} = p^{d-p-d/p}$ choices of $h(x)$. We see this by observing that the leading term of $x(x-1) \cdots (x-(p-1))$ is $x^p$, while the trailing term is $(p-1)!x \equiv -x \pmod{p}$ by Wilson's theorem. We also impose the condition that $c_{d-1} = 0$, i.e.\ the linear condition that $a_{d-p-2} = 0$. These conditions on the $a_i$'s are summarized in the following matrix
	\[\begin{pmatrix}
		* & \cdots& *& -1 & 0 & 0 & \cdots & 0 & 0 & \cdots &0 & 0 & \cdots & 0 & 0  \\
		0 & \cdots& 0& 0  & 1 & * & \cdots & * & -1& \cdots&0 & 0 & \cdots & 0 & 0 \\
		&&&&&&\vdots \\
		0 & \cdots& 0& 0 & 0 & 0 & \cdots & 0 & 0 & \cdots & 1 & * & \cdots & * & -1 \\
		0 & \cdots& 0& 0 & 0 & 0 & \cdots & 0 & 0 & \cdots & 0 & 0 & \cdots & 0 & 1
	\end{pmatrix},\]	
	for which we require $\a = (a_0, \ldots, a_{d-p-2})^T$ to be in the kernel. When this matrix of relations has full rank, each condition is independent.
	 
	It is clear that the first $d/p -1$ rows are independent. The final row is assured to be independent from the others so long as $d/p -1 \geq 2$, i.e.\ $d > 2p$. This is sharp, as illustrated by the case of $p=3$ and $d=6$, when we find the rank of the $2 \times 2$ matrix above is one.	 
	
	Suppose $f_1, f_2$ have $f_1' = f_2' = x(x-1)\cdots(x-(p-1))h(x)$. Then all coefficients are equal except those of $x^{pk}$ for $0 \leq k \leq d/p$. Thus for each $h(x)$ there are $p^{d/p + 1}$ polynomials $f$ for which $f' =  x(x-1)\cdots(x-(p-1))h$. This brings the total number of possible such $f$ to be
	\begin{align*}
		p^{2} & \text{ when } d= p, \\
		p^{p + 1} & \text{ when } d=2p, \\
		p^{d-p} & \text{ when } d > 2p. 
	\end{align*}
	Upon dividing by the total number of forms, $p^{d+1}$, we obtain a lower bound for the proportion of $f$ whose reduction modulo $p$ has at least one Hensel-liftable point, given in the proposition. 
\end{proof}

\subsection{Lower bounds for the adelic density \texorpdfstring{$\rho_{m,d}$}{rho_md}}

Assembling together Propositions \ref{prop:bound_p=1_large}, \ref{prop:bound_pneq1}, \ref{prop:bound_matrix}, and \ref{prop:bound_p|m}, we give a lower bound for $\rho$ which is explicitly computable, at least in principle, when the exponent $m$ is a prime.

\begin{corollary}\label{cor:lower_bound}
	Let $m$ be an odd prime and $d$ an integer divisible by $m$. Define
	\begin{align*}
		L_0(m,d) &= \begin{cases}1 - \frac{1}{m^{m-1}}, & d=m \\ 1 - \frac{1}{m^m}, & d = 2m \\ 1-\frac{1}{m^{m+1}}, & d > 2m,\end{cases} \\		
		L_1^{\wee}(m,d) &= \prod_{\substack{ p \equiv 1 (m)\\ p < d}} \left( 1 - \left(1 - \frac{p-1}{mp}\right)^{p+1} \right) ,\\		
		L_1^{\med}(m,d) &= \prod_{\substack{ p \equiv 1 (m)\\ d < p < (m-1)^2(d-2)^2}} \left( 1 - \left(1 - \frac{p-1}{mp}\right)^{d+1} \right) ,\\
		L_1^{\giant}(m,d) &= 	\prod_{\substack{ p \equiv 1 (m)\\ p \geq (m-1)^2(d-2)^2}} \left( 1- \frac{1}{p^{\frac{d(m-1)}{m}}} \right),\\
		L_{\neq 1}^{\wee}(m,d)	&= \prod_{\substack{p \not\equiv 0,1 (m) \\ p \leq \frac{d}{2}-1}} \left(1 - \frac{1}{p^{2(p+1)}}\right),\\
		L_{\neq 1}^{\giant}(m,d) &= \prod_{\substack{p \not\equiv 0,1 (m) \\ p > \frac{d}{2}-1}} \left( 1 - \frac{1}{p^{d+1}} \right).
	\end{align*}
	Then we have a computable lower bound
	\[\rho_{m,d} \geq L_0(m,d)L_1^{\wee}(m,d)L_1^{\med}(m,d)L_1^{\giant}(m,d)L_{\neq 1}^{\wee}(m,d)L_{\neq 1}^{\giant}(m,d).\]
\end{corollary}

\begin{proof}
	This follows directly from applying Propositions \ref{prop:bound_p=1_large},  \ref{prop:bound_pneq1}, \ref{prop:bound_matrix}, or \ref{prop:bound_p|m} to each local density $\rho(p)$ in 
	\[\rho_{m,d} = \prod_{p \text{ prime}} \rho(p) = \rho(m) \prod_{p \equiv 1 (m)} \rho(p) \prod_{p \not \equiv 0,1 (m)} \rho(p),\]
	splitting the products further into the ranges in the statement as appropriate. When $p \equiv 1 \pmod{p}$, the fraction of nonzero $m$-th power residue classes used in Proposition \ref{prop:bound_matrix} is $\Phi(p) = \frac{p-1}{mp}$.
\end{proof}

\begin{remark}\label{rem:computability}
	Corollary \ref{cor:lower_bound} provides a way to compute an explicit lower bound for $\rho_{m,d}$. Notice that all the products involved are finite, save for $L_1^{\giant}(m,d)$ and $L_{\neq 1}^{\giant}(m,d)$. These products are related to the Riemann zeta function $\zeta(d(m-1)/m)$ and $\zeta(d+1)$, respectively, as the product runs over the appropriate Euler factors, but only for the primes in certain (unions of) conjugacy classes.
	
	This alone ensures that when $d$ is sufficiently large, these products are close to one, because they are part of the tail of $\zeta(s)$ and $\zeta(s) \to 1$ as $s \to \infty$. Explicit values, valid to several decimal places, of similar products of this form were computed in \cite[p.\ 26 -- 34]{mathar}.
\end{remark}

We can make this result less explicit, but somewhat more pleasant by fixing a gonality $m$ and allowing $d \to \infty$, which is tantamount to allowing $g \to \infty$. 

\begin{lemma}\label{lem:limits}
	For a fixed prime $m$, we compute the limits of some of the products defined in Corollary \ref{cor:lower_bound} as $d \to \infty$;
	\begin{align*}
		\lim_{d \to \infty} L_1^{\med}(m,d) &= 1 ,\\
		\lim_{d \to \infty} L_1^{\giant}(m,d) &= 1,\\
		\lim_{d \to \infty} L_{\neq 1}^{\giant}(m,d) &= 1.
	\end{align*}
\end{lemma}

\begin{proof}
	For $L_1^{\giant}(m,d)$ and $ L_{\neq 1}^{\giant}(m,d) $, the conclusion follows from recognizing that as $d \to \infty$, the product consists of a subset of factors of the  convergent product $\zeta(s)$ for $s=\frac{d(m-1)}{m}, d+1$ respectively. Thus the limit is necessarily 1.
	
	For $L_1^{\med}(m,d)$ the conclusion requires more work. First, observe that since the product in $ L_1^{\med}(m,d)$ runs over primes congruent to 1 modulo $m$, we have $p \geq m+1$. This implies
	\[1-\left(1-\frac{p-1}{mp}\right)^{d+1} \geq 1 - \left( \frac{m}{m+1}\right)^{d+1}.\]
	Thus we have
	\[1 \geq L_1^{\med}(m,d) \geq \prod_{\substack{ p \equiv 1 (m)\\ d < p < (m-1)^2(d-2)^2}} \left( 1 - \left( \frac{m}{m+1}\right)^{d+1} \right) \geq \prod_{\substack{ p \equiv 1 (m)\\  p < (m-1)^2(d-2)^2}} \left( 1 - \left( \frac{m}{m+1}\right)^{d+1} \right).\]
	We can compute the limit of the rightmost expression as $d \to \infty$ by taking logarithms. 
	
	Use $\pi_{1,m}(X)$ to denote the number of primes $p \equiv 1 \pmod{m}$  with $p \leq X$, so
	\[\log \prod_{\substack{ p \equiv 1 (m)\\  p < (m-1)^2(d-2)^2}} \left( 1 - \left( \frac{m}{m+1}\right)^{d+1} \right) = \left( \pi_{1,m}((m-1)^2(d-2)^2-1) \right) \log \left( 1 - \left( \frac{m}{m+1}\right)^{d+1} \right).\]
	Using the Taylor series for the logarithm, we have
	\begin{align*}
		\left|\log \left( 1 - \left( \frac{m}{m+1}\right)^{d+1} \right) \right| &= \sum_{j \geq 1} \frac{1}{j}\left(\frac{m}{m+1}\right)^{(d+1)j} \\
		&\leq \sum_{j \geq 1}\left(\frac{m}{m+1}\right)^{(d+1)j} \\
		&= \frac{(\frac{m}{m+1})^{d+1}}{1 - (\frac{m}{m+1})^{d+1}}.
	\end{align*}
	Finally, we observe that upon taking limits, we have
	\[\lim_{d \to \infty} \pi_{1,m}((m-1)^2(d-2)^2-1) \frac{(\frac{m}{m+1})^{d+1}}{1 - (\frac{m}{m+1})^{d+1}} = 0,\]
	since the exponential $(\frac{m}{m+1})^{d+1}$ decays more quickly than $\pi_{1,m}((m-1)^2(d-2)^2-1)$, which is bounded above by a (fixed) polynomial in $d$.
	
	Thus as $d$ grows, $L_1^{\med}(m,d)$ sits in between 1 and another product approaching 1, so we must have $\lim_{d \to \infty} L_1^{\med}(m,d) = 1$.
\end{proof}

This gives us a way to see where these lower bounds are going for a fixed prime $m$ dividing $d$ as $d$ grows, in an entirely computable way. We can restate this as follows.

\begin{corollary}\label{cor:limit_rho}
	Let $m$ be a fixed odd prime. Then
	\[\liminf_{d \to \infty} \rho_{m,d} \geq \left(1 - \frac{1}{m^{m+1}}\right) \prod_{p \equiv 1 (m)} \left( 1 - \left(1-\frac{p-1}{mp}\right)^{p+1} \right) \prod_{p \not\equiv 0,1 (m)} \left(1 - \frac{1}{p^{2(p+1)}}\right).\]
	When $m=2$ we have
	\[\liminf_{d \to \infty} \frac{\rho_{2,d}}{\rho_{2,d}(\infty)} \geq \frac{7}{8} \prod_{p > 2} \left( 1 - \left(1 - \frac{p-1}{2p}\right)^{p+1} \right) \approx 0.66120.\]
\end{corollary}

\begin{proof}
	By Corollary \ref{cor:lower_bound} and Lemma \ref{lem:limits}, we need only take the limits of $L_0(m,d)$, $L_1^{\wee}(m,d)$, and $L_{\neq 1}^{\wee}(m,d)$ as $d \to \infty$, since $L_1^{\med}(m,d), L_1^{\giant}(m,d)$, and $L_{\neq 1}^{\giant}$ can all be made arbitrarily close to 1.
	
	In the case of $m=2$, Corollary \ref{cor:lower_bound} applies to the local densities at the finite places, but $\rho_{2,d}(\infty) \neq 1$. Thus by taking a limit as $d \to \infty$, we obtain a lower bound for $\liminf_{d \to \infty} \frac{\rho_{2,d}}{\rho_{2,d}(\infty)}$. 
\end{proof}

Moreover, the infinite products in Corollary \ref{cor:limit_rho} are straightforward to compute to several decimal places of precision. By recognizing that $p \geq m+1$ in the first product, we have that 
\[\prod_{p \equiv 1 (m)} \left( 1 - \left(1-\frac{p-1}{mp}\right)^{p+1} \right) \geq \prod_{\substack{p \equiv 1 (m) \\ p \leq A}} \left( 1 - \left(1-\frac{p-1}{mp}\right)^{p+1} \right) \prod_{\substack{p \equiv 1 (m) \\ p > A}} \left( 1 - \left(\frac{m}{m+1}\right)^{p+1} \right).\]
The rightmost factor is seen to converge to 1 quickly, e.g.\ by taking logarithms and comparing to a geometric series. In fact, we have
\begin{equation}
\label{eq:tail_A}
	\prod_{\substack{p \equiv 1 (m) \\ p > A}} \left( 1 - \left(\frac{m}{m+1}\right)^{p+1} \right) \geq 1 - \left( \frac{m}{m+1} \right)^{A+1},
\end{equation}
so we may choose $A$ large enough so this factor is as close to 1 as desired. It remains to compute the factors for the finitely many $p \leq A$. Furthermore, this is quite well behaved in $m$, in that for any level of precision, we need only choose $A$ to be a sufficiently large multiple $Cm$, where $C$ does not depend on $m$.

For the other factor, we employ a similar strategy and compare to the Riemann zeta function,
\begin{equation}
\label{eq:tail_B}
	\prod_{p \not\equiv 0,1 (m)} \left(1 - \frac{1}{p^{2(p+1)}}\right) \geq \prod_{\substack{p \not\equiv 0,1 (m) \\ p \leq B}} \left(1 - \frac{1}{p^{2(p+1)}}\right) \prod_{p > B} \left(1 - \frac{1}{p^{2(B+2)}}\right).
\end{equation}
The rightmost factor is the tail of $\zeta(2B + 4)^{-1}$, which converges to 1 rapidly. For example, taking $B=10$ is sufficient to show
\begin{align*}
	\prod_{p \not\equiv 0,1 (m)} \left(1 - \frac{1}{p^{2(p+1)}}\right) &\geq \frac{1}{\zeta(24)} \prod_{\substack{p \not\equiv 0,1 (m) \\ p \leq 10}} \left(1 - \frac{1}{p^{2(p+1)}}\right) \prod_{p \leq 10}\left(1 - \frac{1}{p^{24}}\right)^{-1}.
\end{align*}
When $m > 7$, the first product runs over all primes up to 10, producing the lower bound of at least $0.98422$.

\begin{corollary}\label{cor:liminf_uniform}
	Let $m$ be a fixed odd prime. Then
	\[\liminf_{d \to \infty} \rho_{m,d} \geq 0.83511.\]
\end{corollary}

\begin{proof}
Direct computation shows the result holds for $m=3, 5, 7$; see Example \ref{ex:lim_infs}. Suppose now that $m \geq 11$ is an odd prime. By the above discussion, using \eqref{eq:tail_B} with $B=10$, we have
	\[	\prod_{p \not\equiv 0,1 (m)} \left(1 - \frac{1}{p^{2(p+1)}} \right) \geq 0.98422...\]
	and we note that $1-\frac{1}{m^{m+1}} \geq 1 - \frac{1}{11^{12}}$, so it remains to bound the $\prod_{p \equiv 1 (m)} \left( 1 - \left(1-\frac{p-1}{mp}\right)^{p+1} \right)$ factor from below for an arbitrary prime $m \geq 11$.
	
	Using \eqref{eq:tail_A} with $A = 20m$, we observe that if $p \equiv 1 \pmod{m}$ for $p \leq 20m$, then $p = 2km+1$ for some $1 \leq k \leq 9$. Furthermore, we have $3 \mid 2km + 1$ for three such values of $k$, so we can omit those $k$'s. Since these factors are increasing in $p$, we achieve a lower bound by omitting $k=3, 6, 9$. We have then
	\begin{align*}
		\prod_{\substack{p \equiv 1 (m) \\ p \leq A}} \left( 1 - \left(1-\frac{p-1}{mp}\right)^{p+1} \right) & \geq \prod_{k=1,2,4,5,7,8} \left( 1 - \left(1-\frac{2km}{m(2km+1)}\right)^{2km+2} \right),\\
		\prod_{\substack{p \equiv 1 (m) \\ p > A}} \left( 1 - \left(\frac{m}{m+1}\right)^{p+1} \right) & \geq 1 - \left( \frac{m}{m+1} \right)^{20m+1} \geq  1 - \left( \frac{11}{12} \right)^{221} \approx 0.999999995.
	\end{align*}
	For the latter, we can take $m=11$, as the right hand side is seen to be increasing in $m$. This factor is very nearly 1, and thus will have negligible impact. For the former, we observe the right hand side is decreasing in $m$, so it suffices to take a limit as $m \to \infty$. We have
	\begin{align*}\lim_{m \to \infty} \left( 1 - \left(1-\frac{2km}{m(2km+1)}\right)^{2km+2} \right) &= 1-e^{-2k},\\
	\implies	\prod_{\substack{p \equiv 1 (m) \\ p \leq A}} \left( 1 - \left(1-\frac{p-1}{mp}\right)^{p+1} \right) & \geq \prod_{k=1,2,4,5,7,8} \left(1 - e^{-2k} \right) \approx 0.84850.
	\end{align*}
	Taken together, this verifies the claim for $m \geq 11$. 
\end{proof}

\subsection{Examples}

In this subsection, we compute numerical lower bounds for $\rho_{m,d}$ for selected $(m,d)$ values. The reader primarily interested in such numerical values --- especially for $d$ sufficiciently large relative to a fixed $m$ --- may find these lower bounds sufficient for their purposes without going through the considerable additional effort of computing local densities $\rho(p)$ exactly, as we do later only in the case of $(m,d) = (3,6)$.

\begin{Example}[$m=3, d=6$]\label{ex:m=3}
	In the case where $(m,d) = (3,6)$ the genus of $C_f$ is generically $4$. We can use Corollary \ref{cor:lower_bound} to bound $\rho_{3,6}$ by computing
	\begin{align*}
		L_0(3,6) &= 1 - \frac{1}{3^3} = \frac{26}{27},\\
		L_1^{\wee}(3,6) &= 1, \\
		L_1^{\med}(3,6) &= \prod_{\substack{p \equiv 1 (3) \\ p \leq 61}} \left(1 - \left(1-\frac{p-1}{3p}\right)^7\right) \approx 0.59724,\\
		L_1^{\giant}(3,6) &= \prod_{\substack{p \equiv 1 (3) \\ p > 61}} \left(1-\frac{1}{p^4}\right) = \prod_{\substack{p \equiv 1 (3) \\ p \leq 61}} \left(1-\frac{1}{p^4}\right)^{-1} \prod_{\substack{p \equiv 1 (3)}} \left(1-\frac{1}{p^4}\right) \approx 0.9999998, \\
		L_{\neq 1}^{\wee}(3,6) &= \left(1 - \frac{1}{2^6}\right) = \frac{63}{64}, \\
		L_{\neq 1}^{\giant}(3,6) &= \prod_{\substack{p \equiv 2 (3) \\ p > 2}} \left(1 - \frac{1}{p^7}\right) = \left(1 - \frac{1}{2^7}\right)^{-1} \prod_{p \equiv 2 (3)} \left(1 - \frac{1}{p^7}\right) \approx 0.999987,
	\end{align*}
	to find that
	\[\rho_{3,6} \geq 0.56612.\]
	That is, at least 56\% of curves $y^3 = f(x,z)$ over $\Q$ with $\deg f = 6$ are locally soluble. Note that the infinite products $\prod_{\substack{p \equiv 1 (3)}} \left(1-\frac{1}{p^4}\right)$ and $\prod_{p \equiv 2 (3)} \left(1 - \frac{1}{p^7}\right)$ above are termed Euler modulo products and denoted $\zeta_{3,1}(4)$ and $\zeta_{3,2}(7)$ respectively in \cite[see p.\ 25]{mathar}, and the values used in the above computations were taken from that paper.
	
	There is room for improvement in the lower bound above for $L_1^\med(3,6)$, due to the fact that we expect Proposition \ref{prop:bound_matrix} to be missing many liftable points; see Remark \ref{rem:crudeness}. Since there are only seven primes involved in $L_1^\med(3,6)$, one may use a computer algebra system to enumerate all sextic forms $f(x,z)$ and search for points that lift. The results are tabulated in Table \ref{tab:rho36_computer}.  See Appendix \ref{appx:code} for a more detailed description of this approach.
		
	\begin{table}[h]
	\caption{Lower bounds for $\rho_{3,6}(p)$ for primes $p \equiv 1 \pmod{3}$ up to $p=61$}
	\label{tab:rho36_computer}
	\begin{center}
	{\renewcommand{\arraystretch}{1.25}
	\begin{tabular}{|c | c|}
		\hline $p$ & $\rho_{3,6}(p) \geq$\\ \hline
	7  & $\frac{810658}{823543} \approx 0.98435$\\ 
	13 & $\frac{62655132}{62748517} \approx 0.99851$\\ 
	19 & $\frac{893660256}{893871739} \approx 0.99976$ \\
	31 & $\frac{27512408250}{27512614111} \approx 0.99999$\\
	37 & $\frac{94931742132}{94931877133} \approx 0.999998$\\
	43 & $\frac{271818511748}{271818611107} \approx{0.9999996}$\\
	61 & $\frac{3142742684700}{3142742836021} \approx 0.99999995$ \\ \hline
	\end{tabular}}
	\end{center}
	\end{table}
	We can then take the product of these lower bounds and use them in place of $L_{1}^{\med}(3,6)$ in the calculations above, producing
	\[\rho_{3,6} \geq 0.93134.\]
\end{Example}

\begin{Example}[$m = 5, d=5$]\label{ex:m=5}
	In the case where $(m,d) = (5,5)$, we use Corollary \ref{cor:lower_bound} to obtain
	\begin{align*}
		L_0(5,5) &= 1 - \frac{1}{5^4} = 0.9984,\\
		L_1^{\wee}(5,5) &= 1, \\
		L_1^{\med}(5,5) &= \prod_{\substack{p \equiv 1 (5) \\ p \leq 131}} \left( 1 - \left( 1 - \frac{p-1}{5p} \right)^6 \right) \approx 0.10671, \\
		L_1^{\giant}(5,5) &= \prod_{\substack{p \equiv 1 (5) \\ p > 131}} \left( 1 - \frac{1}{p^4} \right) \approx 0.999999994, \\
		L_{\neq 1}^{\wee}(5,5) &= 1,\\
		L_{\neq 1}^{\giant}(5,5) &= \prod_{p \not\equiv 0,1 (5)} \left(1 - \frac{1}{p^6} \right) \approx 0.98301.
	\end{align*}
	Putting these together with Corollary \ref{cor:lower_bound} yields the following (albeit somewhat disappointing) bound,
	\[\rho_{5,5} \geq 0.10473.\]
	
	As with Example \ref{ex:m=3}, the primes $p \equiv 1 \pmod{5}$ which are too small for Proposition \ref{prop:bound_p=1_large} to apply are the limiting factor in this approach. We  again improve these values with a computer search as described in Appendix \ref{appx:code}, and tabulate them below in Table \ref{tab:rho55_computer}
	\begin{table}[h]
	\caption{Lower bounds for $\rho_{5,5}(p)$ for primes $p \equiv 1 \pmod{5}$ up to $p=131$}
	\label{tab:rho55_computer}
	\begin{center}
	{\renewcommand{\arraystretch}{1.25}
	\begin{tabular}{|c | c|}
		\hline $p$ & $\rho_{3,6}(p) \geq$\\ \hline
	11  & $\frac{1729840}{1771561} \approx 0.97644$\\ 
	31 & $\frac{ 887443392}{887503681} \approx 0.99993$\\ 
	41 & $\frac{4750102896}{4750104241} \approx 0.9999997$ \\
	61 & $\frac{ 51520371384}{51520374361} \approx 0.99999994$\\
	71 & $\frac{128100279888}{128100283921} \approx 0.99999996$\\
	101 & $\frac{1061520142440}{1061520150601} \approx 0.999999992$\\
	131 & $\frac{5053913130552}{5053913144281} \approx 0.999999997$ \\ \hline
	\end{tabular}}
	\end{center}
	\end{table}
	
	This produces the considerable improvement
	\[\rho_{5,5} \geq  0.95826,\]
	demonstrating once again the outsize impact that small primes have on the adelic density, as well as the limitations of Proposition \ref{prop:bound_matrix}.

	We also note that these computations revealed that for all $p > 31$, irreducible curves of the form $y^5 = \overline{f}(x,z)$ where $\deg f = 5$ possess a smooth $\F_p$-point, and hence a lift $y^5 = f(x,z)$ possesses a $\Q_p$-point. This improves on the range of validity for Proposition \ref{prop:bound_p=1_large} in the case when $m=5$ and $d=5$.
\end{Example}

\begin{Example}[$d \to \infty$]\label{ex:lim_infs}
	For a fixed prime $m$, we consider the behavior of the lower bound for $\rho_{m,d}$ given by Corollary \ref{cor:lower_bound} as $d$ grows. For example, taking $m=3$, we compute a lower bound for $\rho_{m,d}$ using Corollary \ref{cor:lower_bound} for several values of $d$. These lower bounds, and the approximate values of $L_1^{\sm}$, $L_1^{\med}$, and $L_{\neq 1}^{\sm}$ are included in the Table \ref{tab:rho_3d_d_grows} below.
	
	\begin{table}[h]
	\caption{Lower bounds for $\rho_{3,d}(p)$ for small $d$ via Corollary \ref{cor:lower_bound}}
	\label{tab:rho_3d_d_grows}
	\begin{center}
	{%\renewcommand{\arraystretch}{1.25}
	\begin{tabular}{|c | c| c | c | c |}
		\hline $d$ & $L_1^{\sm}(3,d) \approx$ & $L_1^{\med}(3,d) \approx$ & $L_{\neq 1}^{\sm}(3,d) \approx$ & $\rho_{3,d} \geq $\\ \hline
	6  &    1    & 0.59723 & 0.98437 & 0.56612\\ 
	9  & 0.93223 & 0.69389 & 0.98437 & 0.62890\\ 
	12 & 0.93223 & 0.81839 & 0.98437 & 0.74174\\ 
	15 & 0.92682 & 0.91381 & 0.98437 & 0.82342\\ 
	18 & 0.92682 & 0.96277 & 0.98437 & 0.86753\\ 
	21 & 0.92635 & 0.98536 & 0.98437 & 0.88744\\ \hline
	\end{tabular}}
	\end{center}
	\end{table}

	Note that $L_1^{\med}(3,d)$ is increasing with $d$, as expected by Lemma \ref{lem:limits}. Using Corollary \ref{cor:limit_rho} and the ensuing discussion, we can quickly compute a decimal approximation of a lower bound for $\liminf_{d \to \infty} \rho_{3,d}$. Taking $A=60$ in \eqref{eq:tail_A} and $B=10$ in \eqref{eq:tail_B}, we find 
	\begin{align*}
		\prod_{p \equiv 1 (3)} \left(1 - \left(1 - \frac{p-1}{3p}\right)^{p+1}\right) & \geq 0.92635,\\
		\prod_{p \equiv 2 (3)} \left(1 - \frac{1}{p^{2(p+1)}}\right) &\geq 0.98437,
	\end{align*}
	so
	\[\liminf_{d \to \infty} \rho_{3,d} \geq 0.90061.\]
	
	Below in Table \ref{tab:liminf_rhomd}, we compute a lower bound for $\liminf_{d \to \infty}$ for other small prime exponents $m$, where we take $A=20m$ and $B=10$ to compute the necessary infinite products, as above.

	\begin{center}
	\begin{table}[h]
	\caption{Lower bounds for $\liminf_{d \to \infty} \rho_{m,d}$ via Corollary \ref{cor:limit_rho} for selected odd primes $m$}
	\label{tab:liminf_rhomd}
	
	{%\renewcommand{\arraystretch}{1.25}
	\begin{tabular}{|c | c|| c | c |}
		\hline $m$ & $\liminf_{d \to \infty} \rho_{m,d}(p) \geq$ & $m$ & $\liminf_{d \to \infty} \rho_{m,d}(p) \geq$\\ \hline
	3 & 0.90061 & \longvdots{1em} & \longvdots{1em} \\
	5 & 0.89457 & 103 & 0.98183\\
	7 & 0.97143 & 107 & 0.98156\\
	11 & 0.87167 & 109 & 0.98418\\
	13 & 0.96823 & 113 & 0.85336\\ 
	17 & 0.98206 & 127 & 0.96662\\
	19 & 0.98418 & \longvdots{1em} & \longvdots{1em} \\
	23 & 0.86036 & 1009 & 0.98417\\
	29 & 0.85968 & 1013 & 0.84918\\
	31 & 0.98418 & 1019 & 0.85128\\
	37 & 0.96546 & 1021 & 0.98417\\ 
	41 & 0.85737 & \longvdots{1em} & \longvdots{1em} \\ \hline
	\end{tabular}}
	\end{table}
	\end{center}
		
	We make a few brief observations about Table \ref{tab:liminf_rhomd}. First is that these methods will not produce lower bounds exceeding $\prod_{p \not\equiv 0,1 (m)} \left(1 - \frac{1}{p^{2(p+1)}}\right)$. In particular, when $m > 7$ above, our lower bounds do not exceed $0.98422$. However some $m$ values come quite close, e.g.\ $m=19, 31, 109, 1009, 1021$, indicating that for such primes, to improve our lower bounds for $\rho_{m,d}$, we need to improve our bounds for $L_{\neq 1}^\wee(m,d)$. 
	
	Note the drops  present at several $m$ values, e.g.\ $m=11,23,29,113,1013,1019$. These can be explained by considering the smallest prime $p \equiv 1 \pmod{m}$. For example, when $m=7$, the smallest prime $p \equiv 1 \pmod{7}$ is $p=29$, while the smallest prime $p \equiv 1 \pmod{11}$ is lower at $p=23$. The small primes have an outsize impact on our lower bounds for the infinite product $\prod_{p \equiv 1 (m)} \left(1 - \left(1 - \frac{p-1}{mp}\right)^{p+1}\right)$. 
\end{Example}

\begin{Example}[$m = 4$]\label{ex:m=4}
	To illustrate the similarities and differences when $m$ is composite, consider $m=4$ and $d$ a multiple of $4$. While we cannot apply Corollaries \ref{cor:lower_bound} or \ref{cor:limit_rho} directly, the methods of this section nevertheless apply to determine $\rho_{4,d}(p)$ for the finite primes $p$.
	
	For the prime $p=2$ and primes $p \equiv 1 \pmod{4}$, Propositions \ref{prop:bound_p=1_large}, \ref{prop:bound_matrix}, and \ref{prop:bound_p|m} apply as usual, with the genus $g_0 = \frac{3(d-2)}{2}.$	 If $p \equiv 3 \pmod{4}$, we observe that $\left(\F_p^\times\right)^4 = \left(\F_p^\times\right)^2$, so $C_f$ has an $\F_p$-point if and only if the hyperelliptic curve $y^2 - f(x,z)$ has one. This allows us to apply the result of Proposition \ref{prop:bound_p=1_large} for all primes $p \equiv 1 \pmod{4}$ such that $p \geq (d-2)^2$, instead of $p \geq 4g_0^2$. Thus whenever $d \geq 8$ we have
	\begin{align*}
		\frac{\rho_{4,d}}{\rho_{4,d}(\infty)} &\geq \frac{7}{8} \prod_{\substack{p \equiv 1 (4) \\ p < d}} \left(1 - \left(1 - \frac{p-1}{4p}\right)^{p+1} \right) \prod_{\substack{p \equiv 1 (4) \\ d \leq p  <  4g_0^2}} \left(1 - \left(1 - \frac{p-1}{4p}\right)^{p+1} \right) \prod_{\substack{p \equiv 1 (4) \\ p \geq 4g_0^2}} \left(1 - \frac{1}{p^{d/2}} \right)\\
		&\times \prod_{\substack{p \equiv 3 (4) \\ p < d}} \left(1 - \left(1 - \frac{p-1}{2p}\right)^{p+1} \right) \prod_{\substack{p \equiv 3 (4) \\ d \leq p  <  (d-2)^2}} \left(1 - \left(1 - \frac{p-1}{2p}\right)^{p+1} \right) \prod_{\substack{p \equiv 3 (4) \\ p \geq (d-2)^2}} \left(1 - \frac{1}{p^{d/2}} \right).
	\end{align*}
	In the case of $d=4$ we replace $7/8$ by $3/4$ to account for the behavior at $p=2$.
	
	For small values of $d$, this produces the values in Table \ref{tab:rho_4d}. In the limit as $d \to \infty$, we find
	\[\liminf_{d \to \infty} \frac{\rho_{4,d}}{\rho_{4,d}(\infty)} \geq 0.49471,\]
	which is slightly worse than the lower bound for $m=2$ case computed in Example \ref{ex:lim_infs}, as one might expect, given that for primes $p \equiv 1 \pmod{4}$ there are fewer quartic residues in $\F_p$.
	
	To deal with the infinite place, we observe that $C_f$ has a real point precisely when $f(x,z)$ takes a positive value. When $m$ is even and $d=4$, \cite{BCF_genus_one} rigorously show $\rho_{m,4}(\infty) \geq 0.873914$. For $4 \leq d \leq 20$, we obtained the approximations for $\rho_{m,d}(\infty)$ using a Monte Carlo approach with $10^7$ samples and recorded them in Table \ref{tab:rho_4d}.
	\begin{center}
	\begin{table}[h]
	\caption{Lower bounds for $\rho_{4,d}/\rho_{4,d}(\infty)$ and approximations of $\rho_{4,d}(\infty)$ for small $d$}
	\label{tab:rho_4d}
	{%\renewcommand{\arraystretch}{1.25}
	\begin{tabular}{|c | c| c |}
		\hline $d$ & $\rho_{4,d}/\rho_{4,d}(\infty) \geq$ & $\rho_{4,d}(\infty) \approx$ \\ \hline
	4  & 0.10125 & 0.8739562 \\ 
	8  & 0.01711 & 0.9183913 \\ 
	12 & 0.03419 & 0.9378118 \\ 
	16 & 0.08218 & 0.9493136 \\ 
	20 & 0.14848 & 0.9568297 \\ \hline
	\end{tabular}}
	\end{table}	
	\end{center}	
\end{Example}

\section{Upper bounds for the proportion}
\label{sec:upper_bounds}

In this section, we consider obstructions to local solubility to give upper bounds on $\rho_{m,d}(p)$, and thus $\rho_{m,d}$. Our primary goal is to show that even in the limit as $d \to \infty$, strictly fewer than 100\% of superelliptic curves are everywhere locally soluble; see Corollary \ref{cor:limsup_uniform}. For this, it is sufficient to study the behavior at $p=2$.

\begin{lemma}\label{lem:upper_bound_p=2}
	Fix an  integer $m \geq 2$ and suppose $d \geq 6$ is divisible by $m$. Then
	\[\rho_{m,d}(2) \leq 1 - \frac{1}{2^9} +  \frac{1}{2^{d+4}}.\]
\end{lemma}

\begin{proof}
	
	Consider the degree $d$ forms $f(x,z)$ which reduce modulo 2 as
	\begin{equation}\label{eq:p=2_lots_double_roots}
		\overline{f}(x,z) = h(x,z)x^2(x+z)^2z^2
	\end{equation}
	for $h(x,z) \in \F_2[x,z]$ a nonzero form of degree $d-6$. There are $2^{d+1-6}-1$ such forms, so the probability that $f(x,z)$ reduces this way is $\frac{1}{2^6} - \frac{1}{2^{d+1}}$.
	
	We now claim that if $f(x,z)$ satisfies \eqref{eq:p=2_lots_double_roots} then the probability of $y^m = f(x,z)$ having no $\Q_2$-solutions is at least $\frac{1}{8}$. Consider first a solution $[x:y:1] \equiv [0:0:1] \pmod{2}$, i.e.\ $x,y \in 2\Z_2$. A necessary condition for such a solution to exist is for $2^2 \mid c_0$, which occurs with probability $\frac{1}{2}$. The same argument shows that $2^2 \mid c_d$ and $2^2 \mid \sum_{i=0}^d c_i$ are necessary for the $\F_2$-solutions $[1:0:1]$ and $[1:0:0]$ to lift to $\Q_2$-solutions, each occuring with probability 1/2. Thus the chance of none of these necessary conditions being met is $1/8$.
	
	Combining this with our earlier calculation, we may compute a lower bound for the probability of $f(x,z)$ for which $C_f$ has no $\Q_2$-points, hence
	\[1 - \rho_{m,d}(2) \geq \frac{1}{8}\left(\frac{1}{2^6} - \frac{1}{2^{d+1}}\right) = \frac{1}{2^9} - \frac{1}{2^{d+4}}.\]
	Rearranging the inequality gives the desired result.
\end{proof}
	
Taking limits as $d \to \infty$, we obtain upper bounds for the limiting behavior of $\rho_{m,d}$, complementing Corollary \ref{cor:liminf_uniform}.

\begin{corollary}\label{cor:limsup_uniform}
	Fix an integer $m \geq 2$. Then 
	\[\limsup_{d \to \infty} \rho_{m,d} \leq 1- \frac{1}{2^9} \approx 0.99804.\]
\end{corollary}

\begin{proof}
	We make the trivial observation that $\rho_{m,d} \leq \rho_{m,d}(2)$ and apply Lemma \ref{lem:upper_bound_p=2}. Since this bound is uniform in $d$ (and in fact in $m$) we can take the limit as $d \to \infty$.
\end{proof}

We conclude this section by turning our attention to primes $p$ more generally. Following the convention established in Appendix \ref{appx:counting_forms}, let $N_{d,0}$ denote the number of binary degree $d$ forms in $\F[x,z]$ up to scaling which have no roots. Let $N_{d,\irr}$ denote the number of irreducible such forms. We have the trivial observation that $N_{d,\irr} \leq N_{d,0}$. In the following lemma only, we use $\mu$ to denote the usual M\"obius function.

\begin{lemma}\label{lem:upper_bound}
	Fix positive integers $m \geq 2$ and $d$ divisible by $m$ such that $(m,d) \neq (2,2)$. Then
	\[\rho_{m,d}(p) \leq 1 - \frac{p-1}{p^{2d+2} - p^{d+1}} N_{d,0} \leq 1 - \frac{p-1}{d\left(p^{2d+2} - p^{d+1}\right)} \sum_{e \mid d} \mu \left(\frac{d}{e} \right) p^e.\]
\end{lemma}

\begin{proof}
	Suppose $\overline{f}(x,z) = 0$, which occurs with probability $\frac{1}{p^{d+1}}$. Then for any point on $C_f$, we must have $p \mid y$. Since $m \geq 2$, this implies $p^2 \mid f(x,z)$, or equivalently that $[x:z]$ is a root of $\frac{1}{p}f(x,z)$. If $\frac{1}{p}f(x,z)$, viewed as a binary form of degree $d$ up to scaling over $\F_p$, has no roots, then $C_f$ is insoluble. This proves the first inequality, since we have shown
	\[1- \rho_{m,d}(p) \geq \frac{1}{p^{d+1}} \left(\frac{N_{d,0}}{p^d+p^{d-1} + \cdots + p + 1}\right) = \frac{1}{p^{d+1}} \left( \frac{(p-1)N_{d,0}}{p^{d+1}-1}\right).\]
	
	For the second inequality, we use the observation that $N_{d,\irr} \leq N_{d,0}$. We also have that, up to scaling, we may assume that an irreducible degree $d$ form is monic. A standard result in elementary number theory shows
	\[N_{d, \irr} = \frac{1}{d} \sum_{e \mid d} \mu \left(\frac{d}{e}\right) p^e,\]
	see e.g.\ \cite[\S 7.2, Corollary 2]{IrelandRosen}, and this is sufficient to yield the second inequality.
\end{proof}

While Lemma \ref{lem:upper_bound} is not suitable for giving nontrivial upper estimates for $\limsup_{d \to \infty} \rho_{m,d}$, it is sufficient to allow us to conclude $\rho_{m,d}(p) < 1$ for all primes $p$. In the case where $(m,d)=(3,6)$ and $p \equiv 2 \pmod{3}$, Lemma \ref{lem:form_counts} states that $N_{6,0} \sim \frac{53}{144}p^6$; thus the bound in Lemma \ref{lem:upper_bound} asymptotically becomes
\[1 - \rho_{3,6}(p) \gg \frac{53}{144}p^{-7} ,\]
which is seen to be sharp by Theorem \ref{thm:exact_rho_3_6}.

\section{An exact formula for the \texorpdfstring{$m=3$}{m=3} and \texorpdfstring{$d=6$}{d=6} case}
\label{sec:m=3_d=6}

The goal of this section is to prove Theorem \ref{thm:exact_rho_3_6}, giving an exact formula for $\rho_{m,d}(p)$ when $m=3$ and $d=6$. As in Section \ref{sec:lower_bounds}, the idea is to study when $C$ has solutions modulo $p^n$ which may be lifted via Hensel's lemma. Here however, we must be more careful in dealing with the case that these points may not be smooth. Some of the strategies we employ resemble those of \cite{BCF_genus_one} for genus one curves, but here there are far more cases to check. We also comment that this strategy to achieve exact density formulas may not generalize well to cases of larger $m$ (see Remark \ref{rem:lifting_double_roots}).

We lay out the idea of the argument below in \S \ref{sec:exact_setup}, detailing the five cases of interest. In \S \ref{sec:exact_geom}, we give geometric arguments to deal with three (easier) cases. \S \ref{sec:exact_intermediate} contains intermediate results which will be used multiple times thereafter, giving a flavor for the type of argument to compute exact local densities. The final two cases are then handled in \S \ref{sec:exact_conjugate} and \S \ref{sec:exact_triple}, culminating in the proof of Theorem \ref{thm:exact_rho_3_6}  in \S \ref{sec:exact_proof_rho}.

\subsection{Setup}\label{sec:exact_setup}

Let $m=3$, $d=6$, and $F$ the defining polynomial of $C_f$ for $f(x,z)$ a binary sextic form,
\begin{equation}\label{eq:sec_homog}
	F = y^3 - f(x,z) = y^3 - c_6x^6 - c_5x^5z - c_4x^4z^2 - c_3x^3z^3 - c_2x^2z^4 - c_1xz^5 - c_0z^6.
\end{equation}
Let $\overline{F} \in \F_p[x,y,z]$ denote the image under the reduction modulo $p$ map. As we will see in Lemma \ref{lem:factorization_counts} there are five possible ways that $\overline{F}$ could factor in $\overline{\F_p}[x,z][y]$:
\begin{enumerate}[label = \arabic*.]
	\item $\overline{F}$ is absolutely irreducible;
	\item $\overline{F}$ has three distinct linear factors over $\F_p$, i.e.\ $\overline{F} = \prod_{i=1}^3 (y - h_i(x,z))$ for binary quadratic forms $h_i(x,z) \in \F_p[x,z]$;
	\item $\overline{F}$ has a linear factor over $\F_p$ and a pair of conjugate factors over $\F_{p^2}$, i.e.\ $\overline{F} = (y - h(x,z))(y - g_1(x,z))(y - g_2(x,z))$ for binary quadratic forms $h(x,z) \in \F_p[x,z]$, and a conjugate pair $g_1, g_2 \in \F_{p^2}[x,z]$;
	\item $\overline{F}$ has three conjugate factors over $\F_{p^3}$, $\overline{F} = \prod_{i=1}^3 (y - g_i(x,z))$;
	\item $\overline{F}$ has a triple root, $\overline{F} = (y - h(x,z))^3$ where $h(x,z) \in \F_p[x,z]$.
\end{enumerate}

\begin{remark}
	More generally, one could study factorization of weighted homogeneous polynomials
	\[y^3 + h(x,z)y^2 + g(x,z)y + f(x,z),\]
	similar to the generalized binary quartics of \cite{BCF_genus_one}, and obtain more diverse factorization types. Since we are interested in superelliptic curves, we will stick to the five factorization types above.
\end{remark}

We define an auxiliary condition $(*)$, which is satisfied by $F$ if and only if $\overline{c_6} \notin \F_p^3$. Equivalently, we have
\begin{equation}
\label{eq:condition_star}
	F \text{ satisfies } (*) \iff y^3 - c_6 \text{ is irreducible in } \F_p[y].
\end{equation}
If $F$ satisfies $(*)$ then $C_f$ has no point at infinity modulo $p$, as $\overline{F}(1,y,0) = 0$ has no solutions. It is analogous to the condition $(*)$ as defined in \cite{BCF_genus_one} and plays a similar role, making appearances in \S \ref{sec:exact_conjugate} and \S \ref{sec:exact_triple}. We denote by $\rho^*(p)$ the local density of curves for which $F$ satisfies condition $(*)$. 

The first thing we need to know is how often each of the factorization types 1 --- 5 appear for $\overline{F}$. This is computed below for all $\overline{F}$ and for those satisfying condition $(*)$ in Lemma \ref{lem:factorization_counts}.

\begin{lemma}\label{lem:factorization_counts}
	Let $\overline{F}$ correspond to the reduction of a superelliptic curve of the form \eqref{eq:sec_homog}. The table below indicates the frequencies for which each factorization type 1 -- 5 appear, as (the reductions of) $c_6, \ldots c_0$ range from $0$ to $p-1$.
\begin{center}
\begin{tabular}{|l || l | l | l |}
	\hline Factorization type & $p = 3$ & $p \equiv 1 \pmod{3}$ & $p \equiv 2 \pmod{3}$ \\ \hline  
	1. Abs. irr. & 2160& $p^3(p^4 - 1)$ & $p^3(p^4 - 1)$\\
	2. 3 distinct linear over $\F_p$ & 0 & $\frac{1}{3}(p^3 - 1)$ & 0\\
	3. Linear + conj. & 0 & 0 & $p^3 - 1$\\
	4. 3 conjugate factors & 0 & $\frac{2}{3}(p^3 -1)$ & 0\\
	5. Triple factor & 27 & 1 & 1\\ \hline
	Total & $3^7$ & $p^7$ & $p^7$ \\ \hline
\end{tabular}
\end{center}
The following table lists the analogous counts of factorization types when condition $(*)$ is satisfied.
\begin{center}
\begin{tabular}{|l || l | l | l |}
	\hline Factorization type & $p = 3$ & $p \equiv 1 \pmod{3}$ & $p \equiv 2 \pmod{3}$ \\ \hline  
	1. Abs. irr. & 0 & $\frac{2}{3}p^2(p-1)(p^4 - 1)$ & 0\\
	2. 3 distinct linear over $\F_p$ & 0 & 0 & 0\\
	3. Linear + conj. & 0 & 0 & 0\\
	4. 3 conjugate factors & 0 & $\frac{2}{3}p^2(p -1)$ & 0\\
	5. Triple factor & 0 & 0 & 0\\ \hline
	Total & 0 & $\frac{2}{3}p^6(p-1)$ & 0 \\ \hline
\end{tabular}
\end{center}
\end{lemma}

\begin{proof} 
    Assume for the moment that $p \neq 3$. If $\overline{F}$ is not absolutely irreducible over $\F_p$, then by Lemma \ref{lem:m-th_power_lemma} we must have $\overline{f}(x,z) = ah(x,z)^3$ for $a \in \F_p$ and $h \in \F_p[x,z]$. We may further assume that $h$ has leading term 1. Thus $\overline{F}$ factors as
    \[\overline{F} = (y - \alpha h(x,z))(y - \omega \alpha h(x,z))(y - \omega^2 \alpha h(x,z))\]
    where $\alpha^3 = a$ and $\omega$ is a primitive third root of unity defined possibly over $\F_{p^2}$. It is now clear that the five listed possibilities are the only possible factorizations of $\overline{F}$, and the only way to find a triple factor is if $h_0 = 0$, so $\overline{F} = y^3$.

	Assume now that $\overline{f} \neq 0$ and consider $p \equiv 1 \pmod{3}$. Then $\omega \in \F_p$, so $\overline{F}$ has either 3 linear factors over $\F_p$ or three conjugate factors over $\F_{p^3}$, depending on the value of $a$, putting us in type 2 or 4. We know that there are $\frac{p-1}{3}$ nonzero cubic residues and $\frac{2(p-1)}{3}$ nonresidues mod $p$, so we obtain the stated counts by recognizing that there are $p^2 + p + 1$ ways to choose $h$ for each nonzero value of $a$.
	
	Assume now that $p \equiv 2 \pmod{3}$. In this case, the cube map is an isomorphism of $\F_p^\times$, but $\omega \notin \F_p$, so we must be in type 3, because $a$ is always in $\F_p^\times$. Hence there are again $(p-1)(p^2 + p + 1) = p^3 - 1$ ways in which $\overline{F}$ is reducible, this time all landing in type 3.
	
	In either case, we have exhausted the possibilities for which $\overline{F}$ is (absolutely) reducible. Since there were $p^3$ of these in total, we are left with $p^7 - p^3$ occurences of type 1. 
	
	Condition $(*)$ can only be satisfied in the case that $p \equiv 1 \pmod{3}$, because otherwise there exists a root to $y^3 - c_6 \equiv 0 \pmod{p}$. Since there are $\frac{2}{3}(p-1)$ nonzero cubic nonresidues that could be the residue of $c_6$, and the other coefficients of $\overline{F}$ may be chosen freely, there are $\frac{2}{3}p^6(p-1)$ total choices of $\overline{F}$ satisfying $(*)$. These clearly cannot come from a triple factor or distinct linear factors over $\F_p$, so the only possibilities are types 1 or 4. Since we need the $x^2$ term of $h$ to be nonzero, there are only $\frac{2}{3}p^2(p-1)$ ways to factor in type 4 while satisfying condition $(*)$, and subtracting this from the total gives the frequency of type 1.
	
	Now we consider the case of $p=3$. Since $\overline{\F_3}$ is characteristic 3, if $\overline{F}$ is reducible, we have $\overline{F} = y^3 - h(x,z)^3 = (y - h(x,z))^3$. Thus there are $3^3 = 27$ ways to choose $h(x,z)$ in this case, all of which give rise to a triple factor. The remaining $3^7 - 3^3 = 2160$ choices of $\overline{F}$ must all be absolutely irreducible over $\F_3$.
\end{proof}

Let $\xi_i$ denote the density of the set of $f(x,z)$ for which $F = y^3 - f(x,z)$ has reduction $\overline{F}$ with factorization type $i$ for $1 \leq i \leq 5$. Similarly, let $\xi_i^*$ denote the density of $f(x,z)$ for which the associated $F$ also satisfies condition $(*)$. The counts of Lemma \ref{lem:factorization_counts} allow us to compute $\xi_i$ and $\xi_i^*$ directly.

\begin{corollary}\label{cor:factorization_probs}
	The densities $\xi_i$ are given by the table below.
	\begin{center}
\begin{tabular}{|l || l | l | l |}
	\hline $\xi_i$ & $p = 3$ & $p \equiv 1 \pmod{3}$ & $p \equiv 2 \pmod{3}$ \\ \hline
	$\xi_1$ & 80/81 & $1 - \frac{1}{p^4}$ & $1 - \frac{1}{p^4}$\\
	$\xi_2$ & 0& $\frac{1}{3p^7}(p^3 - 1)$ & 0\\
	$\xi_3$ & 0& 0 & $\frac{1}{p^7}(p^3 - 1)$\\
	$\xi_4$ &0& $\frac{2}{3p^7}(p^3 - 1)$ & 0\\
	$\xi_5$ & 1/81 & $\frac{1}{p^7}$ & $\frac{1}{p^7}$\\ \hline
\end{tabular}
\end{center}
If $p \equiv 1 \pmod{3}$ and $F$ satisfies condition $(*)$, the nonzero densities $\xi^*_i$ are
\[\xi^*_1 = 1 - \frac{1}{p^4}, \quad \quad \xi^*_4 = \frac{1}{p^4}.\]
\end{corollary}

Returning to local solubility, suppose the reduction of $F$ as given in \eqref{eq:sec_homog} has factorization type $i$ and let $\sigma_i$ denote the density of $F$ possessing a $\Q_p$ solution. Let $\sigma_i^*$ denote the density of $F$ with factorization type $i$ satisfying $(*)$ having a $\Q_p$-solution. Then we have
\begin{align}
	\label{eq:rho_sum} \rho(p) &= \sum_{i=1}^5 \xi_i \sigma_i, \\
	\label{eq:rho_star_sum} \rho^*(p) &= \xi_1^* \sigma_1^* + \xi_4^*\sigma_4^*.
\end{align}
Thus to obtain an exact formula for the local density $\rho(p)$, we may consider separately the local densities $\sigma_i$ with prescribed factorization types.

Along the way, we will also need to know the frequency of various factorization types for binary forms of degrees up to 6. In the proofs of Lemmas \ref{lem:sigma4_double_root}, \ref{lem:taus}, and \ref{lem:thetas} for instance, it is often useful to know the the proportion of such forms having types of roots. Namely, having no roots can be used to see that there are no points, while having a simple root implies the existence of a $\Q_p$-point by Hensel lifting arguments. We will also need to know how often roots have higher multiplicities to determine the exact probabilities.

More precisely, we will count nonzero degree $d$ forms over $\F_p$ in two ways: up to scaling by $\F_p^\times$, and monic forms. 

\begin{definition}[monic]
	A degree $d$ binary form $f(x,z)$ is \textbf{monic} if $f(1, 0) = 1$.
\end{definition}

It is straightforward to see that there are $p^d$ distinct monic degree $d$ forms $f(x,z)$ over $\F_p$, by writing
\[f(x,z) = \sum_{i=0}^d c_ix^i z^{d-i},\]
taking $c_d = 1$ and choosing $c_i$ freely for $0 \leq i < d$. We can then use this to determine that there are $p^d + p^{d-1} + \cdots + p + 1$ distinct degree $d$ forms up to scaling by a nonzero constant. If $c_d \neq 0$, then scaling $f$ by $\frac{1}{c_d}$ yields a monic form. More generally, there is a unique way --- up to scaling by a nonzero constant --- to write 
\[f(x,z) = z^k f_0(x,z),\]
where $f_0$ is a monic degree $d-k$ form for some $0 \leq k \leq d$. Counting the number of such forms gives the desired total.

Since $\F_p[x,z]$ is a unique factorization domain, there is a unique way to factor a degree $d$ form $f(x,z)$, considered up to scaling, into its factors, also considered up to scaling. Similarly, if $f(x,z)$ is monic, it can be factored uniquely into monic factors.

\begin{definition}[factorization type]
	Let $f(x,z)$ be a degree $d$ considered up to scaling (resp.\ monic) binary form with unique factorization
	\[f(x,z) = \prod_{i=1}^r f_i(x,z)^{a_i},\]
	where the irreducible factors $f_i(x,z)$ are also considered up to scaling (resp.\ monic). The \textbf{factorization type} of $f(x,z)$ is the data of the degrees $d_i = \deg f_i$ and multiplicities $a_i$ of the factors. This can be shortened to $(d_1^{a_1} d_2^{a_2} \cdots d_r^{a_r})$. 
	
	For convenience (and uniqueness of the type) we adopt a lexicographic ordering, that $d_i \leq d_{i+1}$ for all $i$ and $a_i \leq a_{i+1}$ if $d_i = d_{i+1}$. When $a_i = 1$, it is omitted. If $f(x,z)$ has a linear factor of multiplicity one, its factorization type takes the form $(1d_2^{a_2} \cdots d_r^{a_r})$. This case will be denoted more compactly as $(1*)$.
\end{definition}

The proportion of degree $d$ forms up to scaling (resp.\ degree $d$ monic forms) possessing certain factorization types, indexed by an integer $i$, is denoted $\eta_{d,i}$ (resp.\ $\eta_{d,i}'$ for monic forms). The following lemma is a direct consequence of Lemma \ref{lem:form_counts}, which can be found in Appendix \ref{appx:counting_forms}, and will be used repeatedly.

\begin{lemma}\label{lem:etas}
	The proportions $\eta_{d,i}$, $\eta_{d,i}'$ are given as follows, for $2 \leq d \leq 6$.
	\begin{center} \small \renewcommand{\arraystretch}{2}
	\begin{tabularx}{\textwidth}{| c | l | c | c |}
		\hline $\boldsymbol{d}$ & \textbf{Fact.\ type} & $\boldsymbol{\eta_{d,i}}$ & $\boldsymbol{\eta_{d,i}'}$ \\ \hline 
		\endhead
		
		\hline \endfoot
		
		\multirow{3}*{$2$} & 0. No roots & $\displaystyle \frac{{\left(p - 1\right)} p}{2  {\left(p^{2} + p + 1\right)}} $ & $\displaystyle  \frac{p - 1}{2  p} $ \\*
		 & 1. $\displaystyle (1*)$ & $\displaystyle  \frac{{\left(p + 1\right)} p}{2  {\left(p^{2} + p + 1\right)}} $ & $\displaystyle  \frac{p - 1}{2  p}$ \\*
		 & 2. $\displaystyle (1^2)$ & $\displaystyle  \frac{p + 1}{p^{2} + p + 1} $ & $\displaystyle \frac{1}{p}$ \\[1ex] \hline 
		 
		 \multirow{3}*{$3$} & 0. No roots & $\displaystyle  \frac{{\left(p - 1\right)} p}{3  {\left(p^{2} + 1\right)}} $ & $\displaystyle  \frac{{\left(p + 1\right)} {\left(p - 1\right)}}{3  p^{2}} $ \\*
		 & 1. $(1*)$ & $\displaystyle  \frac{{\left(2  p + 1\right)} p}{3  {\left(p^{2} + 1\right)}} $ & $\displaystyle  \frac{2  {\left(p + 1\right)} {\left(p - 1\right)}}{3  p^{2}} $ \\*
		 & 2. $(1^3)$ & $\displaystyle  \frac{1}{p^{2} + 1} $ & $\displaystyle  \frac{1}{p^{2}} $ \\[1ex] \hline 
		 
		 \multirow{5}*{$4$} & 0. No roots & $\displaystyle \frac{{\left(3  p^{2} + p + 2\right)} {\left(p - 1\right)} p}{8  {\left(p^{4} + p^{3} + p^{2} + p + 1\right)}} $ & $\displaystyle \frac{{\left(3  p^{2} + p + 2\right)} {\left(p - 1\right)}}{8  p^{3}}$ \\*
		 & 1. $(1*)$ & $\displaystyle \frac{{\left(5  p^{2} + p + 2\right)} {\left(p + 1\right)} p}{8  {\left(p^{4} + p^{3} + p^{2} + p + 1\right)}} $ & $\displaystyle \frac{{\left(5  p^{2} + 3  p + 2\right)} {\left(p - 1\right)}}{8  p^{3}}$ \\*
		 & 2. $(1^2 2)$ & $\displaystyle \frac{{\left(p + 1\right)} {\left(p - 1\right)} p}{2  {\left(p^{4} + p^{3} + p^{2} + p + 1\right)}}$ & $\displaystyle \frac{p - 1}{2  p^{2}}$ \\*
		 & 3. $(1^2 1^2)$ & $\displaystyle \frac{{\left(p + 1\right)} p}{2  {\left(p^{4} + p^{3} + p^{2} + p + 1\right)}}$ & $\displaystyle \frac{p - 1}{2  p^{3}} $ \\*
		 & 4. $(1^4)$ & $ $ $\displaystyle \frac{p + 1}{p^{4} + p^{3} + p^{2} + p + 1}$ & $\displaystyle \frac{1}{p^{3}}$ \\[1ex] \hline 
		 
		  \multirow{6}*{$5$} & 0. No roots & $\displaystyle \frac{{\left(11  p^{2} - 5  p + 6\right)} {\left(p - 1\right)} p}{30  {\left(p^{2} + p + 1\right)} {\left(p^{2} - p + 1\right)}}$ & $\displaystyle \frac{{\left(11  p^{2} - 5  p + 6\right)} {\left(p + 1\right)} {\left(p - 1\right)}}{30  p^{4}}$ \\*
		 & 1. $(1*)$ & $\displaystyle \frac{{\left(19  p^{3} + 6  p^{2} + 4  p + 1\right)} p}{30  {\left(p^{2} + p + 1\right)} {\left(p^{2} - p + 1\right)}}$ & $\displaystyle \frac{{\left(19  p^{3} + 14  p^{2} + 4  p - 6\right)} {\left(p - 1\right)}}{30  p^{4}} $ \\*
		 & 2. $(1^2 3)$ & $\displaystyle \frac{{\left(p + 1\right)} {\left(p - 1\right)} p}{3  {\left(p^{2} + p + 1\right)} {\left(p^{2} - p + 1\right)}}$ & $\displaystyle \frac{{\left(p + 1\right)} {\left(p - 1\right)}}{3  p^{3}}$ \\*
		 & 3. $(1^3 2)$ & $\displaystyle \frac{{\left(p - 1\right)} p}{2  {\left(p^{2} + p + 1\right)} {\left(p^{2} - p + 1\right)}}$ & $\displaystyle \frac{p - 1}{2  p^{3}}$ \\*
		 & 4. $(1^2 1^3)$ & $\displaystyle \frac{p}{{\left(p^{2} + p + 1\right)} {\left(p^{2} - p + 1\right)}}$ & $\displaystyle \frac{p - 1}{p^{4}}$ \\*
		 & 5. $(1^5)$ & $\displaystyle \frac{1}{{\left(p^{2} + p + 1\right)} {\left(p^{2} - p + 1\right)}}$ & $\displaystyle \frac{1}{p^{4}}$ \\[1ex] \hline 
		 
		 \multirow{10}*{$6$} & 0. No roots & $\displaystyle \frac{{\left(53  p^{4} + 26  p^{3} + 19  p^{2} - 2  p + 24\right)} {\left(p - 1\right)} p}{144  {\left(p^{6} + p^{5} + p^{4} + p^{3} + p^{2} + p + 1\right)}} $ & $\displaystyle \frac{{\left(53  p^{4} + 26  p^{3} + 19  p^{2} - 2  p + 24\right)} {\left(p - 1\right)}}{144  p^{5}} $ \\*
		 & 1. $(1*)$ & $\displaystyle \frac{{\left(91  p^{4} + 26  p^{3} + 23  p^{2} + 16  p - 12\right)} {\left(p + 1\right)} p}{144  {\left(p^{6} + p^{5} + p^{4} + p^{3} + p^{2} + p + 1\right)}}$ & $\displaystyle \frac{{\left(91  p^{3} - 27  p^{2} + 50  p - 48\right)} {\left(p + 1\right)} {\left(p - 1\right)}}{144  p^{5}}$ \\*
		 & 2. $(1^2 4)$ or $(1^2 2 2)$ & $\displaystyle \frac{{\left(3  p^{2} + p + 2\right)} {\left(p + 1\right)} {\left(p - 1\right)} p}{8  {\left(p^{6} + p^{5} + p^{4} + p^{3} + p^{2} + p + 1\right)}}$ & $\displaystyle \frac{{\left(3  p^{2} + p + 2\right)} {\left(p - 1\right)}}{8  p^{4}}$ \\*
		 & 3. $(1^2 1^2 2)$ & $\displaystyle \frac{{\left(p + 1\right)} {\left(p - 1\right)} p^{2}}{4  {\left(p^{6} + p^{5} + p^{4} + p^{3} + p^{2} + p + 1\right)}}$ & $\displaystyle \frac{{\left(p - 1\right)}^{2}}{4  p^{4}}$ \\*
		 & 4. $(1^2 1^2 1^2)$ & $\displaystyle \frac{{\left(p + 1\right)} {\left(p - 1\right)} p}{6  {\left(p^{6} + p^{5} + p^{4} + p^{3} + p^{2} + p + 1\right)}}$ & $\displaystyle \frac{{\left(p - 1\right)} {\left(p - 2\right)}}{6  p^{5}} $ \\*
		 & 5. $(1^3 3)$ & $\displaystyle \frac{{\left(p + 1\right)}^{2} {\left(p - 1\right)} p}{3  {\left(p^{6} + p^{5} + p^{4} + p^{3} + p^{2} + p + 1\right)}}$ & $\displaystyle \frac{{\left(p + 1\right)} {\left(p - 1\right)}}{3  p^{4}}$ \\*
		 & 6. $(1^31^3)$ & $\displaystyle \frac{{\left(p + 1\right)} p}{2  {\left(p^{6} + p^{5} + p^{4} + p^{3} + p^{2} + p + 1\right)}} $ & $\displaystyle \frac{p - 1}{2  p^{5}}$ \\*
		 & 7. $(1^4 2)$ & $\displaystyle \frac{{\left(p + 1\right)} {\left(p - 1\right)} p}{2  {\left(p^{6} + p^{5} + p^{4} + p^{3} + p^{2} + p + 1\right)}}$ & $\displaystyle \frac{p - 1}{2  p^{4}}$ \\*
		 & 8. $(1^2 1^4)$ & $\displaystyle \frac{{\left(p + 1\right)} p}{p^{6} + p^{5} + p^{4} + p^{3} + p^{2} + p + 1}$ & $\displaystyle \frac{p - 1}{p^{5}}$ \\*
		 & 9. $(1^6)$ & $\displaystyle \frac{p + 1}{p^{6} + p^{5} + p^{4} + p^{3} + p^{2} + p + 1}$ & $\displaystyle \frac{1}{p^{5}}$ \\[1ex] \hline
	\end{tabularx}
	\end{center}
\end{lemma}

\subsection{Geometric arguments: computing \texorpdfstring{$\sigma_1$, $\sigma_2$, and $\sigma_3$}{sigma1, sigma2, and sigma3}}\label{sec:exact_geom}

When $\overline{F}$ is absolutely irreducible, we can leverage the proof of Proposition \ref{prop:bound_p=1_large} to see that $\sigma_1 = 1$ when $p$ is sufficiently large.

\begin{proposition}\label{prop:sigma_1}
	Suppose $C_f$ is given by \eqref{eq:sec_homog} with $\overline{F}$ absolutely irreducible over $\F_p$. Then the reduction $\overline{C_f}$ has a smooth $\F_p$-point whenever 
	\begin{enumerate}[label = (\roman*)]
		\item $p \equiv 1 \pmod{3}$ and $p > 43$, or
		\item $p \equiv 2 \pmod{3}$ and $p > 2$.
	\end{enumerate}
	In particular, whenever (i) or (ii) above is satisfied we have $\sigma_1 = \sigma_1^* = 1$.
\end{proposition}

\begin{proof}
	The curve $\overline{C_f}$ is cut out by $\overline{F}$. Note that the reduction of $f(x,z)$ is not a cube if $\overline{F}$ is absolutely irreducible. Taking $m=3$ in the proof of Proposition \ref{prop:bound_p=1_large}, we have that since $\overline{f}$ is not a cube, the reduction $\overline{C_f}$ is guaranteed to have a smooth $\F_p$-point when $p  > (m-1)^2(d-2)^2-1 = 63$. Furthermore, the improved Hasse--Weil bound \eqref{eq:hasse-weil_improvement} shows that $\overline{C_f}$ is guaranteed to have a smooth point when $p=61$. Hence in case (i) we have $\sigma_1 = 1$ and $\sigma_1^* = 1$ as well, since this argument is independent of the $(*)$ condition.
	
	When $p \equiv 2 \pmod{3}$, taking $d=6$ in the proof of Proposition \ref{prop:bound_pneq1} shows that for $p > 2$ such that $3 \nmid p$, there always exists an $\F_p$ solution of $\overline{F} = 0$ which is liftable by Hensel's lemma. Hence in case (ii) we have $\sigma_1 = 1$ as well.
\end{proof} 

When $\overline{F}$ has factorization types 2 or 3, which occur when $\overline{F}$ has at least one factor of the form $y - g(x,z)$ where $g$ is a (nonzero) binary quadratic form, we study the $\F_p$-points on the irreducible components of $\overline{C_f}$. See \cite[Proposition 2.6]{BCF_genus_one} for the analogous case for genus one curves.

\begin{proposition}\label{prop:sigma_2=1}
	Suppose $C_f$ is given by \eqref{eq:sec_homog} and $\overline{F}$ has factorization type 2 modulo $p$. Then $C_f$ has a $\Q_p$ point, or equivalently, $\sigma_2 = 1$. 
\end{proposition}

\begin{proof} 
	For $p = 3$ and $p \equiv 2 \pmod{3}$ the result is vacuously true, since factorization type 2 does not occur. Thus we may assume $p \equiv 1 \pmod{3}$, and in particular $p > 3$.
	
	We have $\overline{F} = \prod_{i=1}^3 (y - h_i(x,z))$ for distinct binary quadratics $h_i$, so $\overline{C_f}$ is the union of $\overline{C_i} \colon y = h_i(x,z)$. Each $\overline{C_i}$ has $p+1$ points in $\F_p$, and each distinct $(i,j)$ pair has at most two intersection points. To see this, suppose $(\alpha, \beta, \gamma)$ is on $C_1$ and $C_2$, i.e.\ 
	\[h_1(\alpha, \gamma) = \beta = h_2(\alpha, \gamma).\]
	Thus $(\gamma x - \alpha z)$ is a linear factor of the binary quadratic $h_1 - h_2$, and there are at most two such factors.	
	
	With this in hand, we have a maximum of 6 total intersection points. This gives at least $3(p+1) - 2\cdot 6 = 3(p-3)$ smooth points, so whenever $p > 3$ we can lift one of these $\F_p$ points to a $\Q_p$ point on $C$, giving $\sigma_2 = 1$.
\end{proof}

\begin{proposition}\label{prop:sigma_3=1}
	Suppose $C_f$ is given by \eqref{eq:sec_homog} and $\overline{F}$ has factorization type 3 modulo $p$. Then $C_f$ has a $\Q_p$ point, or equivalently, $\sigma_3 = 1$. 
\end{proposition}

\begin{proof}
	As in the proof of Proposition \ref{prop:sigma_2=1}, the statement is vacuously true for $p=3$ and all primes $p \equiv 1 \pmod{3}$, so we assume $p \equiv 2 \pmod{3}$.
	
	Recall that by the proof of Lemma \ref{lem:factorization_counts} we have
	\[\overline{F} = (y - h(x,z))(y^2 + h(x,z)y + h(x,z)^2),\]
	where $h(x,z)$ is a (nonzero) binary quadratic form over $\F_p$. Let 
	\[\overline{C_1} \colon y = h(x,z), \hspace{2cm} \overline{C_2} \colon y^2 + h(x,z)y + h(x,z)^2 = 0.\]
	In fact, $\overline{C_2}$ is geometrically reducible, factoring over $\F_{p^2}$, where the third root of unity is defined. This means that after a finite extension, we are in the same situation as in the proof of Proposition \ref{prop:sigma_2=1}, and each of the components has at most two intersection points. 
	
	In particular $\overline{C_1}$ has $p+1$ $\F_p$-points and at most 4 $\F_{p^2}$-points of intersection with $\overline{C_2}$, forming (at most) two conjugate pairs. Thus at most two $\F_p$-points of $\overline{C_1}$ intersect with $\overline{C_2}$, so $p-1 > 0$ of the points on $C_1$ are smooth solutions to $\overline{F} = 0$, which we can lift to a $\Q_p$-solution.
\end{proof}

\begin{remark}
	These arguments can be generalized to larger $m$, if one is willing to exclude small primes $p$. Suppose $m$ is prime and $d=km$ for $k \geq 1$. Consider a superelliptic curve of the form $C_f \colon y^m = f(x,z)$, where the reduction of $f$ modulo $p$ is nonzero. We have already seen, through the proofs of Propositions \ref{prop:bound_p=1_large} and \ref{prop:bound_pneq1}, that when $\overline{f}$ is not a perfect $m$-th power, there exists a $\Q_p$ point on $C_f$ for sufficiently large primes $p$.
	
	If $\overline{f}$ is a perfect $m$-th power, then over $\overline{\F_p}$, the reduction $\overline{C_f}$ breaks up into $m$ components
	\[C_i \colon y = h_i(x,z),\]
	where $h_i$ is a (nonzero) binary form of degree $k = \frac{d}{m}$. The argument in the proof of Proposition \ref{prop:sigma_2=1} shows that each $C_i$ intersects with another $C_j$ in at most $k$ points. 
	
	Suppose at least one of the components, say $C_1$, is defined over $\F_p$, thus excluding factorization type 4 when $(m,d) = (3,6)$, which will require more care. Since $C_1$ has $p+1$ $\F_p$-points, we have that $p+1 - k(m-1)$ of these points lift by Hensel's lemma. Thus when $p > km - (k+1)$, the curve $C_f$ is guaranteed a $\Q_p$-point.
	
	The caveat is that relatively few superelliptic curves have these factorization types. Only $p^{k+1}-1$ out of the $p^d$ choices for $\overline{f}(x,z)$ modulo $p$ have that $\overline{f}$ is a nonzero $m$-th power. If $p \equiv 1 \pmod{m}$, then further $\frac{p-1}{mp}$ of those will have a factor defined over $\F_p$. As $d \to \infty$ for fixed $m$, these fail to make up a positive proportion. 
\end{remark}

Considering only factorization types 1 --- 3, we are essentially no better off than in Section \ref{sec:lower_bounds}, where using $m=3$ and $d=6$ we obtain lower bounds
\[\rho_{3,6}(p) \geq \begin{cases} 1 - \frac{1}{p^4} & p \equiv 1 \pmod{3}, \ p > 43\\ 1 - \frac{1}{p^7} & p \equiv 2 \pmod{3}, \ p > 2, \end{cases}\]
via Propositions \ref{prop:bound_p=1_large} and \ref{prop:bound_pneq1}. Thus even just to obtain the improved asymptotics of Theorem \ref{thm:exact_rho_3_6}, it is necessary to consider factorization types 4 and 5. 

\subsection{Intermediate results}\label{sec:exact_intermediate}

The following intermediate results will be used repeatedly, as will the strategies of their proofs. Throughout, $f(x,z)$ denotes a binary sextic form with coefficients $c_i \in \Z_p$ and $C_f$ the equation cut out by $y^3 - f(x,z)$. 

\begin{lemma}\label{lem:alpha_beta}
	Let $p \equiv 1 \pmod{3}$. Suppose $c_4, c_5, c_6 \in p\Z_p$ and $c_3 \in \Z_p$ are fixed, such that the reduction $\overline{c_3}$ is neither a cubic residue nor zero, i.e.\ $\overline{c_3} \notin \F_p^3$. Let $\beta$ be the probability that $C_f$ has a $\Q_p$-point of the form $[x:y:1]$, as $c_0, c_1, c_2$ range over $\Z_p$. Let $\alpha$ denote the same probability, but with $c_0, c_1, c_2 \in p\Z_p$. We have
	\begin{align*}
		\alpha &= \frac{(p^3 + p + 1)}{p^4 + p^3 + p^2 + p + 1}, \\
		\beta &= \frac{p(p^3 + p^2 + 1)}{p^4 + p^3 + p^2 + p + 1}.
	\end{align*}
\end{lemma}

\begin{proof}
The reduction $\overline{C_f}$ is isomorphic to the curve cut out by
\begin{equation}\label{eq:cubic}
   y^3 \equiv \overline{f}(x,z) \equiv c_3x^3 + c_2x^2z + c_1xz^2 + c_0z^3 \pmod{p}.
\end{equation}
We look for a smooth solution $[x:y:1]$ to \eqref{eq:cubic}, which lifts to a $\Q_p$-point on $C_f$ by Hensel's lemma. Note that there are no solutions of the form $[x:y:0]$ because $\overline{c_3} \notin \F_p^3$.
    
The normalization of $\overline{C_f}$ has geometric genus at most 1. Applying the Hasse--Weil bound method to \eqref{eq:cubic} as in the proof of Proposition \ref{prop:bound_p=1_large}, we see that whenever $f(x,z)$ doesn't have a triple root (equivalently $\overline{f} \neq \overline{c_3}(x-\alpha z)^3$ for some $\alpha \in \F_p$), we have
\[\#\overline{C_f}^{\sm}(\F_p) \geq p+1 - 2\sqrt{p} > 0.\]
The rightmost inequality follows from the fact that $p \geq 7$ since $p \equiv 1 \pmod{3}$. Thus we have found our desired $\Q_p$-point whenever $\overline{f} \neq \overline{c_3}(x-\alpha z)^3$. 

The proportion of cubics over $\F_p$ with fixed leading coefficient $c_3$ having a triple root is equal to $\eta_{3,2}' = \frac{1}{p^2}$ (see Lemma \ref{lem:etas}). In this case, after a change of variables, we  may assume \eqref{eq:cubic} is of the form $y \equiv c_3x^3 \pmod{p}$, i.e.\ $c_0, c_1, c_2 \in p\Z_p$. The probability of $C_f$ having a $\Q_p$-point in this case is precisely $\alpha$. Thus we have
\begin{equation}\label{eq:beta}
		\beta = 1 - \frac{1}{p^2} + \frac{\alpha}{p^2}.
	\end{equation}
	
We now assume $c_0, c_1, c_2 \in p\Z_p$, and we are looking to lift solutions of $y^3 \equiv c_3x^3 \pmod{p}$, whose only $\F_p$-solution is the (singular) point $(0,0)$ by our assumption on $c_3$. Thus $p \mid x,y$ is necessary, so looking modulo $p^2$, we see that $p^2 \mid c_0$ is also a necessary condition, which occurs with probability $\frac{1}{p}$ as $c_0$ runs through $p\Z_p$. 

Assuming $p^2 \mid c_0$, we perform a change of variables by replacing $x$ and $y$ by $px$ and $py$. Dividing by $p^2$, the equation for $\overline{C_f}$ is now
\[0 \equiv \frac{c_1}{p} x + \frac{c_0}{p^2} \pmod{p}.\]
If $v_p(c_1) = 1$, then this is merely a linear equation, so for any choice of $y \in \Z_p$, we can find an $\F_p$ solution that lifts to a $\Q_p$-one. This occurs with probability $1-\frac{1}{p}$, so with probability $\frac{1}{p}$ we have $p^2 \mid c_1$.

Assuming $p^2 \mid c_1$, we again find it is necessary for $p^3 \mid c_0$. Dividing the equation for $\overline{C_f}$ by $p^3$ instead of $p^2$ as above, we obtain
\[y^3 \equiv c_3x^3 + \frac{c_2}{p}x^2 + \frac{c_1}{p^2}x + \frac{c_0}{p^3} \pmod{p},\]
which puts us back in the case of $\beta$, where $c_0, c_1, c_2 \in \Z_p$. That is, we have shown
\begin{align}\label{eq:alpha}
 	\nonumber \alpha &= \frac{1}{p} \left( 1 - \frac{1}{p} + \frac{\beta}{p^2}\right)\\
 	&= \frac{1}{p} + \frac{1}{p^2} + \frac{\beta}{p^3}.
\end{align}
Combining \eqref{eq:beta} and \eqref{eq:alpha} and solving simultaneously, we obtain the claimed values.
\end{proof}

The strategy employed in the proof of Lemma \ref{lem:alpha_beta} --- making successive reductions until we reach a case we know and solving a system of equations to determine desired probabilities --- will be used repeatedly. For results with longer proofs, it is convenient to organize the argument with a table. We illustrate this below with the computation for $\alpha$.

\begin{center}
	\begin{tabular}{l l l | l l l l l l l}
	& & & $v(c_6)$ & $v(c_5)$ & $v(c_4)$ & $v(c_3)$ & $v(c_2)$ & $v(c_1)$ & $v(c_0)$ \\
	
	% line 1
	$\alpha =$ & $\alpha_a =$ & $\frac{1}{p}\alpha_b$
		& $\geq 1$ & $\geq 1$ & $\geq 1$ & $=0$ & $\geq 1$ & $\geq 1$ & $\geq 1$\\
		
	% line 2
	 & $\alpha_b =$ & $1 - \frac{1}{p} + \frac{1}{p}\alpha_c$
		& $\geq 1$ & $\geq 1$ & $\geq 1$ & $=0$ & $\geq 1$ & $\geq 1$ & $\geq 2$\\
		
	% line 3
	 & $\alpha_c =$ & $\frac{1}{p}\alpha_d$
		& $\geq 1$ & $\geq 1$ & $\geq 1$ & $=0$ & $\geq 1$ & $\geq 2$ & $\geq 2$\\
		
	% line 4
	 & $\alpha_d =$ & $\alpha_e$
		& $\geq 1$ & $\geq 1$ & $\geq 1$ & $=0$ & $\geq 1$ & $\geq 2$ & $\geq 3$\\
		
	% line 5
	 & $\alpha_e =$ & $\beta$
		& $\geq 4$ & $\geq 3$ & $\geq 2$ & $=0$ & $\geq 0$ & $\geq 0$ & $\geq 0$		
	\end{tabular}
	\end{center}
	
	The first step in the table above was recognizing that $p^2 \mid c_0$ is necessary for a solution to lift. In the second step, we assume $v(c_0) \geq 2$ and compute the probability of a liftable solution when $v(c_1) = 1$, moving to the next line if $p^2 \mid c_1$, and so on. One sees that combining the steps in the table, we achieve the same formula for $\alpha$ in terms of $\beta$ as \eqref{eq:alpha}.
	
	Repeating the argument of the proof of Lemma \ref{lem:alpha_beta}, we obtain similar results when $c_3$ is fixed of valuation 1 or 2. We will use all of these later as well. Note that the probabilities depend on the conjugacy class of $p$ modulo 3, owing to the fact that the probability of a nonzero element of $\F_p$ being a cubic residue differs in each case.

\begin{lemma}\label{lem:master}
	Suppose $c_3, c_4, c_5, c_6 \in \Z_p$ are fixed with $p$-adic valuation given below. Let $\alpha', \beta', \alpha'', \beta''$ denote the probabilities that $C_f$ has a $\Q_p$-point of the form $[x:y:1]$ as $c_0, c_1, c_2$ vary over $\Z_p$ with the specified valuation(s).
	{ \small \begin{align}
		\nonumber &v(c_6)& &v(c_5)& &v(c_4)& &v(c_3)& &v(c_2)& &v(c_1)& &v(c_0)& \\ 
		\label{eq:alpha_prime} &\geq 2& &\geq 2& &\geq 2& &=1& &\geq 1& &\geq 1& &\geq 1& &\displaystyle \alpha' = \begin{cases}  \frac{2  p^{3} + 2  p^{2} + 3}{3  {\left(p^{3} + p^{2} + p + 1\right)}} = \frac{5}{8}, & p = 3 \\[1ex] 
		\frac{2  p^{4} + 2  p^{3} + 3  p + 1}{3  {\left(p^{4} + p^{3} + p^{2} + p + 1\right)}},& p \equiv 1 \pmod{3} \\[1ex]
		\frac{2  p^{4} + 2  p^{3} + 3  p + 3}{3  {\left(p^{4} + p^{3} + p^{2} + p + 1\right)}},& p \equiv 2 \pmod{3} \end{cases}  \\
		\label{eq:beta_prime} &\geq 2& &\geq 2& &\geq 2& &=1& &\geq 1& &\geq 0& &\geq 0& &\displaystyle \beta' = \begin{cases}  \frac{3  p^{3} + 2  p^{2} + 2  p}{3  {\left(p^{3} + p^{2} + p + 1\right)}} = \frac{7}{8}, & p = 3 \\[1ex] 
		\frac{3  p^{4} + p^{3} + 2  p^{2} + 2  p}{3  {\left(p^{4} + p^{3} + p^{2} + p + 1\right)}},& p \equiv 1 \pmod{3} \\[1ex]
		\frac{3  p^{4} + 3  p^{3} + 2  p^{2} + 2  p}{3  {\left(p^{4} + p^{3} + p^{2} + p + 1\right)}},& p \equiv 2 \pmod{3} \end{cases} \\
		\label{eq:alpha_prime_prime} &\geq 3& &\geq 3& &\geq 3& &=2& &\geq 2& &\geq 2& &\geq 2& &\displaystyle \alpha'' = \begin{cases} \frac{2  p^{3} + 2  p^{2} + 3  p}{3  {\left(p^{3} + p^{2} + p + 1\right)}} = \frac{27}{40} , & p = 3 \\[1ex] 
		\frac{2  p^{4} + 2  p^{3} + p^{2} + 3  p}{3  {\left(p^{4} + p^{3} + p^{2} + p + 1\right)}},& p \equiv 1 \pmod{3} \\[1ex]
		\frac{2  p^{4} + 2  p^{3} + 3  p^{2} + 3  p}{3  {\left(p^{4} + p^{3} + p^{2} + p + 1\right)}},& p \equiv 2 \pmod{3} \end{cases} \\
		\label{eq:beta_prime_prime} &\geq 3& &\geq 3& &\geq 3& &=2& &\geq 3& &\geq 3& &\geq 3& &\displaystyle \beta'' = \begin{cases} \frac{3  p^{3} + 2  p + 2}{3  {\left(p^{3} + p^{2} + p + 1\right)}} = \frac{89}{120} , & p = 3 \\[1ex] 
		\frac{p^{4} + 3  p^{3} + 2  p + 2}{3  {\left(p^{4} + p^{3} + p^{2} + p + 1\right)}},& p \equiv 1 \pmod{3} \\[1ex]
		\frac{3  p^{4} + 3  p^{3} + 2  p + 2}{3  {\left(p^{4} + p^{3} + p^{2} + p + 1\right)}},& p \equiv 2 \pmod{3} \end{cases}
	\end{align}}
\end{lemma}

\begin{proof} Starting with $\alpha'$, in order for there to be a liftable $\F_p$-point of the form $[x:y:1]$, we need $y \equiv 0 \pmod{p}$, and a root of the polynomial $\frac{1}{p}(c_3x^3 + c_2x^2 + c_1x + c_0) \pmod{p}$. There is an $\eta_{3,1}' = \frac{2}{3}(1 - \frac{1}{p^2})$ chance of the existence of a simple root, an $\eta_{3,0}' = \frac{1}{3}(1 - \frac{1}{p^2})$ chance of no roots, and an $\eta_{3,2}' = \frac{1}{p^2}$ chance of a triple root (see Lemma \ref{lem:etas}). If we have a triple root, we may assume it is at $x \equiv 0 \pmod{p}$ after a change of variables, allowing us to assume $v(c_2), v(c_1), v(c_0) \geq 2$.

Considering the resulting polynomial mod $p^3$, we have that $x$ is a root if and only if $p^3 \mid c_0$, which occurs with probability $\frac{1}{p}$. Changing variables by replacing $x,y$ by $px, py$ and dividing by $p^3$ gives us that $c_6, c_5, c_4 \in p^2 \Z_p$ (though their valuations may increase), and $v(c_0), v(c_1) \geq 0$, while $v(c_2) \geq 1$ and $c_3$ remains unchanged. This is precisely the case of $\beta'$, giving us
\[\alpha' = \eta_{3,1}' + \frac{\beta'}{p^3} = \frac{2}{3}\left(1 - \frac{1}{p^2}\right) + \frac{\beta'}{p^3} ,\]
regardless of the choice of prime $p$.

Now we compute $\beta'$. If $v(c_1) = 0$ then we can always find a smooth solution to $c_1x + c_0 \equiv 0 \pmod{p}$, as a linear polynomial always has a simple root. If $v(c_1) \geq 1$ then mod $p$ we have $F(x,y,1) \equiv y^3 - c_0$. Suppose $v(c_0) = 0$, for if not we are in the case of $\alpha'$. If $p \equiv 1 \pmod{3}$ this has a liftable solution with probability $1/3$, as $1/3$ of the residue classes in $\F_p^\times$ are cubic residues. If $p \equiv 2 \pmod{3}$, this probability is $1$, since every nonzero residue is a cube. If $p = 3$, then the change of variables $y \mapsto y + a$, where $a \equiv c_0 \pmod{p}$ gives the new equation
\[F(x,y,1) = y^3 - c_3x^3 - c_2x^2 - c_1x - c_0 + a^3,\]
and since $c_0 \equiv a^3 \pmod{p}$ we have that $p \mid c_0$ after a change of variables. Hence, we have
\[\beta' = \begin{cases} (1 - \frac{1}{p}) + \frac{1}{p}(\frac{1}{3}(1 - \frac{1}{p}) + \frac{1}{p}\alpha'), & p \equiv 1 \pmod{3}\\
	(1 - \frac{1}{p}) + \frac{1}{p}((1 - \frac{1}{p}) + \frac{1}{p}\alpha'), & p \equiv 2 \pmod{3} \\ 
	(1 - \frac{1}{p}) + \frac{1}{p}\alpha', & p = 3. \end{cases}\]
	Solving these equations for $\alpha'$ and $\beta'$ gives the values in the tables.

To compute $\alpha''$, we proceed as in the calculation for $\alpha'$. We can compute the probability that $\frac{1}{p^2}(c_3x^3 + c_2x^2 + c_1x + c_0)$ has a simple root or triple root, and notice that if there is a triple root, it can be moved to $x \equiv 0 \pmod{p}$, putting us in the case of $\beta''$. Thus
\[\alpha'' = \eta_{3,1}' + \frac{\beta''}{p^2} = \frac{2}{3}\left(1 - \frac{1}{p^2}\right) + \frac{\beta''}{p^2} .\]

For $\beta''$, we immediately make the change of variables $x\mapsto px, y \mapsto py$ and divide by $p^3$. This doesn't change $c_3$, but puts the valuations of $c_2, c_1, c_2$ at \textit{at least} 2, 1, and 0 respectively. If $p \neq 3$ and $v(c_0) = 0$, then we can compute the probability that $y^3 = c_0$ has a solution depending on the residue class of $p$. If $p \mid c_0$ then we look mod $p^2$, where our polynomial becomes linear. If $v(c_1) \geq 2$ then we must have $p^2 \mid c_0$ in order to have a solution, putting us back in the case of $\alpha''$.

If $p = 3$ then we can take the same approach as for $\beta'$. After changing variables in $y$, we may assume that $p \mid c_0$, and then follow the same argument as $p \neq 3$. Thus the values of $\beta''$ in terms of $\alpha''$ become
\[ \beta'' = \begin{cases} \frac{1}{3}(1 - \frac{1}{p}) + \frac{1}{p}((1 - \frac{1}{p}) + \frac{1}{p^2}\alpha''), &p \equiv 1 \pmod{3}\\
	1 - \frac{1}{p} + \frac{1}{p}((1 - \frac{1}{p}) + \frac{1}{p^2}\alpha''), & p \equiv 2 \pmod{3}\\
	1 - \frac{1}{p} + \frac{1}{p^2}\alpha'', &p = 3.\end{cases}\]
	Solving the equations for $\alpha''$ and $\beta''$ gives the values stated in the table.
\end{proof}

\begin{remark}\label{rem:independence}
	The probabilities in Lemmas \ref{lem:alpha_beta} and \ref{lem:master} are independent of $c_4, c_5, c_6$, even though they may be changed in the second parts of the proofs. This is key, and also noted in the proof of \cite[Lemma 2.8]{BCF_genus_one}. 
\end{remark}

We conclude this subsection with another result which will be used repeatedly in what follows. While it is independent from the later results, we will make reference in the proof to quantities which will be defined and determined in \S \ref{sec:exact_triple}, in an effort to keep the paper compact.

\begin{lemma}\label{lem:mu}
	Fix a prime $p > 31$, or $p>2$ satisfying $p \equiv 2 \pmod{3}$. Suppose $c_4, c_5, c_6$ are fixed with $c_5, c_6 \in p^3 \Z_p$ (resp.\ $p^2\Z_p$) and $v_p(c_4) = 2$ (resp.\ $v_p(c_4) = 1$). Let $\mu$ (resp.\  $\mu'$) denote the proportion of $f$ for which $C_f$ has a $\Q_p$-point of the form $[x:y:1]$ as $c_0, c_1, c_2, c_3$ vary over $p^2 \Z_p$ (resp.\ $p\Z_p$). Then
	\begin{align}
		\label{eq:mu} \mu &= \begin{cases}\frac{45  p^{11} - 6  p^{10} + 5  p^{9} - 30  p^{8} + 69  p^{7} - 29  p^{6} - 39  p^{5} + 81  p^{4} - 120  p^{3} + 60  p^{2} + 108  p - 72}{72  p^{11}}, & p \equiv 1 \pmod{3} \\[2ex]
		\frac{{\left(5  p^{10} - 3  p^{9} + 2  p^{7} + 3  p^{6} - 16  p^{5} + 25  p^{4} - 16  p^{3} - 8  p^{2} + 20  p - 8\right)} {\left(p + 1\right)}}{8  p^{11}},& p \equiv 2 \pmod{3} \end{cases}\\
		\nonumber \mu' &= \eqref{eq:mu_prime}.
	\end{align}
\end{lemma}

\begin{proof}
	Consider first $\mu$, so let $v(c_4) = 2$, $c_5, c_6 \in p^3 \Z_p$, and $c_0, c_1, c_2, c_3$ vary in $p^2 \Z_p$. A necessary condition for $[x:y:1]$ to satisfy $F = 0$ is $p \mid y$, hence $\frac{1}{p^2} \left(c_3x^3 + c_2x^2 + c_1x + c_0 \right) \equiv 0 \pmod{p}$. Thus the probabilility depends on how this quartic factors modulo $p$, the proportions of which are given by $\eta_{4,i}'$ from Lemma \ref{lem:etas}.
	
	If $\frac{1}{p^2}\left(c_3x^3 + c_2x^2 + c_1x + c_0 \right)$ has no roots modulo $p$, then there can necessarily be no liftable solution. If it has a root of multiplicity 1, then we can lift it to a $\Q_p$-solution. If it has a double root, then after composing with an automorphism of $\P^1$, we may assume the root occurs at $[x:z] = [0:1]$, i.e.\ we have $v(c_0), v(c_1) \geq 3$ while $v(c_2) = 2$. Thus we are precisely in the case of $\theta_2$, to be defined in \S \ref{sec:exact_triple} and computed in Lemma \ref{lem:thetas}. Similarly, if the quartic has two double roots, the probability of at least one lifting is given by $\theta_3$. If it has a quadruple root, the probability of it lifting is given by $\theta_7$, which is valid for all primes $p > 31$ or $p > 2$ if $p \equiv 2 \pmod{3}$. Thus we have
	\begin{equation}\label{eq:mu_sum}
		\mu = \eta_{4,1}' + \eta_{4,2}' \theta_2 + \eta_{4,3}' \theta_3 + \eta_{4,4}' \theta_7
	\end{equation}
	which gives the value stated in \eqref{eq:mu}.
	
	For $\mu'$, a similar analysis shows that
	\begin{equation}\label{eq:mu_prime_sum}
		\mu' = \eta_{4,1}' + \eta_{4,2}' \tau_2 + \eta_{4,3}' \tau_3 + \eta_{4,4}' \tau_7
	\end{equation}
	where again the $\tau_i$ are to be defined in \S \ref{sec:exact_triple} and computed in Lemma \ref{lem:taus}. This gives the value of $\mu'$ which given in \eqref{eq:mu_prime}. We comment that while none of the $\tau_i$ use $\mu'$ as defined here, the value of $\mu$ is used in the computation of $\tau_7$, so this rational function for $\mu'$ is valid for all primes $p > 31$ or $p > 2$ if $p \equiv 2 \pmod{3}$.
\end{proof}

\subsection{Three conjugate factors: computing \texorpdfstring{$\sigma_4$}{sigma4}}\label{sec:exact_conjugate}

By Lemma \ref{lem:factorization_counts} and Corollary \ref{cor:factorization_probs}, this type only occurs when $p \equiv 1 \pmod{3}$, so we assume this for the remainder of \S\ref{sec:exact_conjugate}. If $\overline{F}$ has three distinct conjugate factors, none of which are defined over $\F_p$, then the proof of Lemma \ref{lem:factorization_counts} shows that
\[\overline{F}(x,y,z) = y^3 - ah_0(x,z)^3,\]
where $a \notin (\F_p^\times)^3$ is nonzero and $h_0(x,z)$ is a binary quadratic form defined over $\F_p$ up to scaling. Note also that $(*)$ is satisfied whenever $h$ above is monic.

It is thus clear that $\overline{F}$ has no $\F_p$-solutions for which $h_0(x,z) \neq 0$. However, if $h_0(x_0,z_0) = 0$, the point $(x_0,0,z_0)$ is not a smooth point of $\overline{F}$. After considering the possible factorization types of $h_0(x,z)$ we obtain the following value for $\sigma_4$ when $p$ is sufficiently large.

\begin{proposition}\label{prop:sigma_4}
	Suppose $C_f$ is given by \eqref{eq:sec_homog} and $\overline{F}$ has factorization type 4 modulo $p$ for a prime $p > 43$. Then the proportions of $f$ and $f$ satisfying $(*)$ for which $C_f(\Q_p) \neq \emptyset$ are
	\begin{align*}
		\sigma_4 &= \eqref{eq:sigma_4}\\
		\sigma_4^* &= \eqref{eq:sigma_4_star}.
	\end{align*}
\end{proposition}

The proof is given in \S \ref{sec:proof_sigma4}, after studying the factorization types of the binary quadratic form $h_0(x,z)$ individually.

\subsubsection{\texorpdfstring{$h_0(x,z)$}{h0(x,z)} has no roots in \texorpdfstring{$\F_p$}{Fp}.} If this is the case, then there are no $\Q_p$-points, because there are no $\F_p$ solutions to $h_0(x,z) = 0$, which is necessary by the above argument.

The probability of $h_0(x,z)$ having no roots is $\eta_{2,0} = \frac{p(p-1)^2}{2(p^3 - 1)}$. If $F$ satisfies $(*)$, then we may as well assume $h_0(x,z)$ is monic, and the probability of it having no roots is $\eta_{2,0}' = \frac{p-1}{2p}$.

\subsubsection{\texorpdfstring{$h_0(x,z)$}{h0(x,z)} has distinct roots in \texorpdfstring{$\F_p$}{Fp}.} The probability of this occurring is $\eta_{2,1} = \frac{p(p^2-1)}{2(p^3-1)}$, and $\eta_{2,1}' = \frac{p-1}{2p}$ in the case of condition $(*)$. After composing with an automorphism of $\P^1$, we may assume the roots of $h_0$ are located at $[x:z] = [1:0]$ and $[0:1]$, so we have $h_0(x,z) = xz$. Thus we have $\overline{F}(x,y,z) = y^3 - ax^3z^3$, where $a \in \F_p^\times - (\F_p^\times)^3$. We now need to compute the probabilities that $[0:0:1]$ and $[1:0:0]$ lift to $\Q_p$-points. This is analogous to \cite[\S 2.3.2]{BCF_genus_one}, and follows from Lemma \ref{lem:alpha_beta} applied to each root.

\begin{corollary}\label{cor:sigma_4_h_splits}
	Suppose $\overline{F}$ has factorization type 4, with $h_0(x,z)$ having distinct linear factors mod $p$. Then the probability that $F$ has a $\Q_p$-solution is given by 
	\[1 - (1 - \alpha)^2 = \frac{(p^3 + p + 1)(2p^4 + p^3 + 2p^2 + p + 1)}{(p^4 + p^3 + p^2 + p + 1)^2}\]
	where $\alpha$ is as defined in Lemma \ref{lem:alpha_beta}.
\end{corollary}

\begin{proof}
	Suppose $F(x,y,z) \equiv y^3 - ah_0(x,z)^3 \pmod{p}$, where $a \notin \F_p^3$, such that $h_0(x,z)$ has distinct linear factors mod $p$. After a change of coordinates, we may assume that $h_0(x,z) = xz$, giving us
	\[\overline{F}(x,y,z) = y^3 - ax^3z^3.\]
	The only $\F_p$-points of $\overline{F}\equiv 0$ are at $[1:0:0]$ and $[0:0:1]$, both of which are singular.
	
	Since $c_0, c_1, c_2, c_4, c_5, c_6 \in p\Z_p$ and $c_3 \equiv a \notin \F_p^3$, for the point $[0:0:1]$ we are in the case of Lemma \ref{lem:alpha_beta}. The probability that $[0:0:1]$ lifts to a $\Q_p$-point is thus $\alpha$. Note that by Remark \ref{rem:independence}, this only depends on the choices of $c_0, c_1, c_2$. Similarly, but with the roles of $x$ and $z$ reversed, the probability that $[1:0:0]$ lifts is also $\alpha$, and only depends on $c_4, c_5, c_6$. Thus the two probabilities are independent, allowing us to compute the probability that at least one of the points lifts by $1 - (1-\alpha)^2$.
\end{proof}

\subsubsection{\texorpdfstring{$h_0(x,z)$}{h0(x,z)} has double root} We now need to carry out the analysis for when $h_0(x,z)$ has a double root. This occurs with probability $\eta_{2,2} = \frac{p^2 - 1}{p^3 - 1}$ and in the case of $(*)$, $\eta_{2,2}' = \frac{1}{p}$. This case requires more work, which we organize into the following lemma.

\begin{lemma}\label{lem:sigma4_double_root}
	Assume $p > 3$ and suppose $\overline{F}$ has factorization type 4, with $h_0(x,z)$ having a double root modulo $p$. Then the probability that $F$ has a $\Q_p$-solution is given by $\lambda$, where
	\begin{align}\label{eq:lambda_rho_star_relation}
		\lambda = \frac{1}{p^{15}}\rho^*(p) &+ (p-1)\Big(72p^{25} + 72p^{23} + 72p^{22} + 24p^{21} - 24p^{20} + 36p^{19} - 84p^{18} + 72p^{17} - 27p^{16}  \\
		\nonumber & + 18p^{15} - 13p^{14} - 12p^{13} - 36p^{12} + 25p^{11} - 55p^{10} + 12p^9 - 115p^8 + 105p^7 - 178p^6   \\
		\nonumber & + 73p^5 - 35p^4 + 67p^3 - 93p^2 + 36p - 12\Big) \Big/ \Big(72p^{22}(p^5 - 1)\Big)
	\end{align}
	
\end{lemma}

\begin{proof}
	After a change of variables, we may assume $h_0(x,z) = x^2$, so we have $\overline{F}(x,y,z) = y^3 - ax^6$, where $a \notin \F_p^3$. That is, we have $c_6 \in \Z_p$ such that $c_6 \equiv a \pmod{p}$ and $c_0, ..., c_5 \in p\Z_p$. The only $\F_p$-point of $\overline{F}$ is the singular point at $[0:0:1]$ since $a \notin \F_p^3$. Thus $\lambda$ is the probability that this point lifts.
	
	The table below lists the valuations of the coefficients of $f$. For each line, we compute the probability that the singular point lifts or not, and then move on to the next line. The probability for moving to the next line will always  be $\frac{1}{p}$. This will give rise to a linear relation between $\lambda$ and $\rho^*$, the probability that $F$ has a $\Q_p$-point when its leading coefficient is not a cube mod $p$.
	
	\begin{center}
	\begin{tabular}{l l l | l l l l l l l}
	& & & $c_6$ & $c_5$ & $c_4$ & $c_3$ & $c_2$ & $c_1$ & $c_0$ \\
	
	% line 1
	$\lambda =$ & $\lambda_a =$ & $\frac{1}{p}\lambda_b$
		& $=0$ & $\geq 1$ & $\geq 1$ & $\geq 1$ & $\geq 1$ & $\geq 1$ & $\geq 1$\\
	
	% line 2
	& $\lambda_b =$ & $\left(1 - \frac{1}{p}\right) + \frac{1}{p}\lambda_c$
		& $=0$ & $\geq 1$ & $\geq 1$ & $\geq 1$ & $\geq 1$ & $\geq 1$ & $\geq 2$\\
		
	% line 3
	 & $\lambda_c =$ & $\frac{1}{p}\lambda_d$
		& $=0$ & $\geq 1$ & $\geq 1$ & $\geq 1$ & $\geq 1$ & $\geq 2$ & $\geq 2$\\
		
	% line 4
	 & $\lambda_d =$ & $\left(1 - \frac{1}{p}\right) + \frac{1}{p}\lambda_e$
		& $=3$ & $\geq 3$ & $\geq 2$ & $\geq 1$ & $\geq 0$ & $\geq 0$ & $\geq 0$\\
		
	% line 5
	 & $\lambda_e =$ & $\left(1 - \frac{1}{p}\right) + \frac{1}{p}\lambda_f$
		& $=3$ & $\geq 3$ & $\geq 2$ & $\geq 1$ & $\geq 1$ & $\geq 0$ & $\geq 0$\\
		
	% line 6
	 & $\lambda_f =$ & $\Phi(p) + \frac{1}{p}\lambda_g$
		& $=3$ & $\geq 3$ & $\geq 2$ & $\geq 1$ & $\geq 1$ & $\geq 1$ & $\geq 0$\\

	% line 7
	 & $\lambda_g =$ & $\left(1 - \frac{1}{p}\right)\alpha' + \frac{1}{p}\lambda_h$
		& $=3$ & $\geq 3$ & $\geq 2$ & $\geq 1$ & $\geq 1$ & $\geq 1$ & $\geq 1$\\
		
	% line 8
	 & $\lambda_h =$ & $ \left(1 - \frac{1}{p} \right) \left( \frac{p-1}{2p} + \frac{1}{p^2} \right) + \frac{1}{p}\lambda_i$
		& $=3$ & $\geq 3$ & $\geq 2$ & $\geq 2$ & $\geq 1$ & $\geq 1$ & $\geq 1$\\
		
	% line 9
	 & $\lambda_i =$ & $\left(1 - \frac{1}{p}\right) + \frac{1}{p}\lambda_j$
		& $=3$ & $\geq 3$ & $\geq 2$ & $\geq 2$ & $\geq 2$ & $\geq 1$ & $\geq 1$\\
		
	% line 10
	 & $\lambda_j =$ & $\frac{1}{p}\lambda_k$
		& $=3$ & $\geq 3$ & $\geq 2$ & $\geq 2$ & $\geq 2$ & $\geq 2$ & $\geq 1$\\
		
	% line 11
	 & $\lambda_k =$ & $\left(1 - \frac{1}{p}\right)\mu + \frac{1}{p}\lambda_\ell$
		& $=3$ & $\geq 3$ & $\geq 2$ & $\geq 2$ & $\geq 2$ & $\geq 2$ & $\geq 2$\\
		
	% line 12
	 & $\lambda_\ell =$ & $\left(1 - \frac{1}{p}\right)\alpha'' + \frac{1}{p}\lambda_m$
		& $=3$ & $\geq 3$ & $\geq 3$ & $\geq 2$ & $\geq 2$ & $\geq 2$ & $\geq 2$\\	
		
	% line 13
	 & $\lambda_m=$ & $\left(1 - \frac{1}{p} \right) \left( \frac{p-1}{2p} + \frac{(p+2)(2p^2 - 3p + 3)}{6p^3} \right) + \frac{1}{p}\lambda_n$
		& $=3$ & $\geq 3$ & $\geq 3$ & $\geq 3$ & $\geq 2$ & $\geq 2$ & $\geq 2$\\	
	
	% line 14
	 & $\lambda_n =$ & $\left(1 - \frac{1}{p}\right) + \frac{1}{p}\lambda_o$
		& $=3$ & $\geq 3$ & $\geq 3$ & $\geq 3$ & $\geq 3$ & $\geq 2$ & $\geq 2$\\	
		
	% line 15
	 & $\lambda_o =$ & $ \frac{1}{p}\lambda_p$
		& $=3$ & $\geq 3$ & $\geq 3$ & $\geq 3$ & $\geq 3$ & $\geq 3$ & $\geq 2$\\	
		
	% line 16
	 & $\lambda_p =$ & $ \rho^*$
		& $=0$ & $\geq 0$ & $\geq 0$ & $\geq 0$ & $\geq 0$ & $\geq 0$ & $\geq 0$\\	
		
	\end{tabular}
	\end{center}
	
	Putting together the steps above gives \eqref{eq:lambda_rho_star_relation}. Each step is justified below.
	\begin{enumerate}[label = (\alph*)]
		\item The only possible point reduces to $[0:0:1]$, so $p \mid x,y$. Reducing $F(x,y,z)$ mod $p^2$ reveals that $v(c_0) \geq 2$ is necessary. 
		
		\item If $v(c_1) = 1$, which occurs with probability $1-\frac{1}{p}$ then we can fix $y \in p\Z_p$ and look for solutions to $F(x,y,1)$ as a function of $x$. While it is clear $p \mid x$ is necessary, we have $v(F'(x,y,1)) = 1$, so we need to look for solutions mod $p^3$. Looking modulo $p^3$, we have the linear equation $c_1x + c_0 \equiv 0 \pmod{p^3}$, which we can solve, finding something that lifts. If $v(c_1) = 2$ then we move to the next line.
		
		\item Again we have $p \mid x,y$, so we reduce mod $p^3$ and find that it is necessary to have $p^3 \mid c_0$. This occurs with probability $\frac{1}{p}$. Before moving to the next line we replace both $x$ and $y$ with $px$ and $py$, then divide by $p^3$.
		
		\item With probability $1- \frac{1}{p}$ we have $v(c_2) = 0$, in which case $\overline{C_f}$ is isomorphic to
		\[y^3 \equiv c_2x^2z + c_1xz^2 + c_0z^3 \pmod{p}\]
		and its normalization has geometric genus at most 1. Since $\overline{f} \neq ah^3$ for $a \in \F_p$ and $h \in \F_p[x,z]$, we apply the Hasse--Weil bound (see the proof of Proposition \ref{prop:bound_p=1_large}) to find
		\[\#\overline{C_f}^{\sm}(\F_p) \geq p + 1 - 2\sqrt{p} > 1,\]
		where the rightmost inequality holds for all primes $p > 4$. $\overline{C_f}$ has only the point $[1:0:0]$ above infinity, so there must exist some smooth $\F_p$-point $[x:y:1]$ which lifts to a $\Q_p$-point of $C_f$. If $v(c_2) \geq 1$ we move to the next line.

		\item With probability $1 - \frac{1}{p}$ we have $v(c_1) = 0$ and the reduced equation is $\overline{F}(x,y,1) = y^3 - c_1x - c_0$, which is linear in $x$ and thus has a solution with $y \in p\Z_p$. With probability $\frac{1}{p}$ we have $v(c_1) \geq 1$ and move to the next line.
		
		\item The reduced equation is now $\overline{F} = y^3 - c_0 $. The probability that $c_0$ is a nonzero cubic residue is $\Phi(p) = \frac{1}{3}\left(1 - \frac{1}{p}\right)$ (see Proposition \ref{prop:bound_matrix} for the definition). If $p \nmid c_0$ is not a cubic residue, then no point lifts. With probability $\frac{1}{p}$ we have $v(c_0)= 1$ and we move to the next line.
		
		\item With probability $1 - \frac{1}{p}$ we have $v(c_3) = 1$ and we are in the case of $\alpha'$. See \eqref{eq:alpha_prime} from Lemma \ref{lem:master}. With probability $\frac{1}{p}$ we move to the next line.
		
		\item	With probability $1 - \frac{1}{p}$ we have $v(c_2) = 1$. It is clear that for any solution, we must have $p \mid y$ and $\frac{1}{p}(c_2x^2 + c_1x + c_0) \equiv 0 \pmod{p}$. The quadratic $\frac{1}{p}(c_2x^2 + c_1x + c_0)$ has no roots with probability $\eta_{2,0}' = \frac{1}{2}\frac{(p-1)}{p}$ and distinct roots with probability $\eta_{2,1}' = \frac{1}{2}\frac{(p-1)}{p}$. In the case of distinct roots, one can check that either root lifts to the $x$-coordinate of a $\Q_p$-point with $y \in p\Z_p$. A double root occurs with probability $\eta_{2,2}' = \frac{1}{p}$, and the probability of a solution lifting is equal to $\tau_2$, to be defined later shown to be $\tau_2 = \frac{1}{p}$ in Lemma \ref{lem:taus}. Thus we have
		\[\lambda_h = \left(1 - \frac{1}{p} \right) \left( \eta_{2,1}' + \eta_{2,2}'\tau_2 \right) + \frac{1}{p}\lambda_i = \left(1 - \frac{1}{p} \right) \left( \frac{p-1}{2p} + \frac{1}{p^2} \right) + \frac{1}{p}\lambda_i.\]

		\item Reducing mod $p$ we see that any solution must have $p \mid y$. If $v(c_1) = 1$ then $\frac{1}{p}(c_1x + c_0)$ is linear in $x$ and has a solution modulo $p$. A straightforward check shows that an $x$-value solving this equation modulo $p$ lifts to a solution of $F(x,y,1) \equiv 0 \pmod{p^3}$ with $p \mid y$, and hence to a $\Q_p$-solution by Hensel's lemma. With probability $\frac{1}{p}$ we move to the next line.
		
		\item Reducing mod $p^2$, we have $c_0 \in p^2 \Z_p$ is necessary to obtain a $\Q_p$-solution. This occurs with probability $\frac{1}{p}$, and we move to the next line.
		
		\item With probability $1 - \frac{1}{p}$ we have $v(c_4) = 2$ and we are in the case of $\mu$ from Lemma \ref{lem:mu}. With probability $\frac{1}{p}$ we move to the next line.
		
		\item[($\ell$)]  With probability $1 - \frac{1}{p}$ we have $v(c_3) = 2$ and we are in the case of $\alpha''$. See \eqref{eq:alpha_prime_prime} from Lemma \ref{lem:master}. With probability $\frac{1}{p}$ we move to the next line.
		
		\setcounter{enumi}{12}
		\item With probability $1 - \frac{1}{p}$ we have $v(c_2) = 2$. It is clear that for any solution, we must have $p \mid y$ and $\frac{1}{p^2}(c_2x^2 + c_1x + c_0) \equiv 0 \pmod{p}$. The quadratic $\frac{1}{p^2}(c_2x^2 + c_1x + c_0)$ has no roots with probability $\eta_{2,0}' = \frac{1}{2}\frac{(p-1)}{p}$ and distinct roots with probability $\eta_{2,1}' = \frac{1}{2}\frac{(p-1)}{p}$. In the case of distinct roots, one can check that either root lifts to the $x$-coordinate of a $\Q_p$-point with $y \in p\Z_p$. A double root occurs with probability $\eta_{2,2}' = \frac{1}{p}$, and the probability of a solution lifting is equal to $\theta_2$, to be defined later shown to be $\theta_2 = \frac{(p+2)(2p^2 - 3p + 3)}{6p^2}$ in Lemma \ref{lem:thetas}. Thus we have
		\[\lambda_m = \left(1 - \frac{1}{p} \right) \left( \eta_{2,1}' + \eta_{2,2}'\theta_2 \right) + \frac{1}{p}\lambda_n = \left(1 - \frac{1}{p} \right) \left( \frac{p-1}{2p} + \frac{(p+2)(2p^2 - 3p + 3)}{6p^3} \right) + \frac{1}{p}\lambda_n.\]

		\item If $v(c_1) = 2$ then we can lift a root of $\frac{1}{p^2}(c_1x + c_0) \equiv 0 \pmod{p}$ to a solution. If not, we move to the next line.
		
		\item We have $p \mid y$, so reducing modulo $p^3$ shows that $c_0 \in p^3\Z_p$ is necessary, which occurs with probability $\frac{1}{p}$. Replacing $y$ by $py$ and dividing by $p^3$ moves us to the next line.
		
		\item Recalling that $c_6$ is not congruent to a cubic residue mod $p$, we are in the case of $\rho^*$, because we have assumed no conditions on the coefficients except the $(*)$ condition.
	\end{enumerate}
\end{proof}

\subsubsection{Completing the proof of Proposition \ref{prop:sigma_4}}\label{sec:proof_sigma4}

\begin{proof}[Proof of Proposition \ref{prop:sigma_4}]
	Lemma \ref{lem:sigma4_double_root} gives us a linear relation between $\rho^*(p)$ and $\lambda$. Here we give another such relation, and solve the system for $\rho^*(p)$ and $\lambda$. Recall
	\[\rho^*(p) = \xi_1^* + \xi_4^*\sigma_4^*\]
	since $\sigma_1^* = 1$ when $p > 43$. We have given the values of $\xi_i^*$ in Corollary \ref{cor:factorization_probs}. We break $\sigma_4^*$ down into the cases where $h_0(x,z)$ has no roots, distinct simple roots, or a double root, giving us
	\begin{align*}
		\sigma_4^* &= \eta_{2,1}'(1 - (1-\alpha)^2) + \eta_{2,2}' \lambda\\
		&=\left(\frac{p-1}{2p}\right)\left(1 - (1 - \alpha)^2\right) + \frac{1}{p} \lambda 
	\end{align*}
	by Corollary \ref{cor:sigma_4_h_splits}. Combining these, we have the relation
	\[\rho^*(p) = \xi_1^* + \xi_4^*\left(\left(\frac{p-1}{2p}\right)\left(1 - (1 - \alpha)^2\right) + \frac{1}{p} \lambda \right).\]
	We also write $\sigma_4$ as
	\begin{align*}
		\sigma_4 &= \eta_{2,1}(1-(1-\alpha)^2) + \eta_{2,2}\lambda\\
		&= \left( \frac{1}{2}\frac{p(p^2 - 1)}{p^3 - 1} \right)\left( 1 - (1-\alpha)^2\right)  + \left( \frac{p^2 - 1}{p^3 - 1} \right) \lambda.
	\end{align*}
	Using the above relations and \eqref{eq:lambda_rho_star_relation} from Lemma \ref{lem:sigma4_double_root}, we have four equations relating $\sigma_4, \sigma_4^*, \rho^*, \lambda$, which may be solved using computer algebra software to produce \eqref{eq:rho_star} -- \eqref{eq:lambda}. \href{https://github.com/c-keyes/Density-of-locally-soluble-SECs/blob/bd6a8b39ea8c63bf8e7a847063c70998d01ee8aa/SEC_rho36_23Aug21.ipynb}{See here} for an implementation using Sage \cite{sagemath}.
\end{proof}

\subsection{Triple factors: computing \texorpdfstring{$\sigma_5$}{sigma5}}\label{sec:exact_triple}

Suppose $p \neq 3$. Then if $\overline{F}$ has factorization type 5, we have $F\equiv y^3 \pmod{p}$, so the coefficients of $f(x,z)$ are all divisible by $p$. 

For any solution, that is, for any $[x_0: y_0: z_0]$ such that $F(x_0,y_0,z_0)=0$, we have that $y_0^3=f(x_0,z_0)$. Since we are assuming that $[x_0: y_0: z_0]$ is a solution, we also have that $0=y_0^3-f(x_0,z_0)=y_0^3 \pmod{p}$, and thus $p \mid y_0^3$ so we have that $p^3 \mid y_0^3$. Since we have $y_0^3=f(x_0,z_0)$, then $p^3\mid f(x_0,z_0)$.

Writing $f_1=\frac{1}{p} f$, we see that each solution $[x_0: y_0: z_0]$ satisfies $f_1(x_0,z_0)\equiv y_0 \equiv 0 \pmod{p}$. Thus we will consider the different possible factorizations of $f_1(x,z)$ modulo $p$. If $f_1(x,z) \equiv 0\pmod{p}$ then all the coefficients of $f$ are divisible by $p^2$ and so we can write $f_2=\frac{1}{p^2} f$. Now each solution $[x_0: y_0: z_0]$ must satisfy $f_2(x_0,z_0)\equiv y_0 \equiv 0 \pmod{p}$ and so we consider  the different possible factorizations of $f_2(x,z)$ modulo $p$. 

Now if $f_2(x,z)\equiv 0 \pmod{p}$, then all of the coefficients of $f(x,z)$ are divisible by $p^3$. In this case we can replace $y$ by $py$ and divide through by $p^3$ and obtain another arbitrary superelliptic curve with $m=3$ and $d=6$, namely, $y^3=\frac1p f_2(x,z)$ with coefficients in $\mathbb{Z}_p$ in which case, the probability of solubility is $\rho$. This occurs with probability $\frac{1}{p^{14}}$.

For $i=0, \dots, 9$, we denote by $\eta_{6,i}$ the probability of each possible factorization type for a binary sextic modulo $p$ (see Lemma \ref{lem:etas}), and we denote by $\tau_i$ (respectively $\theta_i$) the probability of solubility of $f_1$ (respectively $f_2$) with factorization type $i$ modulo $p$. Thus we have:
\[\sigma_5= \frac{1}{p^{14}}\rho +\left( 1-\dfrac{1}{p^7} \right)\sum\limits_{i=0}^9 \eta_{6,i} \tau_i +\left(\frac{p^7-1}{p^{14}}\right) \sum\limits_{i=0}^9 \eta_{6,i} \theta_i. \]

In order to compute the $\tau_i$ and $\theta_i$, we make the following definitions. Let $\sigma_5'$ be the probability that $F(x,y,z)$ has a $\Q_p$-solution when $v(c_6) = 1$ and $v(c_i) \geq 1$ for $0 \leq i \leq 5$. Take $\sigma_5''$ to be the probability that $F(x,y,z)$ has a $\Q_p$-solution when $v(c_6) = 2$ and $v(c_i) \geq 2$ for $0 \leq i \leq 5$.

\begin{proposition}\label{prop:sigma_5}
	Suppose $C_f$ is given by \eqref{eq:sec_homog} and $\overline{F}$ has factorization type 5 modulo $p$ for a prime $p > 43$ or $p > 2$ with $p \equiv 2 \pmod{3}$. Then the proportion of $f$ for which $C_f(\Q_p) \neq \emptyset$ is a rational function in $p$ given explicitly by
	\[\sigma_5 = \eqref{eq:sigma_5}.\]
	The proportions $\sigma_5'$ and $\sigma_5''$ defined above are also rational functions in $p$,
	\begin{align*}
		\sigma_5' &= \eqref{eq:sigma_5_prime},\\
		\sigma_5'' &= \eqref{eq:sigma_5_prime_prime}.
	\end{align*}
\end{proposition}

The proofs of these equalities are spread across the remainder of this section. $\sigma_5'$ is computed in the proof Lemma \ref{lem:taus}, along with the values of $\tau_i$. $\sigma_5''$ is computed in the proof of Lemma \ref{lem:thetas}, along with the $\theta_i$ values. The proof is completed in \S\ref{sec:exact_proof_rho} with the computation of $\sigma_5$, along with $\rho_{3,6}$.

\begin{lemma}\label{lem:taus}
	The $\tau_i$ values are tabulated below. These hold for all primes $p > 31$, and for $p > 3$ if $p \equiv 2 \pmod{3}$.
	\begin{align*}
		\tau_0 &= 0 & \tau_5 &= \begin{cases} \frac{3  p^{3} + p^{2} + 2  p + 2}{3  {\left(p^{4} + p^{3} + p^{2} + p + 1\right)}}, & p \equiv 1 \pmod{3}\\[2ex] \frac{{\left(3  p^{2} + 2\right)} {\left(p + 1\right)}}{3  {\left(p^{4} + p^{3} + p^{2} + p + 1\right)}},& p \equiv 2 \pmod{3} \end{cases}\\
		\tau_1 &= 1 & \tau_6 &= 1 - (1-\tau_5)^2 = \eqref{eq:tau_6}\\
		\tau_2 &= \frac{1}{p} & \tau_7 &= \eqref{eq:tau_7}\\
		\tau_3 &= 1 - (1-\tau_2)^2 = \frac{2  p - 1}{p^{2}} & \tau_8 &= 1 - (1-\tau_2)(1 - \tau_7) = \eqref{eq:tau_8} \\
		\tau_4 &= 1 - (1 - \tau_2)^3 = \frac{3  p^{2} - 3  p + 1}{p^{3}} & \tau_9 &= \eqref{eq:tau_9}\\
	\end{align*}
\end{lemma}

\begin{proof} Recall that $f_1 = \frac{1}{p}f$ and assume $f_1 \not \equiv 0 \pmod{p}$. We consider the possible factorization types of $f_1$ as a binary sextic form, given by the index $i$ in the $d=6$ row of Lemma \ref{lem:etas}, and compute the probabilities $\tau_i$ of a root $f_1(x,z) \equiv 0$ lifting to a $\Q_p$-point of $C_f$.

\subsubsection*{No roots: \texorpdfstring{$\tau_0$}{tau0}}
If $f_1(x,z) \equiv 0 \pmod{p}$ has no roots in $\mathbb{F}_p$, then $f(x,z)\equiv 0 \pmod{p^2}$ has no solutions and thus $f(x,z)\equiv 0 \pmod{p^3}$ has no solutions, so $\tau_0=0$.

\subsubsection*{Simple roots: \texorpdfstring{$\tau_1$}{tau1}}
If $f_1(x,z) \equiv 0 \pmod{p}$ has a simple root in $\mathbb{F}_p$, it lifts to a $\Q_p$-point on $C_f$ of the form $[x_0 : y_0 : z_0]$ with $p \mid y_0$ (see also $\lambda_i$ in the proof of Lemma \ref{lem:sigma4_double_root}), so $\tau_1=1$.

\subsubsection*{Double roots: \texorpdfstring{$\tau_2$}{tau2}, \texorpdfstring{$\tau_3$}{tau3}, \emph{and} \texorpdfstring{$\tau_4$}{tau4}}
If $f_1(x,z) \equiv 0 \pmod{p}$ has a double root in $\mathbb{F}_p$, we can assume this root occurs at $[0 : 1]$ and the valuations of the coefficients are as in the first line of the following table. 

\begin{center}
	\begin{tabular}{l l l | l l l l l l l}
	& & & $c_6$ & $c_5$ & $c_4$ & $c_3$ & $c_2$ & $c_1$ & $c_0$ \\
	
	 %line 1
	$\tau_2 =$ & $\tau_{2a} =$ & $\frac{1}{p} \tau_{2b}$
		& $\geq 1$ & $\geq 1$ & $\geq 1$ & $\geq 1$ & $= 1$ & $\geq 2$ & $\geq 2$\\
		
	% line 2
	& $\tau_{2b} =$ & $1$
		& $\geq 4$ & $\geq 3$ & $\geq 2$ & $\geq 1$ & $= 0$ & $\geq 0$ & $\geq 0$\\
	\end{tabular}
	\end{center}
	
\begin{enumerate}[label = (\alph*)]
	\item Since $p \mid x$, reducing modulo $p^3$ reveals that $p \mid c_0$ is necessary, which occurs with probability $\frac{1}{p}$. Before moving to the next line, we replace $x,y$ by $px, py$ and divide by $p^3$.

	\item The justification is identical to that of $\lambda_{d}$. 
\end{enumerate}
Thus we have $\tau_2=\frac{1}{p}$.

If $f_1(x,z)$ has two double roots in $\F_p$, then after composing with an automorphism of $\P^1$, they occur at $[0:1]$ and $[1:0]$, and we have $v(c_0), v(c_1), v(c_5), v(c_6) \geq 2$ and $v(c_2) = v(c_4) = 1$. The probability, $\tau_2$, that the root at $[0:1]$ lifts to a $\Q_p$-point depends only on $c_0$; the same argument with $c_6$ shows that $\tau_2$ is the probability $[1:0]$ lifts and thus the two are independent. This allows us to write
\[\tau_3= 1 - \left(1 - \tau_2\right)^2 = \frac{2p-1}{p^2}.\]

If $f_1(x,z)$ has three double roots in $\F_p$, then after composing with an automorphism of $\P^1$, we may assume they occur at $[0:1]$, $[1:1]$, and $[1:0]$. To extend the independence argument above to these three roots, we need to argue that the probability of $[1:1]$ lifting is still $\tau_2$, even if we assume the points at $[0:1]$ and $[1:0]$ do not lift. To see this, we recognize that $f(1,1) = \sum_{i=0}^d c_i$, and our assumption that $f_1$ has a double root at $[1:1]$ is equivalent to requiring $v\left(\sum_{i=0}^d c_i\right) \geq 2$. Running through the proof of $\tau_2$ shows that we need only for $v\left(\sum_{i=0}^d c_i\right) \geq 3$, which occurs with probability $\frac{1}{p}$, independent of the valuations of $c_0, c_1, c_5, c_6$ (i.e.\ independent of the lifting behavior at $[0:1]$ and $[1:0]$). Therefore
\[\tau_4= 1 - \left(1 - \tau_2\right)^3=\frac{3p^2-3p+1}{p^3}.\]

\subsubsection*{Triple roots: \texorpdfstring{$\tau_5$}{tau5} \emph{and} \texorpdfstring{$\tau_6$}{tau6}}
If $f_1(x,z) \equiv 0 \pmod{p}$ has a triple root in $\mathbb{F}_p$ we can assume the valuations of the coefficients are as in line $1$ of the following table.

\begin{center}
	\begin{tabular}{l l l | l l l l l l l}
	& & & $c_6$ & $c_5$ & $c_4$ & $c_3$ & $c_2$ & $c_1$ & $c_0$ \\
	
	% line 1
	$\tau_5 =$ & $\tau_{5a} =$ & $\frac{1}{p}\tau_{5b}$
		& $\geq 1$ & $\geq 1$ & $\geq 1$ & $=1$ & $\geq 2$ & $\geq 2$ & $\geq 2$\\
		
	% line 2
	& $\tau_{5b} =$ & $\beta'$
		& $\geq 4$ & $\geq 3$ & $\geq 2$ & $= 1$ & $\geq 1$ & $\geq 0$ & $\geq 0$\\
		
	\end{tabular}
	\end{center}
\begin{enumerate}[label = (\alph*)]
	\item Since $p \mid x$, reducing modulo $p^3$ reveals that $p \mid c_0$ is necessary, which occurs with probability $\frac{1}{p}$. Before moving to the next line, we replace $x,y$ by $px, py$ and divide by $p^3$.
	
	\item We are in the situation of $\beta' = \eqref{eq:beta_prime}$; see Lemma \ref{lem:master}.
\end{enumerate}

In the case of $\tau_5$, $f(x,z)$ has one triple root and no other roots, therefore any $\mathbb{Q}_p$ point must come from lifting the triple root which happens with probability $\frac{1}{p}\beta'$, thus giving $\tau_5 = \frac{1}{p}\beta'$, which is equal to the stated expression.

In the case of $\tau_6$, $f(x,z)$ has two triple roots. We assumed that one triple root was at $[x:z]=[0:1]$ and we could similarly assume the other triple root is at $[1:0]$. The probabilities of these lifting are independent because the computation of the former involves coefficients $c_2,c_1,c_0$ and the second involves coefficients $c_4,c_5,c_6$. Thus each point lifts to a $\mathbb{Q}_p$ point with probability $\frac{1}{p}\beta'$, and hence 
\[\tau_6 = 1- (1-\tau_5)^2 = \eqref{eq:tau_6}.\]
  
 \subsubsection*{Quadruple roots \texorpdfstring{$\tau_7$}{tau7} and \texorpdfstring{$\tau_8$}{tau8}}
 
If $f_1(x,z) \equiv 0 \pmod{p}$ has a quadruple root in $\mathbb{F}_p$, we can assume the the root occurs at $[0 : 1]$ and the coefficients have valuations as listed in the first line of the following table.  
 
\begin{center}
	\begin{tabular}{l l l | l l l l l l l}
	& & & $c_6$ & $c_5$ & $c_4$ & $c_3$ & $c_2$ & $c_1$ & $c_0$ \\
	
	% line 1
	$\tau_7 =$ & $\tau_{7a} =$ & $\frac{1}{p}\tau_{7b}$
		& $\geq 1$ & $\geq 1$ & $= 1$ & $\geq 2$ & $\geq 2$ & $\geq 2$ & $\geq 2$\\
		
	% line 2
	& $\tau_{7b} =$ & $\left(1-\frac{1}{p}\right)+\frac{1}{p}\tau_{7c}$
		& $\geq 4$ & $\geq 3$ & $= 2$ & $\geq 2$ & $\geq 1$ & $\geq 0$ & $\geq 0$\\
		
	% line 3
	& $\tau_{7c} =$ & $\Phi(p) +\frac{1}{p} \tau_{7d}$
		& $\geq 4$ & $\geq 3$ & $= 2$ & $\geq 2$ & $\geq 1$ & $\geq 1$ & $\geq 0$\\
		
	% line 4
	& $\tau_{7d} =$ & $\left(1 - \frac{1}{p} \right) \left( \frac{p-1}{2p} + \frac{1}{p^2} \right) + \frac{1}{p}\tau_{7e}$
		& $\geq 4$ & $\geq 3$ & $= 2$ & $\geq 2$ & $\geq 1$ & $\geq 1$ & $\geq 1$\\
		
	% line 5
	& $\tau_{7e} =$ & $\left(1-\frac{1}{p}\right)+ \frac{1}{p}\tau_{7f}$
		& $\geq 4$ & $\geq 3$ & $= 2$ & $\geq 2$ & $\geq 2$ & $\geq 1$ & $\geq 1$\\	
		
	% line 6
	& $\tau_{7f} =$ & $\frac{1}{p}\tau_{7g}$
		& $\geq 4$ & $\geq 3$ & $= 2$ & $\geq 2$ & $\geq 2$ & $\geq 2$ & $\geq 1$\\	
			
	% line 7
	& $\tau_{7g} =$ & $\mu$
		& $\geq 4$ & $\geq 3$ & $= 2$ & $\geq 2$ & $\geq 2$ & $\geq 2$ & $\geq 2$\\		
	\end{tabular}
	\end{center}

\begin{enumerate}[label = (\alph*)]
	\item Since $p \mid x$, reducing modulo $p^3$ reveals that $p \mid c_0$ is necessary, which occurs with probability $\frac{1}{p}$. Before moving to the next line, we replace $x,y$ by $px, py$ and divide by $p^3$.

	\item With probability $1 - \frac{1}{p}$ we have $v(c_1) = 0$ and the reduced equation is $\overline{F}(x,y,1) = y^3 - c_1x - c_0$, which is linear in $x$ and thus has a solution with $y \in p\Z_p$. With probability $\frac{1}{p}$ we have $v(c_1) \geq 1$ and move to the next line.
	
	\item The reduced equation is now $\overline{F} = y^3 - c_0 $. The probability that $c_0$ is a nonzero cubic residue is \[\Phi(p) = \begin{cases} \frac{1}{3}\left(1 - \frac{1}{p} \right), & p \equiv 1 \pmod{3} \\ 1 - \frac{1}{p} , & p \equiv 2 \pmod{3}\end{cases}\] (see Proposition \ref{prop:bound_matrix} for the definition of $\Phi(p)$). If $p \nmid c_0$ is not a cubic residue, then no point lifts. With probability $\frac{1}{p}$ we have $v(c_0)= 1$ and we move to the next line.

	\item With probability $1 - \frac{1}{p}$ we have $v(c_2) = 1$. It is clear that for any solution, we must have $p \mid y$ and $\frac{1}{p}(c_2x^2 + c_1x + c_0) \equiv 0 \pmod{p}$. The quadratic $\frac{1}{p}(c_2x^2 + c_1x + c_0)$ has no roots with probability $\eta_{2,0}' = \frac{1}{2}\frac{(p-1)}{p}$ and distinct roots with probability $\eta_{2,1}' = \frac{1}{2}\frac{(p-1)}{p}$. In the case of distinct roots, one can check that either root lifts to the $x$-coordinate of a $\Q_p$-point with $y \in p\Z_p$. A double root occurs with probability $\eta_{2,2}' = \frac{1}{p}$, and the probability of a solution lifting is equal to $\tau_2 = \frac{1}{p}$. Thus we have
		\[\tau_{7d} = \left(1 - \frac{1}{p} \right) \left( \eta_{2,1}' + \eta_{2,2}'\tau_2 \right) + \frac{1}{p}\tau_{7e} = \left(1 - \frac{1}{p} \right) \left( \frac{p-1}{2p} + \frac{1}{p^2} \right) + \frac{1}{p}\tau_{7e}.\]

	\item Modulo $p^2$ we have $\frac{1}{p}(c_1x+c_0)$, so if $v_p(c_1)=1$, which happens with probability $1-\frac{1}{p}$, then there is a solution that lifts to $\mathbb{Q}_p$. 

	\item Reducing modulo $p^2$, we see that $p^2\mid c_0$ is necessary to get a $\mathbb{Q}_p$ solution. This happens with probability $\frac{1}{p}$ and so we move to the next line.

	\item We are now in the case of $\mu = \eqref{eq:mu}$ from Lemma \ref{lem:mu}. 
\end{enumerate}

In the case of $\tau_7$ there is only a quadruple root so any $\mathbb{Q}_p$-point must come from lifting a quadruple root which happens with probability $\tau_7$, computed to be \eqref{eq:tau_7}.

In the case of $\tau_8$, there is a quadruple root and a double root. We can make the usual argument about independence of these lifting based on which coefficients were used for the argument. Then \[\tau_8=1-(1-\tau_7)(1-\tau_2) = \eqref{eq:tau_8}.\]

\subsubsection{Sextuple roots: \texorpdfstring{$\tau_9$}{tau9} and \texorpdfstring{$\sigma_5'$}{sigma5'}}

If $f_1(x,z) \equiv 0 \pmod{p}$ has a sextuple root in $\mathbb{F}_p$, we can assume this root occurs at $[0 : 1]$ and the valuations of the coefficients are as in the first line of the following table.

\begin{center}
	\begin{tabular}{l l l | l l l l l l l}
	& & & $c_6$ & $c_5$ & $c_4$ & $c_3$ & $c_2$ & $c_1$ & $c_0$ \\
	
	% line 1
	$\tau_9 =$ & $\tau_{9a} =\frac1p \tau_{9b}$ & 
		& $= 1$ & $\geq 2$ & $\geq 2$ & $\geq 2$ & $\geq 2$ & $\geq 2$ & $\geq 2$\\
		
	% line 2
	& $\tau_{9b} = \left(1-\frac{1}{p}\right)+\frac{1}{p}\tau_{9c}$ & 
		& $= 4$ & $\geq 4$ & $\geq 3$ & $\geq 2$ & $\geq 1$ & $\geq 0$ & $\geq 0$\\
		
	% line 3
	& $\tau_{9c} =\Phi(p) +\frac{1}{p} \tau_{9d}$ & 
		& $= 4$ & $\geq 4$ & $\geq 3$ & $\geq 2$ & $\geq 1$ & $\geq 1$ & $\geq 0$\\
		
	% line 4
	& $\tau_{9d} = \left(1 - \frac{1}{p} \right) \left( \frac{p-1}{2p} + \frac{1}{p^2} \right) + \frac{1}{p}\tau_{9e}$ & 
		& $= 4$ & $\geq 4$ & $\geq 3$ & $\geq 2$ & $\geq 1$ & $\geq 1$ & $\geq 1$\\
	
	% line 5
	& $\tau_{9e} =\left(1-\frac{1}{p}\right)+ \frac{1}{p}\tau_{9f}$ & 
		& $= 4$ & $\geq 4$ & $\geq 3$ & $\geq 2$ & $\geq 2$ & $\geq 1$ & $\geq 1$\\
		
	% line 6
	& $\tau_{9f} =\frac1p \tau_{9g}  $ & 
		& $= 4$ & $\geq 4$ & $\geq 3$ & $\geq 2$ & $\geq 2$ & $\geq 2$ & $\geq 1$\\
	
	% line 7
	& $\tau_{9g} = \left(1 - \frac{1}{p}\right)\alpha'' + \frac{1}{p}\tau_{9h}$ & 
		& $= 4$ & $\geq 4$ & $\geq 3$ & $\geq 2$ & $\geq 2$ & $\geq 2$ & $\geq 2$\\
		
	% line 8
	& $\tau_{9h} = \left(1 - \frac{1}{p} \right) \left( \frac{p-1}{2p} + \frac{\theta_2}{p} \right) + \frac{1}{p}\tau_{9i}$ & 
		& $= 4$ & $\geq 4$ & $\geq 3$ & $\geq 3$ & $\geq 2$ & $\geq 2$ & $\geq 2$\\	
		
	% line 9
	& $\tau_{9i} = \left(1 - \frac{1}{p}\right) + \frac{1}{p}\tau_{9j}$ & 
		& $= 4$ & $\geq 4$ & $\geq 3$ & $\geq 3$ & $\geq 3$ & $\geq 2$ & $\geq 2$\\	
			
	% line 10
	& $\tau_{9j} = \frac{1}{p}\tau_{9k}$ & 
		& $= 4$ & $\geq 4$ & $\geq 3$ & $\geq 3$ & $\geq 3$ & $\geq 3$ & $\geq 2$\\	
		
	% line 11
	& $\tau_{9k} = \left(1-\frac{1}{p}\right) + \frac{1}{p}\tau_{9\ell}$ & 
		& $= 1$ & $\geq 1$ & $\geq 0$ & $\geq 0$ & $\geq 0$ & $\geq 0$ & $\geq 0$\\	
	
	% line 12
	& $\tau_{9\ell} = \Phi(p) + \left(1 - \Phi(p) - \frac{1}{p} \right)\beta + \frac{1}{p}\tau_{9m}$ & 
		& $= 1$ & $\geq 1$ & $\geq 1$ & $\geq 0$ & $\geq 0$ & $\geq 0$ & $\geq 0$\\		
		
	% line 13
	& $\tau_{9m} = \left(1-\frac{1}{p}\right) + \frac{1}{p}\tau_{9n}$ & 
		& $= 1$ & $\geq 1$ & $\geq 1$ & $\geq 1$ & $\geq 0$ & $\geq 0$ & $\geq 0$\\	
		
	% line 14
	& $\tau_{9n} = \left(1-\frac{1}{p}\right) + \frac{1}{p}\tau_{9o}$ & 
		& $= 1$ & $\geq 1$ & $\geq 1$ & $\geq 1$ & $\geq 1$ & $\geq 0$ & $\geq 0$\\	
		
	% line 15
	& $\tau_{9o} = \Phi(p) + \frac{1}{p}\tau_{9p}$ & 
		& $= 1$ & $\geq 1$ & $\geq 1$ & $\geq 1$ & $\geq 1$ & $\geq 1$ & $\geq 0$\\	
		
	% line 16
	& $\tau_{9p} = \sigma_5'$ & 
		& $= 1$ & $\geq 1$ & $\geq 1$ & $\geq 1$ & $\geq 1$ & $\geq 1$ & $\geq 1$\\	
	\end{tabular}
	\end{center}
	
	\begin{enumerate}[label = (\alph*)]
		\item Since $p \mid x$, reducing modulo $p^3$ reveals that $p^3 \mid c_0$ is necessary, which occurs with probability $\frac{1}{p}$. Before moving to the next line, we replace $x,y$ by $px, py$ and divide by $p^3$.
		
		\item With probability $1 - \frac{1}{p}$ we have $v(c_1) = 0$ and the reduced equation is $\overline{F}(x,y,1) = y^3 - c_1x - c_0$, which is linear in $x$ and thus has a solution with $y \in p\Z_p$. With probability $\frac{1}{p}$ we have $v(c_1) \geq 1$ and move to the next line.
		
		\item The reduced equation is now $\overline{F} = y^3 - c_0 $. The probability that $c_0$ is a nonzero cubic residue is $\Phi(p)$. If $p \nmid c_0$ is not a cubic residue, then no point lifts. With probability $\frac{1}{p}$ we have $v(c_0)= 1$ and we move to the next line.
		
		\item The justification is identical to $\tau_{7d}$, producing 
		\[\tau_{9d} = \left(1 - \frac{1}{p} \right) \left( \eta_{2,1}' + \eta_{2,2}'\tau_2 \right) + \frac{1}{p}\tau_{9e} = \left(1 - \frac{1}{p} \right) \left( \frac{p-1}{2p} + \frac{1}{p^2} \right) + \frac{1}{p}\tau_{9e}.\]
		
		\item Modulo $p^2$ we have $\frac{1}{p}(c_1x+c_0)$, so if $v_p(c_1)=1$, which happens with probability $1-\frac{1}{p}$, then there is a solution that lifts to $\mathbb{Q}_p$. 
		
		\item Reducing modulo $p^2$, we see that $p^2\mid c_0$ is necessary to get a $\mathbb{Q}_p$ solution. This happens with probability $\frac{1}{p}$ and so we move to the next line.
		
		\item If $v(c_3) = 2$, then we are in the situation of $\alpha'' = \eqref{eq:alpha_prime_prime}$ from Lemma \ref{lem:master}. This happens with probability $1 - 1/p$, so with probability $1/p$ we have $v(c_3) \geq 3$ and we move to the next line.
		
		\item With probability $1 - \frac{1}{p}$ we have $v(c_2) = 2$. It is clear that for any solution, we must have $p \mid y$ and $\frac{1}{p^2}(c_2x^2 + c_1x + c_0) \equiv 0 \pmod{p}$. The quadratic $\frac{1}{p^2}(c_2x^2 + c_1x + c_0)$ has no roots with probability $\eta_{2,0}' = \frac{1}{2}\frac{(p-1)}{p}$ and distinct roots with probability $\eta_{2,1}' = \frac{1}{2}\frac{(p-1)}{p}$. In the case of distinct roots, one can check that either root lifts to the $x$-coordinate of a $\Q_p$-point with $y \in p\Z_p$. A double root occurs with probability $\eta_{2,2}' = \frac{1}{p}$, and the probability of a solution lifting is equal to $\theta_2$, given in Lemma \ref{lem:thetas}. Thus we have
		\[\tau_{9h} = \left(1 - \frac{1}{p} \right) \left( \eta_{2,1}' + \eta_{2,2}'\theta_2 \right) + \frac{1}{p}\tau_{9i} = \left(1 - \frac{1}{p} \right) \left( \frac{p-1}{2p} + \frac{\theta_2}{p} \right) + \frac{1}{p}\tau_{9i}.\]
		
		\item Modulo $p^3$ we have $\frac{1}{p^2}(c_1x+c_0)$, so if $v_p(c_1)=2$, which happens with probability $1-\frac{1}{p}$, then there is a solution that lifts to $\mathbb{Q}_p$. 
		
		\item Reducing modulo $p^3$ reveals that $p^3 \mid c_0$ is necessary, which occurs with probability $\frac{1}{p}$. Before moving to the next line, we replace $y$ by $py$ and divide by $p^3$.
		
		\item With probability $1-\frac{1}{p}$ we have $v(c_4) = 0$, in which case the normalization of $\overline{C_f}$ is seen to have geometric genus at most 3. Since $\overline{f} \neq ah^3$ for $a \in \F_p$ and $h \in \F_p[x,z]$, we apply the Hasse--Weil bound (see the proof of Proposition \ref{prop:bound_p=1_large}) to find
		    \[\#\overline{C_f}^{\sm}(\F_p) \geq p+1 - 6 \sqrt{p} > 1,\]
		    where the rightmost inequality holds for all primes $p > 31$. $\overline{C_f}$ has only the point $[1:0:0]$ above infinity, so there must exist some smooth $\F_p$-point $[x:y:1]$ which lifts to a $\Q_p$-point of $C_f$. If $p \equiv 2 \pmod{3}$, it suffices to take $p > 2$ (see the proof of Proposition \ref{prop:bound_pneq1}). If $v(c_4) \geq 1$ we move to the next line.

		\item[($\ell$)]  With probability $1-\frac{1}{p}$ we have $v(c_3) = 0$, in which case the normalization of $\overline{C_f}$ is seen to have geometric genus at most 1. If $p \equiv 1 \pmod{3}$ and $c_3 \in \left( \F_p^\times \right)^3$ is a nonzero cubic residue, then whenever $\overline{f} \neq c_3(x-\alpha z)$ for $\alpha \in \F_p[x,z]$, we apply the Hasse--Weil bound (see the proof of Proposition \ref{prop:bound_p=1_large}) to find
		    \[\#\overline{C_f}^{\sm}(\F_p) \geq p+1 - 2 \sqrt{p} > 3,\]
		    where the rightmost inequality holds for all primes $p > 7$. In this case, $\overline{C_f}$ must possess a smooth $\F_p$-point $[x:y:1]$ which then lifts to a $\Q_p$-point of $C_f$. On the other hand, if $\overline{f} = c_3(x-\alpha z)$ then the fact that $c_3$ is a cubic residue produces a liftable solution. If $c_3$ is not a cubic residue, the probability of solution is given by $\beta$. If $p \equiv 2 \pmod{3}$, it suffices to take $p > 2$ (see the proof of Proposition \ref{prop:bound_p=1_large}). In either case, we have
		    \[\tau_{9\ell} = \Phi(p) + \left(1 - \Phi(p) - \frac{1}{p} \right)\beta + \frac{1}{p}\tau_{9m},\]
		    which is well defined when $p \equiv 2 \pmod{3}$ even though $\beta$ is not, since $1 - \Phi(p) - \frac{1}{p}=0$.
				
		\setcounter{enumi}{12}
		\item The justification is identical to that of $\lambda_d$.

		\item The justification is identical to that of line (b).
		
		\item The justification is identical to that of line (c).
		
		\item This is the definition of $\sigma_5'$. 
	\end{enumerate}	
	The table above gives us a relation between $\tau_9$ and $\sigma_5'$, while the definition of $\sigma_5'$ gives another:
		\[\sigma_5' = \sum_{i=0}^9 \eta_{6,i}' \tau_i.\]
		Solving the two simultaneously give the values of $\tau_9 = \eqref{eq:tau_9}$ and $\sigma_5' = \eqref{eq:sigma_5_prime}$.	
\end{proof}

\begin{remark}
\label{rem:lifting_double_roots}
	The independence argument used in the computation of $\tau_4$ in terms of $\tau_2$ makes use of the 3-transitivity of $\Aut \P^1$. This argument breaks down if attempting to lift more than three such roots independently, suggesting that more care may be needed to compute $\rho_{m,d}(p)$ exactly when $d \geq 8$.
\end{remark}

\begin{lemma}\label{lem:thetas}
	The $\theta_i$ values are tabulated below. These hold for all primes $p > 31$, and for $p > 3$ if $p \equiv 2 \pmod{3}$.
	\begin{align*}
		\theta_0 &= 0 & \theta_5 &= \begin{cases} \frac{p^{4} + 3  p^{3} + 2  p + 2}{3  {\left(p^{4} + p^{3} + p^{2} + p + 1\right)}}, & p \equiv 1 \pmod{3}\\[1ex] \frac{{\left(3  p^{3} + 2\right)} {\left(p + 1\right)}}{3  {\left(p^{4} + p^{3} + p^{2} + p + 1\right)}},& p \equiv 2 \pmod{3} \end{cases}\\
		\theta_1 &= 1 & \theta_6 &= \eqref{eq:theta_6}\\
		\theta_2 &= \begin{cases} \frac{{\left(2  p^{2} - 3  p + 3\right)} {\left(p + 2\right)}}{6  p^{3}}, & p \equiv 1 \pmod{3}\\[1ex] \frac{{\left(2  p^{2} - 3  p + 2\right)} {\left(p + 1\right)}}{2  p^{3}},& p \equiv 2 \pmod{3} \end{cases} & \theta_7 &= \eqref{eq:theta_7}\\
		\theta_3 &= 1 - (1-\theta_2)^2 = \eqref{eq:theta_3} & \theta_8 &= \eqref{eq:theta_8}\\
		\theta_4 &= 1 - (1-\theta_2)^3 = \eqref{eq:theta_4} & \theta_9 &= \eqref{eq:theta_9}\\
	\end{align*}
\end{lemma}

\begin{proof} Recall that $f_2 = \frac{1}{p^2}f$ and assume $f_2 \not \equiv 0 \pmod{p}$. We consider the possible factorization types of $f_2$ as a binary sextic form, given by the index $i$ in the $d=6$ row of Lemma \ref{lem:etas}, and compute the probabilities $\theta_i$ of a root $f_2(x,z) \equiv 0$ lifting to a $\Q_p$-point of $C_f$.

\subsubsection*{No roots: \texorpdfstring{$\theta_0$}{theta0}}
If $f_2(x,z) \equiv 0 \pmod{p}$ has no roots in $\mathbb{F}_p$, then $f(x,z)\equiv 0 \pmod{p^3}$ has no solutions, so $\theta_0=0$.

\subsubsection*{Simple roots: \texorpdfstring{$\theta_1$}{theta1}}
If $f_2(x,z) \equiv 0 \pmod{p}$ has a simple root in $\mathbb{F}_p$, it lifts to a $\Q_p$-point on $C_f$ of the form $[x_0 : y_0 : z_0]$ with $p \mid y_0$, so $\theta_1 = 1$.

\subsubsection*{Double roots: \texorpdfstring{$\theta_2$}{theta2}, \texorpdfstring{$\theta_3$}{theta3}, \emph{and} \texorpdfstring{$\theta_4$}{theta4}}
If $f_2(x,z) \pmod{p}$ has a double root in $\mathbb{F}_p$, after composition with an automorphism of $\P^1$, we can assume the root occurs at $[0:1]$. Replacing $x,y$ by $px, py$ and dividing by $p^3$ we obtain the valuations of the coefficients listed in the first line of the following table.

\begin{center}
	\begin{tabular}{l l l | l l l l l l l}
	& & & $c_6$ & $c_5$ & $c_4$ & $c_3$ & $c_2$ & $c_1$ & $c_0$ \\
	
	% line 2
	$\theta_2 =$  & $\theta_{2a} =$ & $\Phi(p) +\frac{1}{p}\theta_{2b}$
		& $\geq 5$ & $\geq 4$ & $\geq 3$ & $\geq 2$ & $= 1$ & $\geq 1$ & $\geq 0$\\
		
	% line 3
	& $\theta_{2b} =$ & $\frac{1}{2}\left(1-\frac{1}{p}\right)+\frac{1}{p}\theta_{2c}$
		& $\geq 5$ & $\geq 4$ & $\geq 3$ & $\geq 2$ & $= 1$ & $\geq 1$ & $\geq 1$\\
	
	% line 4
	& $\theta_{2c} =$ & $\frac{1}{p}\theta_{2d}$
		& $\geq 5$ & $\geq 4$ & $\geq 3$ & $\geq 2$ & $= 1$ & $\geq 2$ & $\geq 2$\\
		
	% line 5
	& $\theta_{2d} =$ & $1$
		& $\geq 8$ & $\geq 6$ & $\geq 4$ & $\geq 2$ & $= 0$ & $\geq 0$ & $\geq 0$\\
	\end{tabular}
	\end{center}

\begin{enumerate}[label = (\alph*)]
	\item After the change of variables, reduced equation is $\overline{F} = y^3 - c_0 $. The probability that $c_0$ is a nonzero cubic residue is $\Phi(p)$. If $p \nmid c_0$ is not a cubic residue, then no point lifts. With probability $\frac{1}{p}$ we have $v(c_0)= 1$ and we move to the next line.
	
	\item In this case, we consider the quadratic $\frac{1}{p}(c_2x^2+c_1x+c_0)$ over $\mathbb{F}_p$. If it has a simple root, which happens with probability $\eta_{2,1}' = \frac{1}{2}\left(1-\frac{1}{p}\right)$, these lift to $\mathbb{Q}_p$ points with $y \in p\Z_p$. If it has no roots, the equation is insoluble, and if the quadratic has a double root, which happens with probability $\eta_{2,2}' = \frac{1}{p}$, we can shift it to $[0: 1]$, giving the valuations in the next line.
	
\item Reducing modulo $p^3$ reveals that $p^3 \mid c_0$ is necessary, which occurs with probability $\frac{1}{p}$. Before moving to the next line, we replace $x,y$ by $px,py$ and divide by $p^3$.

\item  The justification is identical to that of $\lambda_{d}$. 
\end{enumerate} 

This gives the expression for $\theta_2$ in the statement. The same independence arguments as for $\tau_2$ and $\tau_3$ in the proof of Lemma \ref{lem:taus} apply here to give
\begin{align*}
	\theta_3 &= 1 - (1-\theta_2)^2 = \eqref{eq:theta_3}. 
\end{align*}

For $\theta_4$, we need to modify the argument from $\tau_4$ slightly. We observe that for $p > 2$ if $c_1, \ldots, c_6$ are fixed (satisfying the conditions above) then as $c_0$ varies, the probability of a lift is precisely $\theta_2$. This is already clear for steps $\theta_{2a}$ and $\theta_{2c}$ above. To see why this holds for $\theta_{2b}$, we consider the quadratic
\[x^2 + \frac{c_1}{c_2}x + \frac{c_0}{c_2},\]
whose discriminant $\frac{c_1^2}{c_2^2}-4\frac{c_0}{c_2}$ determines its factorization type. This discriminant is linear in $c_0$, so for fixed $c_1, c_2$, as $c_0$ runs through $p\Z_p$, it will take quadratic residue/nonresidue values with probility $\frac{p-1}{2p}$, and be divisible by $p$ with probability $\frac1p$. 

Hence, we see that for $p \neq 2$, we can view $\theta_2$ as depending only on the value of $c_0$. Thus, as in the determination of $\tau_4$ in the proof of Lemma \ref{lem:taus}, we see that after moving the roots to $[0:1], [1:1],$ and $[1:0]$, the lifting behavior at $[1:1]$ is independent of the other points, making
\[\theta_4 = 1 - \left( 1-\theta_2\right)^3= \eqref{eq:theta_4}.\]

In the case of $p=2$ --- which will be needed in \S \ref{sec:p=2} --- we observe that the proof of $\theta_2$ above shows that lifting $[0:1]$ depends only on $c_0, c_1$, and the valuation of $c_2$. thus the lifting behavior of $[1:0]$ depends only on $c_5, c_6$, and the valuation of $c_4$. For $\theta_4$, the other double root is located at $[1:1]$, and the lifting argument depends on 
\[f(1,1)= \sum_{i=0}^6 c_i \quad \text{and} \quad f'(1,1) = \sum_{i=0}^6 ic_i.\]
The latter can be controlled by $c_3$, while the former may be controlled by writing $c_2 = 4 + 8c_2'$ for some $c_2'\in \Z_2$ and letting $c_2'$ vary. This is independent of $c_0, c_1, c_5, c_6$, and so we have that $\theta_4 = 1 - (1 - \theta_2)^3 = \eqref{eq:theta_4}$ for $p=2$ as well.

\subsubsection*{Triple roots: \texorpdfstring{$\theta_5$}{theta5} \emph{and} \texorpdfstring{$\theta_6$}{theta6}} 
If $f_2(x,z) \pmod{p}$ has a triple root in $\mathbb{F}_p$, after composition with an automorphism of $\P^1$, we can assume the root occurs at $[0:1]$. Replacing $x,y$ by $px, py$ and dividing by $p^3$ we obtain the valuations of the coefficients listed in the first line of the following table.

\begin{center}
	\begin{tabular}{l l l | l l l l l l l}
	& & & $c_6$ & $c_5$ & $c_4$ & $c_3$ & $c_2$ & $c_1$ & $c_0$ \\
		
	% line 2
	$\theta_5 =$& $\theta_{5a} =$ & $\Phi(p)+\frac{1}{p}\theta_{5b}$
		& $\geq 5$ & $\geq 4$ & $\geq 3$ & $=2$ & $\geq 2$ & $\geq 1$  &$\geq 0$\\
		
	% line 3
	& $\theta_{5b} =$ & $\left(1-\frac{1}{p} \right)+\frac{1}{p}\theta_{5c}$ 
		& $\geq 5$ & $\geq 4$ & $\geq 3$ & $= 2$ & $\geq 2$ & $\geq 1$  & $\geq 1$\\
		
	%line 4
	& $\theta_{5c} =$ & $\frac{1}{p}\theta_{5d}$
		& $\geq 5$ & $\geq 4$ & $\geq 3$ & $= 2$ & $\geq 2$ & $\geq 2$  & $\geq 1$\\
	
	%line 5
	& $\theta_{5d} =$ & $\alpha''$
		& $\geq 5$ & $\geq 4$ & $\geq 3$ & $= 2$ & $\geq 2$ & $\geq 2$  & $\geq 2$\\
				
	\end{tabular}
	\end{center}

\begin{enumerate}[label = (\alph*)]
	\item After the change of variables, reduced equation is $\overline{F} = y^3 - c_0 $. The probability that $c_0$ is a nonzero cubic residue is $\Phi(p)$. If $p \nmid c_0$ is not a cubic residue, then no point lifts. With probability $\frac{1}{p}$ we have $v(c_0)= 1$ and we move to the next line.
	
	\item Modulo $p^2$ we have $\frac{1}{p}(c_1x+c_0)$, so if $v_p(c_1)=1$, which happens with probability $1-\frac{1}{p}$, then there is a solution that lifts to $\mathbb{Q}_p$. 
	
	\item Reducing modulo $p^2$ reveals that $p^2 \mid c_0$ is necessary, which occurs with probability $\frac{1}{p}$, and we move to the next line. 
	
	\item We are in the situation of $\alpha'' = \eqref{eq:alpha_prime_prime}$ from Lemma \ref{lem:master}.
\end{enumerate}

This gives the expression for $\theta_5$ in the statement. The same independence argument as for $\tau_5$ and $\tau_6$ in the proof of Lemma \ref{lem:taus} apply to give
\[\theta_6 = 1 - (1-\theta_5)^2 = \eqref{eq:theta_6}.\]  
  
\subsubsection*{Quadruple roots: \texorpdfstring{$\theta_7$}{theta7} \emph{and} \texorpdfstring{$\theta_8$}{theta8}}

If $f_2(x,z) \pmod{p}$ has a quadruple root in $\mathbb{F}_p$, after composition with an automorphism of $\P^1$, we can assume the root occurs at $[0:1]$. Replacing $x,y$ by $px, py$ and dividing by $p^3$ we obtain the valuations of the coefficients listed in the first line of the following table.

\begin{center}
	\begin{tabular}{l l l | l l l l l l l}
	& & & $c_6$ & $c_5$ & $c_4$ & $c_3$ & $c_2$ & $c_1$ & $c_0$ \\
	
	% line 2
	$\theta_7 =$ &  $\theta_{7a} =$ &$\Phi(p)+\frac{1}{p}\theta_{7b}$
		& $\geq 5$ & $\geq 4$ & $= 3$ & $\geq 3$ & $\geq 2$ & $\geq 1$  &$\geq 0$\\
		
	% line 3
	& $\theta_{7b} =$ & $\left(1-\frac{1}{p} \right)+\frac{1}{p}\theta_{7c}$ 
		& $\geq 5$ & $\geq 4$ & $= 3$ & $\geq 3$ & $\geq 2$ & $\geq 1$  & $\geq 1$\\
		
	%line 4
	& $\theta_{7c} =$ & $\frac{1}{p}\theta_{7d}$
		& $\geq 5$ & $\geq 4$ & $= 3$ & $\geq 3$ & $\geq 2$ & $\geq 2$  & $\geq 1$\\
	
	%line 5
	& $\theta_{7d} =$ & $  \left(1-\frac{1}{p}\right) \left(\frac{1}{2}\left( 1- \frac{1}{p} \right)+ \frac{1}{p}\theta_{2} \right)  +  \frac{1}{p}\theta_{7e}$
		& $\geq 5$ & $\geq 4$ & $= 3$ & $\geq 3$ & $\geq 2$ & $\geq 2$  & $\geq 2$\\
		
	%line 6
	& $\theta_{7e} =$ & $ \left(1-\frac{1}{p} \right) + \frac{1}{p}\theta_{7f}$
		& $\geq 5$ & $\geq 4$ & $= 3$ & $\geq 3$ & $\geq 3$ & $\geq 2$  & $\geq 2$\\
		
	%line 7
	& $\theta_{7f} =$ & $ \frac{1}{p}\theta_{7g}$
		& $\geq 5$ & $\geq 4$ & $= 3$ & $\geq 3$ & $\geq 3$ & $\geq 3$  & $\geq 2$\\
		
	%line 13
	& $\theta_{7g} =$ & $ 1$
		& $\geq 2$ & $\geq 1$ & $= 0$ & $\geq 0$ & $\geq 0$ & $\geq 0$  & $\geq 0$\\
	\end{tabular}
	\end{center}
	
\begin{enumerate}[label = (\alph*)]
	\item The reduced equation is $\overline{F} = y^3 - c_0 $. The probability that $c_0$ is a nonzero cubic residue is $\Phi(p)$. If $p \nmid c_0$ is not a cubic residue, then no point lifts. With probability $\frac{1}{p}$ we have $v(c_0)= 1$ and we move to the next line.

\item Modulo $p^2$ we have $\frac{1}{p}(c_1x+c_0)$, so if $v_p(c_1)=1$, which happens with probability $1-\frac{1}{p}$, then there is a solution that lifts to $\mathbb{Q}_p$. 

	\item Reducing modulo $p^2$ reveals that $p^2 \mid c_0$ is necessary, which occurs with probability $\frac{1}{p}$, and we move to the next line. 

\item With probability $1 - \frac{1}{p}$ we have $v(c_2) = 2$. It is clear that for any solution, we must have $p \mid y$ and $\frac{1}{p^2}(c_2x^2 + c_1x + c_0) \equiv 0 \pmod{p}$. The quadratic $\frac{1}{p^2}(c_2x^2 + c_1x + c_0)$ has no roots with probability $\eta_{2,0}' = \frac{1}{2}\frac{(p-1)}{p}$ and distinct roots with probability $\eta_{2,1}' = \frac{1}{2}\frac{(p-1)}{p}$. In the case of distinct roots, one can check that either root lifts to the $x$-coordinate of a $\Q_p$-point with $y \in p\Z_p$. A double root occurs with probability $\eta_{2,2}' = \frac{1}{p}$, and the probability of a solution lifting is equal to $\theta_2$. Thus we have
		\[\theta_{7d} = \left(1 - \frac{1}{p} \right) \left( \eta_{2,1}' + \eta_{2,2}'\theta_2 \right) + \frac{1}{p}\theta_{7e} = \left(1 - \frac{1}{p} \right) \left( \frac{p-1}{2p} + \frac{\theta_2}{p} \right) + \frac{1}{p}\theta_{7e}.\]
		
	\item Modulo $p^3$ we have $(c_1x+c_0)$, so if $v_p(c_1)=2$, which happens with probability $1-\frac{1}{p}$, then there is a solution that lifts to $\mathbb{Q}_p$. 
		
	\item Reducing modulo $p^3$ reveals that $p^3 \mid c_0$ is necessary, which occurs with probability $\frac{1}{p}$. Before moving to the next line, we replace $y$ by $py$ and divide by $p^3$.

\item The justification is identical to that of $\tau_{9k}$.
\end{enumerate}

This gives the expression for $\theta_7$ in \eqref{eq:theta_7}. The same argument as for $\tau_7$ and $\tau_8$ in the proof of Lemma \ref{lem:taus} applies here to give
\[ \theta_8=1-(1-\theta_7)(1-\theta_2)=\eqref{eq:theta_8}.\]

\subsubsection*{Sextuple roots: \texorpdfstring{$\theta_9$}{theta9}}
If $f_2(x,z) \pmod{p}$ has a sextuple root in $\mathbb{F}_p$, after composition with an automorphism of $\P^1$, we can assume the root occurs at $[0:1]$. Replacing $x,y$ by $px, py$ and dividing by $p^3$ we obtain the valuations of the coefficients listed in the first line of the following table.

\begin{center}
	\begin{tabular}{l l l | l l l l l l l}
	& & & $c_6$ & $c_5$ & $c_4$ & $c_3$ & $c_2$ & $c_1$ & $c_0$ \\
	
	% line 2
	$\theta_9 =$ & $\theta_{9a} = \Phi(p) + \frac1p \theta_{9b}$ & 
		& $= 5$ & $\geq 5$ & $\geq 4$ & $\geq 3$ & $\geq 2$ & $\geq 1$ & $\geq 0$\\
		
	% line 3
	& $\theta_{9b} =\left(1-\frac{1}{p} \right)+\frac{1}{p}\theta_{9c}$ & 
		& $= 5$ & $\geq 5$ & $\geq 4$ & $\geq 3$ & $\geq 2$ & $\geq 1$ & $\geq 1$\\
		
	% line 4
	& $\theta_{9c} = \frac1p \theta_{9d}$ & 
		& $= 5$ & $\geq 5$ & $\geq 4$ & $\geq 3$ & $\geq 2$ & $\geq 2$ & $\geq 1$\\
		
	% line 5
	& $\theta_{9d} =   \left(1-\frac{1}{p}\right) \left(\frac{1}{2}\left( 1- \frac{1}{p} \right)+ \frac{1}{p}\theta_{2} \right)  +  \frac{1}{p}\theta_{9e}$ & 
		& $= 5$ & $\geq 5$ & $\geq 4$ & $\geq 3$ & $\geq 2$ & $\geq 2$ & $\geq 2$\\
	
	% line 6
	& $\theta_{9e} =   \left(1-\frac{1}{p}\right) +\frac1p \theta_{9f}$ & 
		& $= 5$ & $\geq 5$ & $\geq 4$ & $\geq 3$ & $\geq 3$ & $\geq 2$ & $\geq 2$\\
		
	% line 7
	& $\theta_{9f} =\frac1p \theta_{9g}$ & 
		& $= 5$ & $\geq 5$ & $\geq 4$ & $\geq 3$ & $\geq 3$ & $\geq 3$ & $\geq 2$\\
		
	% line 9
	& $\theta_{9g} =  \Phi(p) + \left(1 - \Phi(p) - \frac{1}{p} \right)\beta + \frac{1}{p}\theta_{9h}$ & 
		& $= 2$ & $\geq 2$ & $\geq 1$ & $\geq 0$ & $\geq 0$ & $\geq 0$ & $\geq 0$\\	
		
	% line 10
	& $\theta_{9h} = \left(1-\frac{1}{p}\right)+\frac1p \theta_{9i}$ & 
		& $= 2$ & $\geq 2$ & $\geq 1$ & $\geq 1$ & $\geq 0$ & $\geq 0$ & $\geq 0$\\	
			
	% line 11
	& $\theta_{9i} = \left(1-\frac{1}{p}\right)+\frac1p \theta_{9j}$ & 
		& $= 2$ & $\geq 2$ & $\geq 1$ & $\geq 1$ & $\geq 1$ & $\geq 0$ & $\geq 0$\\	
		
	% line 12
	& $\theta_{9j} = \Phi(p)+\frac1p \theta_{9k}$ & 
		& $= 2$ & $\geq 2$ & $\geq 1$ & $\geq 1$ & $\geq 1$ & $\geq 1$ & $\geq 0$\\	
		
	% line 13
	& $\theta_{9k} =\left(1 - \frac{1}{p}\right) \mu' + \frac1p{\theta_{9\ell}}$ & 
		& $= 2$ & $\geq 2$ & $\geq 1$ & $\geq 1$ & $\geq 1$ & $\geq 1$ & $\geq 1$\\	
	
	% line 14
	& $\theta_{9\ell} = \left(1 - \frac{1}{p}\right) \alpha' +\frac1p \theta_{9m}$ & 
		& $= 2$ & $\geq 2$ & $\geq 2$ & $\geq 1$ & $\geq 1$ & $\geq 1$ & $\geq 1$\\	
		
	% line 15
	& $\theta_{9m} = \left(1 - \frac{1}{p}\right)\left(\frac{p-1}{2p} + \frac{1}{p^2} \right) + \frac1p\theta_{9n} $ & 
		& $= 2$ & $\geq 2$ & $\geq 2$ & $\geq 2$ & $\geq 1$ & $\geq 1$ & $\geq 1$\\	
		
	% line 16
	& $\theta_{9n} = \left(1 - \frac{1}{p}\right) + \frac1p\theta_{9o}$ & 
		& $= 2$ & $\geq 2$ & $\geq 2$ & $\geq 2$ & $\geq 2$ & $\geq 1$ & $\geq 1$\\	
		
	% line 17
	& $\theta_{9o} = \frac1p\theta_{9p}$ & 
		& $= 2$ & $\geq 2$ & $\geq 2$ & $\geq 2$ & $\geq 2$ & $\geq 2$ & $\geq 1$\\	
		
	% line 18
	& $\theta_{9p} = \sigma_5''$ & 
		& $= 2$ & $\geq 2$ & $\geq 2$ & $\geq 2$ & $\geq 2$ & $\geq 2$ & $\geq 2$\\	
	\end{tabular}
	\end{center}
	
	\begin{enumerate}[label = (\alph*)]
	
	\item The reduced equation is now $\overline{F} = y^3 - c_0$. The probability that $c_0$ is a nonzero cubic residue is $\Phi(p)$. If $p \nmid c_0$ is not a cubic residue, then no point lifts. With probability $\frac{1}{p}$ we have $v(c_0)= 1$ and we move to the next line.

	\item Modulo $p^2$ we have $\frac{1}{p}(c_1x+c_0)$, so if $v_p(c_1)=1$, which happens with probability $1-\frac{1}{p}$, then there is a solution that lifts to $\mathbb{Q}_p$. 
		
	\item Reducing modulo $p^2$, we see that $p^2\mid c_0$ is necessary to get a $\mathbb{Q}_p$ solution. This happens with probability $\frac{1}{p}$ and so we move to the next line.
		
		\item The justification is identical to $\theta_{7d}$, producing 
		\[\theta_{9d} = \left(1 - \frac{1}{p} \right) \left( \eta_{2,1}' + \eta_{2,2}'\theta_2 \right) + \frac{1}{p}\theta_{9e} = \left(1 - \frac{1}{p} \right) \left( \frac{p-1}{2p} + \frac{\theta_2}{p} \right) + \frac{1}{p}\theta_{9e}.\]
		
		\item Modulo $p^3$ we have $\frac{1}{p^2}(c_1x+c_0)$, so if $v_p(c_1)=2$, which happens with probability $1-\frac{1}{p}$, then there is a solution that lifts to $\mathbb{Q}_p$. 
		
		\item Reducing modulo $p^3$, we see that $p^3\mid c_0$ is necessary to get a $\mathbb{Q}_p$ solution. This happens with probability $\frac{1}{p}$. Before moving to the next line, we replace $y$ by $py$ and divide by $p^3$.	
			
		\item The justification is identical to that of $\tau_{9\ell}$.
			
		\item The justification is identical to that of $\lambda_d$.
				
		\item With probability $1 - \frac{1}{p}$ we have $v(c_1) = 0$ and the reduced equation is $\overline{F}(x,y,1) = y^3 - c_1x - c_0$, which is linear in $x$ and thus has a solution with $y \in p\Z_p$. With probability $\frac{1}{p}$ we have $v(c_1) \geq 1$ and move to the next line.
		
		\item The reduced equation is now $\overline{F} = y^3 - c_0 $. The probability that $c_0$ is a nonzero cubic residue is $\Phi(p)$. If $p \nmid c_0$ is not a cubic residue, then no point lifts. With probability $\frac{1}{p}$ we have $v(c_0)= 1$ and we move to the next line.
		
		\item If $v(c_4)=1$ we are in the situation of $\mu' = \eqref{eq:mu_prime}$ from Lemma \ref{lem:mu}. This happens with probability $1 - \frac1p$, so with probability $\frac1p$ we have $v(c_4) \geq 2$ and we move to the next line. 
	
	 \item[$(\ell)$] If $v(c_3) = 1$, then we are in the situation of $\alpha' = \eqref{eq:alpha_prime_prime}$ from Lemma \ref{lem:master}. This happens with probability $1 - \frac1p$, so with probability $\frac1p$ we have $v(c_3) \geq 2$ and we move to the next line.
	
	\setcounter{enumi}{12}
		\item The justification is identical to $\tau_{9d}$ of Lemma \ref{lem:taus}, producing 
		\[\theta_{9m} = \left(1 - \frac{1}{p} \right) \left( \eta_{2,1}' + \eta_{2,2}'\tau_2 \right) + \frac{1}{p}\theta_{9n} = \left(1 - \frac{1}{p} \right) \left( \frac{p-1}{2p} + \frac{1}{p^2} \right) + \frac{1}{p}\theta_{9n}.\]
	
		\item The justification is identical to that of line (b). 

		\item The justification is identical to that of line (c). 
		
		\item	This is the definition of $\sigma_5''$.
	\end{enumerate}

The table above gives us a relation between $\theta_9$ and $\sigma_5''$, while the definition of $\sigma_5''$ gives another:
		\[\sigma_5'' = \sum_{i=0}^9 \eta_{6,i}' \theta_i.\]
		Solving the two simultaneously give the values of $\theta_9 = \eqref{eq:theta_9}$ and $\sigma_5'' = \eqref{eq:sigma_5_prime_prime}$.	
\end{proof}

\subsubsection{Completing the proofs of Proposition \ref{prop:sigma_5} and Theorem \ref{thm:exact_rho_3_6}}
\label{sec:exact_proof_rho}

To complete the proof of Proposition \ref{prop:sigma_5}, we must compute $\sigma_5$. In doing so, we will also compute the exact value of $\rho$, thereby completing the proof of part of Theorem \ref{thm:exact_rho_3_6} as well.

\begin{proof}[Proof of Proposition \ref{prop:sigma_5}]

Recall that $\sigma_5' = \eqref{eq:sigma_5_prime}$ and $\sigma_5'' = \eqref{eq:sigma_5_prime_prime}$ were computed in the proofs of Lemmas \ref{lem:taus} and \ref{lem:thetas}, respectively. Thus all that remains is to compute $\sigma_5$.

Recall that $\sigma_5$ is related to $\rho$ by
\[\sigma_5= \frac{1}{p^{14}}\rho +\left( 1-\dfrac{1}{p^7} \right)\sum\limits_{i=0}^9 \eta_{6,i} \tau_i +\left(\frac{p^7-1}{p^{14}}\right) \sum\limits_{i=0}^9 \eta_{6,i} \theta_i, \]
where the values of $\eta_{6_i}, \tau_i, \theta_i$ are given in Lemmas \ref{lem:etas}, \ref{lem:taus}, and \ref{lem:thetas}, respectively. On the other hand, we have
\[\rho = \sum_{i=1}^5 \xi_i \sigma_i,\]
with $\xi_i$ given in Corollary \ref{cor:factorization_probs} and $\sigma_i$ for $1 \leq i \leq 4$ given in Propositions \ref{prop:sigma_1}, \ref{prop:sigma_2=1}, \ref{prop:sigma_3=1}, \ref{prop:sigma_4}, respectively. We can thus solve the two equations above for $\sigma_5$ and $\rho$ as rational functions in $p$,
\begin{align*}
	\rho &= \eqref{eq:rho},\\
	\sigma_5 &= \eqref{eq:sigma_5},
\end{align*}
thereby completing the proof of Proposition \ref{prop:sigma_5}. \href{https://github.com/c-keyes/Density-of-locally-soluble-SECs/blob/bd6a8b39ea8c63bf8e7a847063c70998d01ee8aa/SEC_rho36_23Aug21.ipynb}{See here} again for an implementation in Sage \cite{sagemath}. 

Thus we have verified that for $i=1,2$, we have $\rho(p) = R_i(p)$ for an explicit rational function $R_i(t)$ and all sufficiently large primes $p \equiv i \pmod{3}$ as stated in Theorem \ref{thm:exact_rho_3_6}. It remains to observe the asymptotic behavior, i.e.\ that when $p \equiv 1 \pmod{3}$, this explicit function satisfies
\[1 - \rho(p) \sim \frac{2}{3}p^{-4},\]
and if $p \equiv 2 \pmod{3}$ then
\[1 - \rho(p) \sim \frac{53}{144} p^{-7}.\]
\end{proof}

\subsection{Small primes}
\label{sec:small_primes}

All that remains to prove Theorem \ref{thm:exact_rho_3_6} is to compute $\rho(p)$ for the remaining eight primes $p$, not handled directly by Propositions \ref{prop:sigma_1}, \ref{prop:sigma_2=1}, \ref{prop:sigma_3=1}, \ref{prop:sigma_4}, and \ref{prop:sigma_5}, namely $p=2, 3, 7, 13, 19, 31, 37, 43$. We handle the cases of $p=2$ and $p=3$ separately from the six remaining primes $p \equiv 1 \pmod{3}$ and conclude this section with the the calculation of the $\rho_{3,6} = 96.94\%$.

\subsubsection{\texorpdfstring{$p = 2$}{p=2}}\label{sec:p=2} Suppose $p=2$. By the proof of Proposition \ref{prop:bound_pneq1}, for all binary sextic forms $f(x,z)$ such that $\overline{f} \neq 0, x^2(x+z)^2z^2$, we can lift a point on the reduction $\overline{C_f}$ to a $\Q_2$-point of $C_f$. Thus we first restrict our attention to lifting $\F_2$-points of 
\[y^3 = x^2(x + z)^2z^2.\]
 By the same argument as that for $\theta_4$ in the proof of Lemma \ref{lem:thetas}, the probability of $[0 : 0 : 1]$, $[1 : 0 : 1]$, or $[1 : 0 : 0]$ lifting to a $\Q_2$-point are equal and independent. Thus it suffices to determine how often $[0 : 0 : 1]$ lifts. In fact, we will need the following lemma for all primes $p \neq 3$.
 
\begin{lemma}\label{lem:nu}
	Let $p\neq 3$ be a prime. Fix $c_2$ such that $v(c_2) = 0$ and $c_3, c_4, c_5, c_6 \in \Z_p$. As $c_0, c_1$ range over $p\Z_p$, let $\nu$ denote the probability that the $\F_p$-solution $[0:0:1]$ to $\overline{F}(x,y,z) = 0$ lifts to a $\Q_p$-solution to $F(x,y,z) = 0$. We have
	\[\nu = \frac{1}{p} \left( \eta_{2,1}' + \eta_{2,2}' \theta_2 \right) = \begin{cases} \frac{3  p^{4} - p^{3} + p^{2} - 3  p + 6}{6  p^{5}} & p \equiv 1 \pmod{3},\\[2ex] \frac{p^{4} + p^{3} - p^{2} - p + 2}{2  p^{5}} & p \equiv 2 \pmod{3}.\end{cases}\]
\end{lemma}

\begin{proof}
	Let $[x:y:z] \equiv [0 : 0 : 1] \pmod{p}$. We observe that $v(c_2x^2z^4 + c_1xz^5) \geq 2$, so if $v(c_0) = 1$, the equation is seen to be insoluble modulo $p^2$. With probability $\frac{1}{p}$ we have $c_0 \in p^2\Z_p$.
	
	Replacing $x$ by $px$, our equation becomes
	\[y^3 = \sum_{i=0}^6 p^ic_ix^i.\]
	Rewriting $c_i$ as $p^ic_i$, we have the valuations (in descending order) are given by
\[		\geq 6 \ \geq 5 \ \geq 4 \ \geq 3 \ = 2 \ \geq 2 \ \geq 2.\]
	Thus it is necessary for $p^3 \mid (c_2x^2 + c_1x + c_0)$, so with probability $\eta_{2,0}' = \frac{p-1}{2p}$, the equation is insoluble. With probability $\eta_{2,1}' = \frac{p-1}{2p}$, we have a lift, and with probability $\eta_{2,2}' = \frac{1}{p}$, we are in the situation of $\theta_2$ of Lemma \ref{lem:thetas}, and the proof is seen to be valid for all $p$. Putting this together yields the giving probability that $[0:0:1]$ lifts to a $\Q_p$-solution.
\end{proof}

\begin{corollary}\label{cor:p=2_lifting}
	Let $f(x,z)$ be a binary sextic form with $\overline{f}(x,z) = x^2(x+z)^2z^2 \pmod{2}$. The probability that $C_f$ has a $\Q_2$-point is 
	\[1 - (1-\nu)^3 = \frac{2675}{4096}.\]
	The probability that $C_f$ has an affine $\Q_2$-point of the form $[x:y:1]$ is 
	\[1 - (1-\nu)^2 = \frac{135}{256}.\]
\end{corollary}

\begin{proof}
	The two statements follow from applying Lemma \ref{lem:nu} to the two affine and three total $\F_2$-points of $\overline{C_f}$ independently (see the argument for $\theta_4$ in the proof of Lemma \ref{lem:thetas}).
\end{proof}

At this point, we have
	\begin{equation}
	\label{eq:p=2_rho}
		\rho_{3,6}(2) = 1- \frac{1}{2^6} + \frac{1}{2^7} \left( \frac{2675}{4096} + \sigma_5 \right),
	\end{equation}
	where $\sigma_5$ is the probability of solubility when $\overline{f}(x,z) = 0$.

To compute $\sigma_5$, we can follow the proofs of Proposition \ref{prop:sigma_5}. For $0 \leq i \leq 6$, the values of $\tau_i$ and $\theta_i$ from Lemmas \ref{lem:taus} and \ref{lem:thetas} hold for $p=2$. Corollary \ref{cor:p=2_lifting} can be used to compute $\theta_7, \mu, \tau_7$, and $\mu'$, in that order. We catalog these values below and highlight the modified steps.
	\begin{align*}
		\theta_7 &= \frac{13575}{16384} & \left(\text{use } \theta_{7g} = \frac{135}{256}\right),\\
		\theta_8 &= \frac{62727}{65536} & (\text{use updated } \theta_7),\\
		\mu &= \frac{90887}{131072} & (\text{use updated } \theta_7 \text{ in } \eqref{eq:mu_sum}),\\
		\tau_7 &= \frac{3760903}{8388608} & (\text{use } \tau_{7g} = \mu),\\
		\tau_8 &= \frac{12149511}{16777216} & (\text{use updated } \tau_8) ,\\
		\mu' &= \frac{40461063}{67108864} & (\text{use updated } \tau_7 \text{ in } \eqref{eq:mu_prime_sum}).
	\end{align*}
	
	We then solve the following equations, using the values above and Corollary \ref{cor:p=2_lifting} appropriately in the calculation of $\tau_9$,
	\begin{align*}
		\tau_9 &= \frac{1}{32768}\sigma_5' + \frac{7283817}{16252928} & \left(\text{use } \tau_{9k} = \frac{1}{2} \cdot \frac{136}{256} + \frac{1}{2}\tau_{9\ell}\right), \\
		\sigma_5' &= \sum_{i=0}^9 \eta_{6,i}' \tau_i & \left(\text{use updated } \tau_7, \tau_8 \right),\\
		\theta_9 &= \frac{1}{32768}\sigma_5'' + \frac{3559852801497}{4260607557632} & \left(\text{use updated } \mu' \text{ in }\theta_{9k} \right),\\% = \left(1 - \frac1p \right) \mu' + \frac1p\theta_{9\ell}\right),\\
		\sigma_5'' &= \sum_{i=0}^9 \eta_{6,i}' \theta_i & \left(\text{use updated } \theta_7, \theta_8 \right),\\
		\sigma_5 &= \frac{1}{2^{14}}\rho +\left( 1-\dfrac{1}{2^7} \right)\sum\limits_{i=0}^9 \eta_{6,i} \tau_i +\left(\frac{2^7-1}{2^{14}}\right) \sum\limits_{i=0}^9 \eta_{6,i} \theta_i  ,\\
		\rho &= 1- \frac{1}{2^6} + \frac{1}{2^7} \left( \frac{2675}{4096} + \sigma_5 \right) & (\text{see } \eqref{eq:p=2_rho}).
	\end{align*}
	This yields
	\begin{equation}
	\label{eq:p=2_rho_exact}
		\rho_{3,6}(2) = \frac{45948977725819217081}{46164832540903014400} \approx 0.99532,
	\end{equation}
	a considerable improvement over $1-\frac{1}{2^6}$.

\subsubsection{Primes \texorpdfstring{$p \equiv 1 \pmod{3}$}{p=1 (mod 3)}} 

Suppose $p$ is one of $p=7, 13, 19, 31, 37, 43$. Here we are not able to conclude that when $\overline{F}(x,y,z)$ is absolutely irreducible, that $C_f$ has a $\Q_p$-point, i.e.\ that $\sigma_1 = 1$. At various other junctures, including the calculations of $\tau_i$ and $\theta_i$ in Lemmas \ref{lem:taus} and \ref{lem:thetas}, we use assumptions about the size of $p$ to conclude that certain equations over $\F_p$ always possess a liftable point. To circumvent this and fix the necessary calculations, we need a few intermediate results.

Consider equations of the form 
\begin{equation}\label{eq:rho33}
	y^3 = c_3x^3 + c_2x^2 + c_1x + c_0
\end{equation}
and denote by $\rho_{3,3}^\aff(p)$ the probability that \eqref{eq:rho33} has an affine $\Q_p$-point as $c_0, c_1, c_2, c_3$ vary in $\Z_p$ with $v(c_3) = 0$.

\begin{lemma}\label{lem:rho33aff}
Let $p \equiv 1 \pmod{3}$. When $p > 7$ we have
\[\rho_{3,3}^\aff(p) = \frac{1}{3} + \frac{2}{3}\beta = \frac{3  p^{4} + 3  p^{3} + p^{2} + 3  p + 1}{3  {\left(p^{4} + p^{3} + p^{2} + p + 1\right)}}.\]
In the case of $p=7$,
\[\rho_{3,3}^\aff(7) = \frac{1}{2058}\left(2002 + 28\alpha\right) = \frac{401245}{411747}.\]
\end{lemma}

\begin{proof}
Note that the justification when $p > 7$ is essentially that of $\tau_{9\ell}$ or $\theta_{9g}$. Whenever $c_3 \in \left(\F_p^\times\right)^3$ we have a solution, as the Hasse bound \eqref{eq:hasse-weil} applies to the normalization of the reduction of \eqref{eq:rho33} whenever $p > 7$, and if the right hand side factors as $c_3(x-a)^3$, we can lift $[a+1:y:1]$ using Hensel's lemma. If $c_3 \notin \left(\F_p^\times\right)^3$, then we are in the situation of $\beta$, giving the first statement.

When $p=7$, a computer search shows that of the 2058 equations \eqref{eq:rho33} over $\F_7$ with $v(c_3) = 0$, 2002 can be lifted via Hensel's lemma, 28 are insoluble, and 28 are of the form $y^3 = c_3(x-a)^3$, where $c_3 \notin \left(\F_p^\times\right)^3$. See the procedure \texttt{count_cubic_forms(p)}, contained in the file \texttt{CountForms.m}, which can be \href{https://github.com/c-keyes/Density-of-locally-soluble-SECs/blob/f492b080352291c758e10fe9f82a49618e7e095b/CountForms.m}{found here}, for an implementation in Magma \cite{magma}. The probability of lifting in this case is given by $\alpha$, proving the second statement.
\end{proof}

The following lemma complements Lemma \ref{lem:nu}, in that it provides the probability that a triple root modulo $p$ lifts to a $\Q_p$-point.

\begin{lemma}\label{lem:pi}
Let $p \equiv 1 \pmod{3}$ and fix $c_4, c_5, c_6 \in \Z_p$. As $c_0, c_1, c_2, c_3$ vary in $\Z_p$ with $v(c_3) = 0$, the $\F_p$-solution $[0:0:1]$ to $\overline{F}(x,y,z) = 0$ lifts to a $\Q_p$-solution to $F(x,y,z) = 0$ with probability
	\[\pi = \frac1p - \frac{1}{p^2} + \frac{1}{p^3}\rho_{3,3}^\aff = \begin{cases} \frac{17694619}{141229221} & p=7, \\[2ex] \frac{3  p^{6} + 3  p^{4} + 3  p^{3} + p^{2} + 1}{3  {\left(p^{4} + p^{3} + p^{2} + p + 1\right)} p^{3}} & p > 7. \end{cases}\]
\end{lemma}

\begin{proof}
	The proof follows techniques similar to ones we have already seen. 
	
	\begin{center}
	\begin{tabular}{l l l | l l l l l l l}
	& & & $c_6$ & $c_5$ & $c_4$ & $c_3$ & $c_2$ & $c_1$ & $c_0$ \\
	
	% line a
	$\pi =$ & $\pi_a =\frac1p \pi_b$ & 
		& $\geq 0$ & $\geq 0$ & $\geq 0$ & $= 0$ & $\geq 1$ & $\geq 1$ & $\geq 1$\\
		
	% line b
	& $\pi_b = \left(1-\frac1p\right) + \frac1p \pi_c$ & 
		& $\geq 6$ & $\geq 5$ & $\geq 4$ & $= 3$ & $\geq 3$ & $\geq 2$ & $\geq 2$\\
		
	% line c
	& $\pi_c = \frac1p \pi_d$ & 
		& $\geq 6$ & $\geq 5$ & $\geq 4$ & $= 3$ & $\geq 3$ & $\geq 3$ & $\geq 2$\\
		
	% line d
	& $\pi_d = \rho_{3,3}^{\aff}$ & 
		& $\geq 3$ & $\geq 2$ & $\geq 1$ & $= 0$ & $\geq 0$ & $\geq 0$ & $\geq 0$\\
	\end{tabular}
	\end{center}
	
	\begin{enumerate}[label = (\alph*)]
		\item We observe that $v(c_3x^3 + c_2x^2 + c_1x) \geq 2$, so it is necessary for $v(c_0) \geq 2$, which occurs with probability $\frac{1}{p}$. At this point, we replace $x$ by $px$ and $c_i$ by $p^ic_i$ and move to the next line.
		
		\item The justification is identical to that of $\theta_{9e}$.
		\item The justification is identical to that of $\theta_{9f}$. 
		\item The probability of finding a solution of the form $[x:y:1]$ is precisely $\rho_{3,3}^\aff(p)$ by definition.
	\end{enumerate}
	Putting these steps together, along with the value of $\rho_{3,3}^\aff$ from Lemma \ref{lem:rho33aff} yields the given formula.	
\end{proof}

Consider now equations of the form 
\begin{equation}\label{eq:rho34aff}
	y^3 = c_4x^4 + c_3x^3 + c_2x^2 + c_1x + c_0
\end{equation}
with $v(c_3) = 0$. Let the proportion of equations \eqref{eq:rho34aff} over $\Z_p$ possessing an affine $\Q_p$-point be denoted $\rho_{3,4}^{\aff}(p)$. This quantity came up in computing $\theta_7$ and $\tau_9$, (see in particular $\theta_{7g}$, $\tau_{9k}$), and hence also the quantities derived from them, including $\mu, \mu', \tau_7, \tau_8, \theta_8,$ and $\theta_9$. When $p > 31$, an application of the Hasse--Weil bound \eqref{eq:hasse-weil} is sufficient to guarantee the existence of a $\Q_p$-point; for $p \leq 31$, we have the following.

\begin{lemma}\label{lem:rho34aff}
	Let $p \equiv 1 \pmod{3}$. For $p \geq 31$ we have $\rho_{3,4}^\aff = 1$. For $p < 31$, we have
	\begin{align*}
		\rho_{3,4}^\aff(7) &= \frac{93877018682}{96889010407} \approx 0.96891, \\
		\rho_{3,4}^\aff(13) &= \frac{813159544}{815730721} \approx 0.99684, \\
		\rho_{3,4}^\aff(19) &= \frac{6856}{6859} \approx 0.99956.
	\end{align*}
\end{lemma}

\begin{proof}
	For $p > 31$ this is a consequence of the Hasse--Weil bound \eqref{eq:hasse-weil}. For the four primes $p \leq 31$, the proof proceeds by enumeration of all binary quartic forms $f(x,z)$ over $\F_p$ with $c_4 \not\equiv 0 \pmod{p}$. If for any $[x:1]$ we have $f(x,1) \in \left(\F_p^\times\right)^3$ or $f(x,1) = 0$ is a root of multiplicity 1, then Hensel's lemma allows us to lift to a $\Q_p$-point. Of course, if $y^3 = f(x,z)$ is insoluble modulo $p$, then there exist no $\Q_p$-points.
	
	The only other possibility is that $f(x,z)$ has one or two double roots. In either case, the probability that such a root lifts to a $\Q_p$-point is $\nu$, by Lemma \ref{lem:nu}. By enumerating all such $f(x,z)$ and determining their value sets and factorizations, we computed $\rho_{3,4}^\aff(p)$ as listed above, finding in particular that $\rho_{3,4}^\aff(31) = 1$. 
	
	This enumeration procedure was implemented in Magma \cite{magma}. The relevant procedure, \texttt{count_quartic_forms(p)}, is contained in the file \texttt{CountForms.m} and can be \href{https://github.com/c-keyes/Density-of-locally-soluble-SECs/blob/f492b080352291c758e10fe9f82a49618e7e095b/CountForms.m}{found here}.
\end{proof}

With a similar approach as that of Lemmas \ref{lem:rho33aff} and \ref{lem:rho34aff}, we can determine $\sigma_1$ and $\sigma_1^*$ exactly for $p = 7, 13, 19, 31, 37, 43$.

\begin{proposition}\label{prop:sigma1_small_p} For the primes $p \equiv 1 \pmod{3}$ with $p \leq 43$, the values of $\sigma_1$ and $\sigma_1^*$ are given below.

\begin{align*}
	\sigma_1(7) &= \frac{577619497568784534247}{586438262710350126300} \approx 0.98496  & \sigma_1^*(7) &= \frac{653206973052553734217}{670215157383257287200} \approx 0.97462 \\
	\sigma_1(13) &= \frac{5931415654903952}{5941011706232655} \approx 0.99838 & \sigma_1^*(13) &= \frac{455813699762383}{457000900479435} \approx 0.99740 \\
	\sigma_1(19) &= \frac{1294027438921}{1294326278072} \approx  0.99976  & \sigma_1^*(19) &= \frac{43009044017}{43024696224} \approx 0.99963\\
	\sigma_1(31) &= \frac{3697903}{3697928} \approx 0.999993 & \sigma_1^*(31) &= \frac{477147}{477152} \approx  0.999989 \\
	\sigma_1(37) &= \frac{937764}{937765} \approx  0.999998 & \sigma_1^*(37) &=  \frac{608279}{608280} \approx 0.999998 \\
	\sigma_1(43) &= \frac{41047793}{41047800} \approx 0.9999998 & \sigma_1^*(43) &= \frac{3818399}{3818400} \approx 0.9999997 	
\end{align*}
\end{proposition}

\begin{proof}

We proceed by enumerating binary sextic forms $f(x,z)$ and checking for liftable points in \texttt{Magma} \cite{magma}; see Appendix \ref{appx:code} for a description, including optimizations necessary to shorten the runtime of these calculations. 

Let $\overline{f}(x,z)$ be a binary sextic form  over $\F_p$ which is not equal to $h(x,z)^3$ for any binary quadratic form $h(x,z)$ (resp.\ also satisfying condition $(*)$). If $\overline{f}(x,z) \in \left(\F_p^\times\right)^3$ or $\overline{f}(x,z) = 0$ is a root of multiplicity 1, then by Hensel's lemma it lifts to a $\Q_p$-point of $C_f$.

If no such $[x:z]$ exist, then the equation is either insoluble, in which case $C_f(\Q_p) = \emptyset$, or the only $\F_p$-points come from multiple roots of $\overline{f}(x,z)$. These could be up to three double roots, or a triple root (note that two triple roots or a sextic root are ruled out by being in factorization case 1). By Lemma \ref{lem:nu}, each double root lifts (independently, by the same arguments as those for $\theta_4$ in the proof of Lemma \ref{lem:thetas}) to a $\Q_p$-point with probability $\nu$, while a triple root lifts with probability $\pi$ by Lemma \ref{lem:pi}.

Summing up the number of forms and weighting by the appropriate probability yields the given values of $\sigma_1$ (resp.\ $\sigma_1^*$). See \eqref{eq:sigma1_13_example} for the case of $p=13$ as an example.
\end{proof}

At this point, we can repeat the calculations of \S \ref{sec:exact_conjugate}, \ref{sec:exact_triple} --- namely those of $\sigma_4$ and $\sigma_5$ --- using the modifications above as appropriate. These modifications are described below; \href{https://github.com/c-keyes/Density-of-locally-soluble-SECs/blob/bd6a8b39ea8c63bf8e7a847063c70998d01ee8aa/SEC_rho36_23Aug21.ipynb}{see here} for an implementation in Sage \cite{sagemath}.
\begin{itemize}
	\item $\theta_{7g} = \rho_{3,4}^{\aff}$ in the proof of Lemma \ref{lem:thetas}. This is used to compute $\theta_7$, which is then used to compute $\mu$, $\tau_7$, and $\mu'$ in succession, and these values are used throughout.
	\item In the proof of Proposition \ref{prop:sigma_4}, the correct $\sigma_1^*$ value from Proposition \ref{prop:sigma1_small_p} must be used in \eqref{eq:rho_star_sum}.
	\item In the calculation of $\tau_9$ in Lemma \ref{lem:taus}, we use
	\begin{align*}
		\tau_{9k} &= \left(1 - \frac1p\right)\rho_{3,4}^\aff(p) + \frac1p \tau_{9\ell} \quad \text{and} \\
		\tau_{9\ell} &= \left(1 - \frac1p\right) \rho_{3,3}^\aff(p) + \frac1p \tau_{9m},
	\end{align*}
	where $\rho_{3,3}^\aff(p)$ and $\rho_{3,4}^\aff(p)$ are given in Lemmas \ref{lem:rho33aff} and \ref{lem:rho34aff}, respectively.
	\item In the calculation of $\theta_9$ in Lemma \ref{lem:thetas}, we use
	\[\theta_{9g} = \left(1 - \frac1p\right) \rho_{3,3}^\aff(p) + \frac1p \theta_{9h}.\]
	\item In the final calculation of $\rho$, we use the correct $\sigma_1$ value from Proposition \ref{prop:sigma1_small_p} in \eqref{eq:rho_sum}.
\end{itemize}
The exact values of $\rho_{3,6}(p)$ are recorded in \eqref{eq:rho7} -- \eqref{eq:rho43}.

\subsubsection{\texorpdfstring{$p=3$}{p=3}}

Suppose now that $p=3$. This case breaks from the others in that when $f(x,z) \not\equiv 0 \pmod{3}$, one cannot determine whether there exists a $\Q_3$-solution to $y^3 = f(x,z)$ from information modulo $p$ alone. Instead, one needs to know information modulo $3^3 = 27$.

In $\Z/27\Z$, the nonzero cubic residue classes are precisely
\[\left(\Z/27\Z^\times\right)^3 = \left\{1, 8, 10, 17, 19, 26\right\}.\]
For $a \in \Z_3$ with $v(a) = 0$, there exists $y \in \Z_3$ satisfying $y^3 = a$ if and only if $a \in \left(\Z/27\Z^\times\right)^3$. This is seen by applying Hensel's lemma, in the form of \eqref{eq:Hensel_improved}, with respect to $y$. Note also that for any $a \in \Z_3$, we have that its residue $a \in \left(\Z/27\Z^\times\right)^3$ if and only if $a + 9 \in \left(\Z/27\Z^\times\right)^3$; this will be used later.

Our approach mirrors that of the other primes $p$ in this section; we first establish some technical results, then use them to adapt our general strategy to work for $p=3$, yielding a value for $\rho_{3,6}(3)$. We begin with the following lemma, which effectively takes the place of $\Phi$ in the proofs of various lifting results.

\begin{lemma}\label{lem:replacing_Phi_for_3}
	Consider the probability of $F(x,y,1)=0$ having a $\Q_3$-solution under the following conditions.
	\begin{enumerate}[label = (\alph*)]
		\item Fix $c_3, c_4, c_5, c_6 \in 3\Z_3$ and vary $c_0 \in \Z_3 - 3\Z_3$ and $c_1, c_2 \in 3\Z_3$. The probability that $F(x,y,1)=0$ has a $\Q_3$-solution is $\frac{19}{27}$.
		\item Fix $c_3, c_4, c_5, c_6 \in 9\Z_3$ and vary $c_0 \in \Z_3 - 3\Z_3$,  $c_1 \in 3\Z_3$, and $c_2 \in 3\Z_3 - 9\Z_3$. The probability that $F(x,y,1)=0$ has a $\Q_3$-solution is $\frac{2}{3}$.
		\item Fix $c_2, c_3, c_4, c_5, c_6 \in 9\Z_3$ and vary $c_0 \in \Z_3 - 3\Z_3$ and $c_1 \in 3\Z_3$. The probability that $F(x,y,1)=0$ has a $\Q_3$-solution is $\frac{7}{9}$.
	\end{enumerate}
\end{lemma}

\begin{proof}
	The restrictions on the $c_i$ guarantee that for any $(x,y) \in \Z_3^2$, we have $v\left(F(x,y,1)\right) = 0$. Thus it suffices to work modulo $3^3 = 27$ and determine if $f(x,1)$ takes a value in $\left(\Z/27\Z^\times\right)^3$. By our earlier observation that nonzero cubic residues modulo 27 are invariant under addition by multiples of 9, we have that $f(x,1) \in \left(\Z/27\Z^\times\right)^3$ if and only if $f(x+3, 1) \in \left(\Z/27\Z^\times\right)^3$, and hence it suffices to check at
	\begin{align*}
		f(0,1) &= c_0,\\
		f(1,1) &= \sum_{i=0}^6 c_i, \text{and}\\
		f(-1,1) &= \sum_{i=0}^6 (-1)^i c_i.
	\end{align*}
We have that $c_0 \in \left(\Z/27\Z^\times\right)^3$ with probability $\frac{1}{3}$. If not, then we may check at the other values.

Consider first (a). If $c_0$ is not in $\left(\Z/27\Z^\times\right)^3$ then let $c \equiv c_0 + c_3 + c_4 + c_5 + c_6 \pmod{27}$. Exactly one of $c$, $c + 3$, and $c - 3$ are in $\left(\Z/27\Z^\times\right)^3$, and as $c_1, c_2$ varying in $3\Z_3$, we have that $c_1 + c_2 \equiv 0, 3, -3 \pmod{9}$ each with equal probability of 1/3. If $\sum c_i \notin \left(\Z/27\Z^\times\right)^3$, then we verify by direct enumeration that the probability of $-c_1 + c_2$ satisfying $\sum (-1)^i c_i \in \left(\Z/27\Z^\times\right)^3$ is also 1/3. Hence we have the probability in (a) is given by
\[\frac{1}{3} + \frac{2}{3}\left(\frac{1}{3} + \frac{2}{3}\left(\frac{1}{3}\right)\right) = \frac{19}{27}.\]

For (b) and (c), we are in a similar situation, except we can ignore $c_3, \ldots, c_6$ entirely as their values will not affect whether $f(x,1)$ takes a value in $\left(\Z/27\Z^\times\right)^3$. To compute (b), we note that if $c_0$ is not in $\left(\Z/27\Z^\times\right)^3$, then $c_0 + c_2$ is with probability $\frac{1}{2}$, since $9 \nmid c_2$. If this is the case, there is a $\frac{1}{3}$ chance that $9 \mid c_1$ and we have $f(1,1) \in \left(\Z/27\Z^\times\right)^3$. If $c_0 + c_2 \notin \left(\Z/27\Z^\times\right)^3$, then there is a $\frac{2}{3}$ chance that one of $f(\pm 1, 1) \in \left(\Z/27\Z^\times\right)^3$. This comes out to the probability
\[\frac{1}{3} + \frac{2}{3} \left( \frac{1}{2} \right) = \frac{2}{3}.\]

For (c) we follow a similar approach, observing that if $c_0$ is not in $\left(\Z/27\Z^\times\right)^3$, we have a $\frac{2}{3}$ chance that one of $c_0 \pm c_1$ is, as only when $9 \mid c_1$ is the sum not a cube modulo 27. Thus we obtain
\[\frac{1}{3} + \frac{2}{3} \left(\frac{2}{3} \right) = \frac{7}{9}.\]\end{proof}

We now compute the probability of lifting a point $[x:0:z]$ on $y^3 = f(x,z)$ when $f(x,z)$ has a double root modulo $3$; after a change of coordinates, we may consider the point $[0 : 0 : 1]$. We call this probability $\nu$ as in Lemma \ref{lem:nu}.

\begin{lemma}\label{lem:nu_p=3}
	Fix $c_2, c_3, c_4, c_5, c_6 \in \Z_3$ and suppose $v(c_2)=0$. As $c_0, c_1$ vary in $3\Z_3$, the $\F_3$-solution $[0:0:1]$ to $\overline{F}(x,y,z) = 0$ lifts to a $\Q_3$-solution to $F(x,y,z) = 0$ with probability
	\[\nu = \frac{43}{243}.\]
\end{lemma}

\begin{proof}
	Following the proof of Lemma \ref{lem:nu}, we have
	\[\nu = \frac1p \left(\eta_{2,1}' + \eta_{2,2}'\theta_2\right).\]
	To compute $\theta_2$, we follow the proof of Lemma \ref{lem:thetas}, except that in the first step we take
	\[\theta_{2a} = \frac{2}{3}\cdot \left(\frac{2}{3} \right) + \frac{1}{3}\theta_{2b},\]
	justifying as follows. With probability $\frac{2}{3}$ we have $v(c_0) = 0$, putting us in the case of Lemma \ref{lem:replacing_Phi_for_3}(c), in which case a lift exists with probability $\frac{2}{3}$, giving the left-hand term. With probability $\frac{1}{3}$ we have $3 \mid c_0$, and we continue with the computation of $\theta_2$ as in the proof of Lemma \ref{lem:thetas}.
	
	Note that in the last step, we may take $\theta_{2d}=1$ as usual, since the partial derivative of the quadratic $c_2x^2 + c_1x + c_0$ can only vanish modulo 3 for at most one value of $x$. Hence one of the other $x$ values may always be used to lift to a $\Q_3$-solution. This results in 
	\[\theta_2 = \frac{16}{27} \hspace{1cm} \text{and} \hspace{1cm} \nu = \frac{43}{243}.\]
\end{proof}

Next, we consider $\sigma_1$. The values of $\rho_{3,3}^\aff$ and $\rho_{3,4}^\aff$ will follow from similar reasoning. Of the 2160 (see Lemma \ref{lem:factorization_counts}) binary sextic forms $\overline{f}(x,z)$ modulo 3 with $\overline{F}$ absolutely irreducible, all but 54 have at least one $[x:z]$ such that the partial derivative of $\overline{F}$ with respect to $x$ (or $z$) is nonvanishing modulo 3, and hence liftable via Hensel's lemma (see Proposition \ref{prop:bound_p|m}). The remaining 54 may be enumerated and are seen to have the factorization types as follows:
\begin{itemize}
	\item 24 have one double root (i.e.\ $\overline{f}(x,z)$ has factorization type $1^24$ or $1^222$) modulo 3,
	\item 12 have two double roots (i.e.\ $\overline{f}$ has type $1^21^22$) modulo 3,
	\item 8 have three double roots (i.e.\ $\overline{f}$ has type $1^21^21^2$) modulo 3, and
	\item 10 have no roots modulo 3.
\end{itemize}
This leads to our determination of $\sigma_1$, as well as $\rho_{3,3}^\aff$ and $\rho_{3,4}^\aff$.

\begin{proposition}\label{prop:sigma_1_p=3}
	When $p=3$, we have
	\begin{align*}
		\sigma_1 &= \frac{5780143846}{5811307335} \approx 0.99463,\\
		\rho_{3,3}^\aff &= \frac{2103}{2183} \approx 0.96335 ,\\
		\rho_{3,4}^\aff &= \frac{4585681}{4782969} \approx 0.95875.
	\end{align*}
\end{proposition}

\begin{proof}
	To compute $\sigma_1$, we need only determine the probability of lifting for each of the four values $[x:z]$, and see that they are independent of one another. When $f(x,z) \neq 0 \pmod{3}$, the probability of lifting is $\frac{1}{3}$, exactly the proportion of residues in $\Z/27\Z^\times$ in the image of the cube map. The probability of a double root modulo $p$ lifting is $\nu$, as after a change of coordinates we may assume $[x:z] = [0:1]$, putting us in the case of Lemma \ref{lem:nu_p=3}.
	
	For the independence, we treat each of the four cases separately, illustrating the argument in the case that $\overline{f}(x,z)$ has exactly one double root modulo 3, occurring at $[0:1]$. By Lemma \ref{lem:nu_p=3} the probability of lifting is $\nu$, and since this proof relies on Lemma \ref{lem:replacing_Phi_for_3}(c), the lifting behavior depends only on the values of $c_0$ and $c_1$. The lifting behavior at $[x:z] = [1:1], [-1:1], [1:0]$, for which $f(x,z) \not\equiv 0 \pmod{3}$, does not depend on the choice of lift of $[x:z]$ to $\Z/27\Z$ by our earlier discussions. Hence this depends only on, say, $c_3$, $c_4$, and $c_6$.
	
	Since lifting a double root depends only on $f(x,z)$ and $f'(x,z)$ modulo 27, lifting a nonzero value depends only on the value itself in this case, and we have 7 coefficients varying, the argument above easily extends to the other three cases. This justifies	
	\begin{align*}
		\sigma_1 = \frac{1}{2160}\left(2106 \right. &+ 10\left(1 - \left(\frac{2}{3}\right)^4\right) + 24 \nu \left(1 - \left(\frac{2}{3}\right)^3\right)  \\
		 & + \left. 12 \left(1 - (1-\nu)^2\right)\left(1 - \left(\frac{2}{3}\right)^2\right) + 8 \left(1 - (1-\nu)^3\right)\left(\frac{1}{3}\right) \right).
	\end{align*}
	
	For $\rho_{3,4}^\aff$, the story is similar, except we are only interested in lifing affine solutions $[x:1]$. Of the 162 quartics $c_4 x^4 + c_3x^3 + c_2x^2 + c_1x + c_0$ modulo 3, all but 18 have nonzero partial derivative, with the cases of no (affine) roots, one (affine) double root, and two (affine) double roots each appearing 6 times. The same lifting probabilities and independence arguments above apply, yielding the stated value.
	
	Finally, to compute $\rho_{3,3}^\aff$, we note that so long as $c_1, c_2$ are not both in $3\Z_3$, the partial derivative does not vanish and we have a root by Hensel's lemma. If $3 \mid c_1, c_2$, then $\overline{f}(x,z) = c_3x^3z^3 + c_0z^6$, and after a change of variables over $\F_3$, we may assume $\overline{f}(x,z) = x^3z^3$. 
	
	Let $f(x,z)\in \Z_3[x,z]$ be a binary sextic form reducing to $\overline{f} = x^3z^3$. Using the same techniques as in the proof of Lemma \ref{lem:replacing_Phi_for_3}, we have that the probability that $[\pm 1: 1]$ has a lift to a $\Q_3$-point is $\frac{1}{3}$ for each. Note that this can be made independent of the choices of $c_4$, $c_5$ and $c_6$, which is necessary to use in our calculations of $\theta_{9g}$ and $\tau_{9\ell}$. If neither of these lift, the probability that $[0:1]$ is computed by first replacing $x$ by $3x$, giving the valuations on $c_3, c_2, c_1, c_0$
	\[ \cdots \ =3 \ \geq 3 \ \geq 2 \ \geq 1.\]
	We compute that we have a lift with probability $\frac{2}{9} + \frac{1}{27}\rho_{3,3}^\aff$, yielding that
	\[\rho_{3,3}^\aff = \frac{8}{9} + \frac{1}{9}\left(\frac{5}{9} + \frac{4}{9}\left(\frac{2}{9} + \frac{1}{27}\rho_{3,3}^\aff\right)\right),\]
	and solving for $\rho_{3,3}^\aff$ gives the stated value.
\end{proof}

Taken together, we once again repeat the calculations of \S\ref{sec:exact_triple} to obtain $\sigma_5$ and $\rho$, using the same modifications we did previously with the primes $p \equiv 1 \pmod{3}$  up to $p=43$. We also must replace $\Phi$ by the appropriate probability in Lemma \ref{lem:replacing_Phi_for_3} wherever necessary in the proofs of Lemmas \ref{lem:taus} and \ref{lem:thetas}. Finally solving
\[\rho_{3,6}(3) = \frac{2160}{2187}\sigma_1 + \frac{1}{27}\sigma_5\]
yields
\begin{equation}\label{eq:rho3}
	\rho_{3,6}(3) = \frac{900175334869743731875930997281}{908381960435133191895132960000} \approx 0.99096.
\end{equation}
Once again, \href{https://github.com/c-keyes/Density-of-locally-soluble-SECs/blob/bd6a8b39ea8c63bf8e7a847063c70998d01ee8aa/SEC_rho36_23Aug21.ipynb}{see here} for an implementation of these calculations in Sage \cite{sagemath}.

\subsubsection{Calculating \texorpdfstring{$\rho_{3,6}$}{rho_36} exactly}

We are now ready to complete the proof of Theorem \ref{thm:exact_rho_3_6}.

\begin{proof}[Proof of Theorem \ref{thm:exact_rho_3_6}]
	We have already seen that $\rho(p) = R_i(p)$ for sufficiently large $p \equiv i \pmod{3}$; see \S \ref{sec:exact_proof_rho}. For the remaining primes, $p=2,3,7,13,19,31,37,43$, we have computed $\rho(p)$ in the preceding sections; see \eqref{eq:p=2_rho_exact}, \eqref{eq:rho7} -- \eqref{eq:rho43}, and \eqref{eq:rho3}. This yields the exact expression
	\[\rho_{3,6} = \prod_{p = 2,3,7,13,19,31,37,43} \rho(p) \prod_{\substack{p \equiv 1 \pmod{3}\\p > 43}} R_1(p) \prod_{\substack{p \equiv 2 \pmod{3}\\p > 2}} R_2(p).\]
	To obtain a numerical value, we compute the product of $\rho(p)$ for all $p \leq 10000$, finding $\prod_{p \leq 10000} \rho(p) \approx 0.96943$. Using the estimates $R_1(t) \geq 1-t^{-4}$ and $R_2(t) \geq 1 - t^{-7}$, we find 
	\[1 \geq \prod_{\substack{p \equiv 1 \pmod{3}\\p > 10000}} R_1(p) \prod_{\substack{p \equiv 2 \pmod{3}\\p > 10000}} R_2(p) \geq 1 - 1.6856\cdot 10^{-14},\]
	which more than suffices to conclude our numerical value is correct to several decimal places. These calculations are recorded in this \href{https://github.com/c-keyes/Density-of-locally-soluble-SECs/blob/bd6a8b39ea8c63bf8e7a847063c70998d01ee8aa/SEC_rho36_23Aug21.ipynb}{Sage notebook} \cite{sagemath}.
\end{proof}

\appendix

\vfill
 \section{Bounds for \texorpdfstring{$\rho_{m,d}(p)$}{rho_md(p)} via computer search}
\label{appx:code}

The lower bounds for $\rho_{3,6}$ and $\rho_{5,5}$ produced by Corollary \ref{cor:lower_bound}, discussed in Examples \ref{ex:m=3} and \ref{ex:m=5}, were limited by the performance of Proposition \ref{prop:bound_matrix} for primes $p \equiv 1 \pmod{m}$ such that $p < 4g^2$. As noted in Remarks \ref{rem:crudeness} and \ref{rem:no_simple_roots}, the proof of Proposition \ref{prop:bound_matrix} likely leaves out many liftable points, including those given by roots of $f(x,z)$ of multiplicity 1. Here we discuss how to use a computer search to improve our lower bounds of $\rho_{m,d}(p)$, an implementation in the case of $\rho_{3,6}(13)$, and how this approach is used in the exact determination of $\sigma_1(p)$ for small primes $p$ in \S \ref{sec:small_primes}.

Suppose $p \nmid m$. Using a computer algebra system it is straightforward to enumerate all binary degree $d$ forms $\overline{f}(x,z)$ over $\F_p$ and for each such $f$, determine whether 
\begin{itemize}
	\item there exists a root $\overline{f}(x_0, z_0) =0$ of multiplicity 1, or
	\item there exists $[x_0 : z_0]$ such that $\overline{f}(x_0, z_0) \in \left(\F_p^\times\right)^m$.
\end{itemize}
In either case, for any $f(x,z) \in \Z_p[x,z]$ such that $f \equiv \overline{f} \pmod{p}$, Hensel's lemma (Theorem \ref{thm:hensel}) ensures $C_f(\Q_p) \neq \emptyset$.

Na\"{i}vely, this amounts to enumerating $p^{d+1}$ polynomials, which quickly becomes prohibitively time consuming. To mitigate this, we first recognize that $C_f$ has a smooth point if and only if $C_{u^mf}$ does for $u \in \F_p^\times$. This corresponds to the change of variables $y \mapsto \frac{y}{u}$. Thus we may assume that the leading coefficient, $c_d$, of $f(x,z)$ is either 0 or equal to one of the representatives of the three cosets in $\F_p^\times / \left( \F_p^\times \right)^m$. This cuts down the running time by about a factor of $p$, as we need only enumerate the coefficients $c_0, \ldots, c_{d-1}$.

To further improve the running time by a factor of $p$, we ran these searches for fixed values of the constant term $c_0$ in parallel. To avoid having to run $p$ such programs, which again becomes cumbersome for large $p$, we observe that for a generator $a \in \F_p^\times$, the change of variables $z \mapsto az$ transforms the constant term by a factor of $a^d$, without affecting the leading term. Thus we may assume $c_0$ to be either 0 or equal to one of the representatives of the $\gcd(d, p-1)$ cosets in $\F_p^\times / \left(\F_p^\times\right)^d$. In particular, $\gcd(d, p-1)$ is bounded by $d$, so the number of parallel computations needed is bounded as $p$ grows.

We implemented the strategy above in Magma \cite{magma} for $(m,d) = (3,6)$ to obtain better bounds for $\rho_{3,6}(p)$ for the seven primes $p$ such that $p \equiv 1 \pmod{3}$ and $p \leq 61$. \href{https://github.com/c-keyes/Density-of-locally-soluble-SECs/blob/f492b080352291c758e10fe9f82a49618e7e095b/CountForms.m}{See here} for the code, which consists of a file \texttt{CountForms.m} containing the necessary routines; namely, the procedure \texttt{count_sextic_forms($p, c_0$)} counts binary sextic forms $\overline{f}(x,z)$ which, after the aforementioned changes of variables, have specified coefficient $c_0$ and such that $y^3 = \overline{f}(x,z)$ has a smooth point. It is convenient to also keep track of whether or not the forms satisfy condition $(*)$, whether $y^3 - \overline{f}(x,z)$ is absolutely irreducible, or both, to give lower bounds for $\rho^*$, $\sigma_1$, and $\sigma_1^*$.

To illustrate this procedure, the output in the $p = 13$ case is tabulated below in Table \ref{tab:p=13_data}. Notice the symmetry present in the table; the $c_0$ and $-c_0$ rows are identical. This is the result of the facts that $\langle 2 \rangle = \F_{13}^\times$, and $2^6 \equiv -1 \pmod{13}$, and our observations above about the change of variables $z \mapsto 2z$. Considering only the presence of Hensel-liftable points and insoluble equations, these computations produce the following bounds.
\begin{align*}
  0.99851 \approx \frac{62655132}{62748517}  &\leq \rho(13) \ \leq \frac{4819929}{4826809} \approx 0.99857\\
  0.99735 \approx \frac{740621}{742586} &\leq \rho^*(13)  \leq \frac{370433}{371293} \approx 0.99768\\
  0.99837 \approx \frac{8605}{8619}  &\leq \sigma_1(13)  \leq \frac{43034}{43095} \approx 0.99858 \\
  0.99738 \approx \frac{105803}{106080}  &\leq \sigma_1^*(13) \leq \frac{26459}{26520} \approx 0.99769
\end{align*}

\begin{center}
\begin{table}[h]
\caption{Counts of binary sextic forms $f(x,z) \in \F_{13}[x,z]$ with smooth points for specified constant coefficient, using \texttt{count_sextic_forms($13, c_0$)}}
\label{tab:p=13_data}
\begin{tabular}{|c| c c c |}
	\hline $c_0$ & Hensel & insoluble & total  \\ \hline
	0  & 4825604 & 0     & 4826809 \\ 
	1  & 4814593 & 10608 & 4826809 \\
	2  & 4820270 & 5680  & 4826809 \\
	3  & 4820634 & 5364  & 4826809 \\
	4  & 4820619 & 5364  & 4826809 \\
	5  & 4813393 & 12024 & 4826809 \\
	6  & 4820255 & 5680  & 4826809 \\
	7  & 4820255 & 5680  & 4826809 \\
	8  & 4813393 & 12024 & 4826809 \\
	9  & 4820619 & 5364  & 4826809 \\
	10 & 4820634 & 5364  & 4826809 \\
	11 & 4820270 & 5680  & 4826809 \\
	12 & 4814593 & 10608 & 4826809 \\ \hline
	Totals & 62645132 & 89440 & 62748517 \\ \hline
\end{tabular}
\end{table}
\end{center}

To verify the data in Table \ref{tab:rho36_computer}, we repeat these computations considering only the Hensel--liftable and insoluble equations for the seven primes $p \equiv 1 \pmod{3}$ with $p \leq 61$. The resulting lower bounds are recorded in Table \ref{tab:rho36_magma_data} below along with the approximate runtime of an instance of \texttt{count_sextic_forms($p, c_0$)} on a server with four Intel Xeon E5-4627 CPUs, a total of 40 cores, and 1 TB of memory. As one expected, the complexity is about $O(p^5)$. Note in particular that for $p=61$, the computation reflects the improved Hasse--Weil bound \eqref{eq:hasse-weil_improvement}, which implied $\sigma_1 = \sigma_1^* = 1$ for $p=61$ in Proposition \ref{prop:sigma_1}.

\begin{center} \small
\begin{table}[h]
\label{tab:rho36_magma_data}
\caption{Lower bounds for $\rho$, $\rho^*$, $\sigma_1$, $\sigma_1^*$ for $p\equiv 1 \pmod{3}$ with $p \leq 61$}
\renewcommand{\arraystretch}{1.25}
\begin{tabular}{|c| c | c | c | c | c |}
	\hline $p$ & $\rho \geq $ & $\rho^* \geq$ & $\sigma_1 \geq$ & $\sigma_1^* \geq$ & runtime (s) \\ \hline
	7 & $\frac{810658}{823543} \approx 0.98435$ & $\frac{32731}{33614} \approx 0.97373$ & $\frac{7237}{7350} \approx 0.98463$ & $\frac{32731}{33600} \approx 0.97414$ & 18 \\ 
	13 & $\frac{62645132}{62748517} \approx 0.99835$ & $\frac{740621}{742586} \approx 0.99735$ & $\frac{8605}{8619} \approx 0.99837$ & $\frac{105803}{106080} \approx 0.99738$ & 419 \\ 
	19 & $\frac{893660256}{893871739} \approx 0.99976$ & $\frac{2475177}{2476099} \approx 0.99962$ & $\frac{522607}{522728} \approx 0.99976$ & $\frac{825059}{825360} \approx 0.99963$ & 2961 \\ 
	31 & $\frac{27512408250}{27512614111} \approx 0.99999$ & $\frac{28628820}{28629151} \approx 0.99998$ & $\frac{3697903}{3697928} \approx 0.99999$ & $\frac{477147}{477152} \approx 0.99998$ & 37161  \\ 
	37 & $\frac{94931742132}{94931877133} \approx 0.999998$ & $\frac{69343806}{69343957} \approx 0.999997$ & $\frac{937764}{937765}\approx 0.999998$ & $\frac{608279}{608280}\approx 0.999998$ & 90131 \\ 
	43 & $\frac{271818511748}{271818611107} \approx{0.9999996}$ & $\frac{294016723}{294016886}\approx 0.9999994$ & $\frac{41047793}{41047800} \approx 0.9999998$ & $\frac{3818399}{3818400} \approx 0.9999997$ & 194243\\ 
	61 & $\frac{3142742684700}{3142742836021} \approx 0.99999995$ & $\frac{13845840}{13845841} \approx 0.99999992$ & 1 & 1 & 1091730 \\ \hline
\end{tabular}
\end{table}
\end{center}

For $(m,d) = (5,5)$, a similar procedure using the above mentioned parallelization strategy was used to produce the data for $p \equiv 1 \pmod{5}$ with $p \leq 131$, tabulated in Table \ref{tab:rho55_computer} in Example \ref{ex:m=5}. For other $(m,d)$, this enumeration strategy could be useful in estimating $\rho_{m,d}(p)$ and offers an improvement on Proposition \ref{prop:bound_matrix}, at the cost of the time required.

	This enumeration strategy was also instrumental in determining the exact values of $\sigma_1$ and $\sigma_1^*$ for small primes $p$, i.e.\ the proof of Proposition \ref{prop:sigma1_small_p}. After enumerating all binary sextic forms $\overline{f}(x,z)$ with Hensel-liftable points, we keep track of the factorization type modulo $p$ --- namely the presence of multiple roots --- and determine the lifting probabilities using $\nu$ and $\pi$ (see Lemmas \ref{lem:nu}, \ref{lem:pi}).
	
	For example, when determining $\sigma_1(13)$, we find that 62644400 of the 62746320 forms $\overline{f}(x,z)$ produce $F(x,y,z)$ with a Hensel-liftable point and 88816 are insoluble. Of the remaining $\overline{f}(x,z)$, we find 10920 that have one double root, 2184 having two double roots, and no other factorization types occur. Thus
	\begin{equation}\label{eq:sigma1_13_example}
		\sigma_1(13) = \frac{1}{62746320} \left( 62644400 + 10920 \nu + 2184 (1 - (1-\nu)^2) \right) = \frac{5931415654903952}{5941011706232655},
	\end{equation}
	the value given in Proposition \ref{prop:sigma1_small_p}. This calculation is repeated for the primes $p=7, 19, 31, 37, 43$ and a similar philosophy is used for $p=3$ in Proposition \ref{prop:sigma_1_p=3}.

\section{Counting binary forms by factorization type}\label{appx:counting_forms}

\begin{lemma}\label{lem:form_counts} For $2 \leq d \leq 6$ let $N_{d,i}$ (resp.\ $N_{d,i}'$) denote the number of binary forms $f(x,z)$ over $\F_p$ up to scaling (resp.\ monic) having the factorization types specified by $i$ in the second column of the table below. For $N_{d,i}$ and $N_{d,i}'$ can be computed in terms of $p$ and are tabulated below.
	\begin{center} \small \renewcommand{\arraystretch}{1.2}
	\begin{tabular}{| c | l | c | c |}
		\hline $\boldsymbol{d}$ & \textbf{Fact.\ type} & $\boldsymbol{N_{d,i}}$ & $\boldsymbol{N_{d,i}'}$ \\ \hline 
		\multirow{3}*{$2$} & 0. No roots & $\frac{1}{2}  {\left(p - 1\right)} p$ & $\frac{1}{2}  {\left(p - 1\right)} p$ \\
		 & 1. $(1*)$ & $\frac{1}{2}  {\left(p + 1\right)} p$ & $\frac{1}{2}  {\left(p - 1\right)} p$ \\
		 & 2. $(1^2)$ & $p + 1$ & p\\ \hline 
		 
		 \multirow{3}*{$3$} & 0. No roots & $ \frac{1}{3}  {\left(p + 1\right)} {\left(p - 1\right)} p $ & $ \frac{1}{3}  {\left(p + 1\right)} {\left(p - 1\right)} p $ \\
		 & 1. $(1*)$ & $ \frac{1}{3}  {\left(2  p + 1\right)} {\left(p + 1\right)} p $ & $ \frac{2}{3}  {\left(p + 1\right)} {\left(p - 1\right)} p $ \\
		 & 2. $(1^3)$ & $ p + 1 $ & $ p $ \\ \hline
		 
		 \multirow{5}*{$4$} & 0. No roots & $ \frac{1}{8}  {\left(3  p^{2} + p + 2\right)} {\left(p - 1\right)} p $ & $ \frac{1}{8}  {\left(3  p^{2} + p + 2\right)} {\left(p - 1\right)} p $ \\
		 & 1. $(1*)$ & $ \frac{1}{8}  {\left(5  p^{2} + p + 2\right)} {\left(p + 1\right)} p $ & $ \frac{1}{8}  {\left(5  p^{2} + 3  p + 2\right)} {\left(p - 1\right)} p $ \\
		 & 2. $(1^2 2)$ & $ \frac{1}{2}  {\left(p + 1\right)} {\left(p - 1\right)} p $ & $ \frac{1}{2}  {\left(p - 1\right)} p^{2} $ \\
		 & 3. $(1^2 1^2)$ & $ \frac{1}{2}  {\left(p + 1\right)} p $ & $ \frac{1}{2}  {\left(p - 1\right)} p $ \\
		 & 4. $(1^4)$ & $ p + 1 $ & $ p $ \\ \hline
		 
		  \multirow{6}*{$5$} & 0. No roots & $ \frac{1}{30}  {\left(11  p^{2} - 5  p + 6\right)} {\left(p + 1\right)} {\left(p - 1\right)} p $ & $ \frac{1}{30}  {\left(11  p^{2} - 5  p + 6\right)} {\left(p + 1\right)} {\left(p - 1\right)} p $ \\
		 & 1. $(1*)$ & $ \frac{1}{30}  {\left(19  p^{3} + 6  p^{2} + 4  p + 1\right)} {\left(p + 1\right)} p $ & $ \frac{1}{30}  {\left(19  p^{3} + 14  p^{2} + 4  p - 6\right)} {\left(p - 1\right)} p $ \\
		 & 2. $(1^2 3)$ & $ \frac{1}{3}  {\left(p + 1\right)}^{2} {\left(p - 1\right)} p $ & $ \frac{1}{3}  {\left(p + 1\right)} {\left(p - 1\right)} p^{2} $ \\
		 & 3. $(1^3 2)$ & $ \frac{1}{2}  {\left(p + 1\right)} {\left(p - 1\right)} p $ & $  \frac{1}{2}  {\left(p - 1\right)} p^{2} $ \\
		 & 4. $(1^2 1^3)$ & $ {\left(p + 1\right)} p $ & $ {\left(p - 1\right)} p $ \\
		 & 5. $(1^5)$ & $ p+1 $ & $ p $ \\ \hline
		 
		 \multirow{10}*{$6$} & 0. No roots & $ \frac{1}{144}  {\left(53  p^{4} + 26  p^{3} + 19  p^{2} - 2  p + 24\right)} {\left(p - 1\right)} p $ & $ \frac{1}{144}  {\left(53  p^{4} + 26  p^{3} + 19  p^{2} - 2  p + 24\right)} {\left(p - 1\right)} p $ \\
		 & 1. $(1*)$ & $ \frac{1}{144}  {\left(91  p^{4} + 26  p^{3} + 23  p^{2} + 16  p - 12\right)} {\left(p + 1\right)} p $ & $ \frac{1}{144}  {\left(91  p^{3} - 27  p^{2} + 50  p - 48\right)} {\left(p + 1\right)} {\left(p - 1\right)} p $ \\
		 & 2. $(1^2 4)$ or $(1^2 2 2)$ & $ \frac{1}{8}  {\left(3  p^{2} + p + 2\right)} {\left(p + 1\right)} {\left(p - 1\right)} p $ & $ \frac{1}{8}  {\left(3  p^{2} + p + 2\right)} {\left(p - 1\right)} p^{2} $ \\
		 & 3. $(1^2 1^2 2)$ & $ \frac{1}{4}  {\left(p + 1\right)} {\left(p - 1\right)} p^{2} $ & $ \frac{1}{4}  {\left(p - 1\right)}^{2} p^{2} $ \\
		 & 4. $(1^2 1^2 1^2)$ & $ \frac{1}{6}  {\left(p + 1\right)} {\left(p - 1\right)} p $ & $ \frac{1}{6}  {\left(p - 1\right)} {\left(p - 2\right)} p $ \\
		 & 5. $(1^3 3)$ & $ \frac{1}{3}  {\left(p + 1\right)}^{2} {\left(p - 1\right)} p $ & $ \frac{1}{3}  {\left(p + 1\right)} {\left(p - 1\right)} p^{2} $ \\
		 & 6. $(1^31^3)$ & $ \frac{1}{2}  {\left(p + 1\right)} p $ & $ \frac{1}{2}  {\left(p - 1\right)} p $ \\
		 & 7. $(1^4 2)$ & $ \frac{1}{2}  {\left(p + 1\right)} {\left(p - 1\right)} p $ & $ \frac{1}{2}  {\left(p - 1\right)} p^{2} $ \\
		 & 8. $(1^2 1^4)$ & $ {\left(p + 1\right)} p $ & $ {\left(p - 1\right)} p $ \\
		 & 9. $(1^6)$ & $ p+1 $ & $ p $ \\ \hline
	\end{tabular}
	\end{center}
\end{lemma}

\begin{proof}
	This is an elementary computation, as noted in \cite[Lemma 2.3]{BCF_genus_one}, in which each subsequent row is obtained from the previous one. We give a proof for $d=6$ here, assuming the results in the previous rows of the table. To obtain the result for $d < 6$ --- or indeed any $d$ value if one is patient --- one can use the same idea.
	
	Let $f(x,z)$ be a degree $6$ binary form over $\F_p$, up to scaling. We first consider all cases in which $f$ has a multiple root but no simple root, which are precisely Types 2 -- 9 above. Several of these can be calculated via combinatorics alone, beginning with Case 9, where $f$ has a sextuple root, or factorization type $(1^6)$. There are exactly $p+1$ such roots, corresponding to the $p+1$ distinct linear factors up to scaling, so we have $p+1$ forms up to scaling. Case 8 is similar: to give a form of type $(1^2 1^4)$ up to scaling, it suffices to identify  a distinct linear factor for each root. Since the multiplicities are different, order matters, giving $(p+1)p$ possibilities. Types 4 and 6 ($(1^21^21^2)$ and $(1^31^3)$ respectively) are similar, but order does not matter, so  there are $\binom{p+1}{3}$ and $\binom{p+1}{2}$ possibilities, respectively.
	
	To deal with Type 7, $(1^4 2)$, we use the $d=2$ row in the table to determine how many binary quadratic forms there are up to scaling, and multiply this by $p+1$, the number of possible $1^4$ factors. Similar arguments work for $(1^3 3)$ and $(1^2 4)$, using the result for degrees 3 and 4.
	
	The case of $(1*)$ is the most involved, so we break the calculation down into cases based on the number of distinct simple roots $f$ has, i.e.\ the number of 1's appearing in its factorization type. If $f$ has 6 distinct simple roots, there are $\binom{p+1}{6}$ possibilities for $f$. It is not possible to have exactly 5 simple roots, so we move to the case of 4 distinct simple roots times a quadratic which has no double roots. There are $\binom{p+1}{4} N_{2,0}$ possibilities when the quadratic is irreducible and $\binom{p}{4} N_{2,2}$ when the quadratic has a double root. Note there is one fewer linear factor to choose the 4 simple roots from, since they must avoid the double root of the quadratic factor.
	
	Continuing along this line, we find
	\begin{align*}
		N_{6,1} = \binom{p+1}{6} &+ \binom{p+1}{4} N_{2,0} + \binom{p}{4} N_{2,2} + \binom{p+1}{3} N_{3,0} + \binom{p}{3} N_{3,2}  \\
		&+ \binom{p+1}{2} N_{4,0} + \binom{p}{2} (N_{4,2} + N_{4,4}) + \binom{p-1}{2} N_{4,3} \\
		&+ (p+1)N_{5,0} + p(N_{5,2} + N_{5,3} + N_{5,5})+ (p-1)N_{5,4},\end{align*}
	which may be computed via computer algebra software; \href{https://github.com/c-keyes/Density-of-locally-soluble-SECs/blob/bd6a8b39ea8c63bf8e7a847063c70998d01ee8aa/SEC_rho36_23Aug21.ipynb}{see here} for an implementation in Sage \cite{sagemath}. To conclude, we recognize that Types 1 -- 9 are precisely those $f$ possessing a root. Therefore we have
	\[N_{6,0} = \left(\sum_{i=0}^6 p^i \right) - \left(\sum_{j=1}^d N_{6,j}\right).\]
	
	The same strategies work for computing each $N_{6,i}'$ for the monic case. The only differences are that there are $p$ choices of linear factors, rather than $p+1$, due to the monicity assumption, and that the appropriate monic quantities $N_{d,i}'$ are used in place of $N_{d,i}$ throughout.	
\end{proof}

\begin{proof}[Proof of Lemma \ref{lem:etas}]
Let $\eta_{d,i}$ (resp.\ $\eta_{d,i}'$) denote the proportion of binary degree $d$ forms $f(x,z)$ over $\F_p$ up to scaling by $\F_p^\times$ (resp. $f(x,z)$ is monic) having a factorization type corresponding to $i$ in Lemma \ref{lem:form_counts}. By our earlier observations,
\[\eta_{d,i} = \frac{N_{d,i}}{\sum_{j=0}^d p^j}, \hspace{2cm} \eta_{d,i}' = \frac{N_{d,i}'}{p^d}. \]
These values are precisely those in Lemma \ref{lem:etas}.
\end{proof}

\begin{remark}
	There is no serious obstacle to extending Lemma \ref{lem:form_counts} and thus Lemma \ref{lem:etas} to higher degrees. In fact, one would likely have to do so in order to compute exact formulas for  $\rho_{m,d}$ when $d > 6$. 
\end{remark}

\vfill
\section{Explicit formulas for rational functions}
\label{appx:explicit_formulas}

\begin{align}
\label{eq:rho}
	\rho &= \begin{cases}
		\begin{minipage}{11cm}\tiny
			\begin{dmath*}  
				\left( 1296  p^{57} + 3888  p^{56} + 9072  p^{55} + 16848  p^{54} + 27648  p^{53} + 39744  p^{52} + 53136  p^{51} + 66483  p^{50} + 80019  p^{49} + 93141  p^{48} + 107469  p^{47} + 120357  p^{46} + 135567  p^{45} + 148347  p^{44} + 162918  p^{43} + 176004  p^{42} + 190278  p^{41} + 203459  p^{40} + 218272  p^{39} + 232083  p^{38} + 243639  p^{37} + 255267  p^{36} + 261719  p^{35} + 264925  p^{34} + 265302  p^{33} + 261540  p^{32} + 254790  p^{31} + 250736  p^{30} + 241384  p^{29} + 226503  p^{28} + 214137  p^{27} + 195273  p^{26} + 170793  p^{25} + 151839  p^{24} + 136215  p^{23} + 118998  p^{22} + 105228  p^{21} + 94860  p^{20} + 80471  p^{19} + 67048  p^{18} + 52623  p^{17} + 40617  p^{16} + 28773  p^{15} + 19247  p^{14} + 12109  p^{13} + 7614  p^{12} + 3420  p^{11} + 756  p^{10} - 2248  p^{9} - 4943  p^{8} - 6300  p^{7} - 6894  p^{6} - 5994  p^{5} - 2448  p^{4} - 648  p^{3} + 324  p^{2} + 1296  p + 1296 \right)
\Big/
\left( 1296  {\left(p^{12} - p^{11} + p^{9} - p^{8} + p^{6} - p^{4} + p^{3} - p + 1\right)} {\left(p^{8} - p^{6} + p^{4} - p^{2} + 1\right)} \denomtimes {\left(p^{6} + p^{5} + p^{4} + p^{3} + p^{2} + p + 1\right)} {\left(p^{4} + p^{3} + p^{2} + p + 1\right)}^{3} {\left(p^{4} - p^{3} + p^{2} - p + 1\right)} \denomtimes {\left(p^{2} + p + 1\right)} {\left(p^{2} + 1\right)} p^{11} \right), 
			\end{dmath*}
		\end{minipage} & p \equiv 1 \pmod{3}\\
		& \\[-2ex]
		\begin{minipage}{11cm} \tiny
			\begin{dmath*} 
				\left( 144  p^{57} + 432  p^{56} + 1008  p^{55} + 1872  p^{54} + 3168  p^{53} + 4608  p^{52} + 6336  p^{51} + 8011  p^{50} + 9803  p^{49} + 11357  p^{48} + 13061  p^{47} + 14525  p^{46} + 16295  p^{45} + 17875  p^{44} + 19654  p^{43} + 21212  p^{42} + 23030  p^{41} + 24563  p^{40} + 26320  p^{39} + 27771  p^{38} + 29711  p^{37} + 30859  p^{36} + 31135  p^{35} + 31525  p^{34} + 31510  p^{33} + 29436  p^{32} + 28502  p^{31} + 28616  p^{30} + 26856  p^{29} + 25087  p^{28} + 25057  p^{27} + 23041  p^{26} + 19921  p^{25} + 18119  p^{24} + 16287  p^{23} + 13798  p^{22} + 12140  p^{21} + 10844  p^{20} + 9191  p^{19} + 7480  p^{18} + 5839  p^{17} + 4265  p^{16} + 2909  p^{15} + 1943  p^{14} + 1109  p^{13} + 590  p^{12} + 604  p^{11} + 372  p^{10} - 144  p^{9} - 87  p^{8} - 84  p^{7} - 678  p^{6} - 618  p^{5} - 144  p^{4} - 168  p^{3} - 156  p^{2} + 144  p + 144 \right)
\Big/
\left( 144  {\left(p^{12} - p^{11} + p^{9} - p^{8} + p^{6} - p^{4} + p^{3} - p + 1\right)} {\left(p^{8} - p^{6} + p^{4} - p^{2} + 1\right)} \denomtimes {\left(p^{6} + p^{5} + p^{4} + p^{3} + p^{2} + p + 1\right)} {\left(p^{4} + p^{3} + p^{2} + p + 1\right)}^{3} {\left(p^{4} - p^{3} + p^{2} - p + 1\right)} \denomtimes {\left(p^{2} + p + 1\right)} {\left(p^{2} + 1\right)} p^{11} \right),
			\end{dmath*}
		\end{minipage}
		& p \equiv 2 \pmod{3} \end{cases}\\[2ex]	
		\label{eq:rho_star}
	\rho^* &= \begin{minipage}{14cm} \small
			\begin{dmath*}  
				\left( 72  p^{34} + 216  p^{33} + 432  p^{32} + 720  p^{31} + 1008  p^{30} + 1224  p^{29} + 1260  p^{28} + 1296  p^{27} + 1152  p^{26} + 1080  p^{25} + 1068  p^{24} + 1032  p^{23} + 1104  p^{22} + 1092  p^{21} + 1116  p^{20} + 1089  p^{19} + 1104  p^{18} + 1088  p^{17} + 1126  p^{16} + 1149  p^{15} + 1017  p^{14} + 906  p^{13} + 830  p^{12} + 634  p^{11} + 360  p^{10} + 441  p^{9} + 378  p^{8} + 194  p^{7} + 280  p^{6} + 327  p^{5} + 93  p^{4} + 36  p^{3} + 60  p^{2} - 36  p - 72 \right)
\Big/
\left( 72  {\left(p^{8} - p^{6} + p^{4} - p^{2} + 1\right)} {\left(p^{4} + p^{3} + p^{2} + p + 1\right)}^{3}  {\left(p^{4} - p^{3} + p^{2} - p + 1\right)} {\left(p^{2} + 1\right)} {\left(p + 1\right)} p^{7} \right)
			\end{dmath*}
		\end{minipage}\\[2ex]
		\label{eq:sigma_4}
	\sigma_4 &= \begin{minipage}{14cm} \small
			\begin{dmath*}  
				\left( 72  p^{29} + 180  p^{28} + 396  p^{27} + 684  p^{26} + 1044  p^{25} + 1392  p^{24} + 1608  p^{23} + 1824  p^{22} + 1848  p^{21} + 1872  p^{20} + 1845  p^{19} + 1860  p^{18} + 1844  p^{17} + 1882  p^{16} + 1905  p^{15} + 1845  p^{14} + 1878  p^{13} + 2018  p^{12} + 2110  p^{11} + 2124  p^{10} + 2349  p^{9} + 2214  p^{8} + 1850  p^{7} + 1504  p^{6} + 1119  p^{5} + 525  p^{4} + 216  p^{3} + 96  p^{2} - 36  p - 72 \right)
\Big/
\left( 72  {\left(p^{8} - p^{6} + p^{4} - p^{2} + 1\right)} {\left(p^{4} + p^{3} + p^{2} + p + 1\right)}^{3} {\left(p^{4} - p^{3} + p^{2} - p + 1\right)} {\left(p^{2} + p + 1\right)} {\left(p^{2} + 1\right)} p^{2} \right)
			\end{dmath*}
		\end{minipage}\\[2ex]
		\label{eq:sigma_4_star}
	\sigma_4^* &= \begin{minipage}{14cm} \small
			\begin{dmath*}  
				\left( 72  p^{29} + 108  p^{28} + 288  p^{27} + 432  p^{26} + 648  p^{25} + 852  p^{24} + 960  p^{23} + 1104  p^{22} + 1092  p^{21} + 1116  p^{20} + 1089  p^{19} + 1104  p^{18} + 1088  p^{17} + 1126  p^{16} + 1149  p^{15} + 1089  p^{14} + 1122  p^{13} + 1262  p^{12} + 1354  p^{11} + 1368  p^{10} + 1593  p^{9} + 1530  p^{8} + 1202  p^{7} + 1000  p^{6} + 759  p^{5} + 309  p^{4} + 108  p^{3} + 60  p^{2} - 36  p - 72 \right)
\Big/
\left( 72  {\left(p^{8} - p^{6} + p^{4} - p^{2} + 1\right)} {\left(p^{4} + p^{3} + p^{2} + p + 1\right)}^{3} {\left(p^{4} - p^{3} + p^{2} - p + 1\right)} {\left(p^{2} + 1\right)} {\left(p + 1\right)} p^{3} \right)
			\end{dmath*}
		\end{minipage}\\[2ex]
		\label{eq:sigma_5}
	\sigma_5 &= \begin{cases}
		\begin{minipage}{11cm} \tiny
			\begin{dmath*}  
				\left( 819  p^{50} + 2691  p^{49} + 6309  p^{48} + 12573  p^{47} + 21573  p^{46} + 32895  p^{45} + 45387  p^{44} + 59238  p^{43} + 73080  p^{42} + 86742  p^{41} + 100547  p^{40} + 114472  p^{39} + 128439  p^{38} + 141579  p^{37} + 157131  p^{36} + 169247  p^{35} + 184741  p^{34} + 203094  p^{33} + 219096  p^{32} + 237726  p^{31} + 261800  p^{30} + 276904  p^{29} + 283923  p^{28} + 291645  p^{27} + 286281  p^{26} + 267993  p^{25} + 254943  p^{24} + 240039  p^{23} + 222678  p^{22} + 208152  p^{21} + 198396  p^{20} + 183383  p^{19} + 170848  p^{18} + 156267  p^{17} + 142677  p^{16} + 128205  p^{15} + 115607  p^{14} + 101365  p^{13} + 86670  p^{12} + 73512  p^{11} + 57564  p^{10} + 39824  p^{9} + 25201  p^{8} + 13608  p^{7} + 2430  p^{6} - 2106  p^{5} - 864  p^{4} - 1080  p^{3} - 540  p^{2} + 1296  p + 1296 \right)
\Big/
\left( 1296  {\left(p^{12} - p^{11} + p^{9} - p^{8} + p^{6} - p^{4} + p^{3} - p + 1\right)} {\left(p^{8} - p^{6} + p^{4} - p^{2} + 1\right)} \denomtimes {\left(p^{6} + p^{5} + p^{4} + p^{3} + p^{2} + p + 1\right)} {\left(p^{4} + p^{3} + p^{2} + p + 1\right)}^{3} {\left(p^{4} - p^{3} + p^{2} - p + 1\right)} \denomtimes {\left(p^{2} + p + 1\right)} {\left(p^{2} + 1\right)} p^{4} \right), 
			\end{dmath*}
		\end{minipage} & p \equiv 1 \pmod{3}\\
		& \\[-2ex]
		\begin{minipage}{11cm} \tiny
			\begin{dmath*} 
				\left( 91  p^{50} + 299  p^{49} + 701  p^{48} + 1397  p^{47} + 2429  p^{46} + 3767  p^{45} + 5347  p^{44} + 6982  p^{43} + 8684  p^{42} + 10358  p^{41} + 12035  p^{40} + 13648  p^{39} + 15243  p^{38} + 17183  p^{37} + 18907  p^{36} + 19903  p^{35} + 21877  p^{34} + 23878  p^{33} + 24684  p^{32} + 26774  p^{31} + 30344  p^{30} + 31608  p^{29} + 32719  p^{28} + 34705  p^{27} + 34273  p^{26} + 31873  p^{25} + 30647  p^{24} + 28815  p^{23} + 26470  p^{22} + 24668  p^{21} + 23516  p^{20} + 21719  p^{19} + 20152  p^{18} + 18367  p^{17} + 16793  p^{16} + 15005  p^{15} + 13607  p^{14} + 11765  p^{13} + 10094  p^{12} + 8524  p^{11} + 6708  p^{10} + 4464  p^{9} + 3081  p^{8} + 1788  p^{7} + 330  p^{6} - 186  p^{5} - 168  p^{3} - 156  p^{2} + 144  p + 144 \right)
\Big/
\left( 144  {\left(p^{12} - p^{11} + p^{9} - p^{8} + p^{6} - p^{4} + p^{3} - p + 1\right)} {\left(p^{8} - p^{6} + p^{4} - p^{2} + 1\right)} \denomtimes {\left(p^{6} + p^{5} + p^{4} + p^{3} + p^{2} + p + 1\right)} {\left(p^{4} + p^{3} + p^{2} + p + 1\right)}^{3} {\left(p^{4} - p^{3} + p^{2} - p + 1\right)} \denomtimes {\left(p^{2} + p + 1\right)} {\left(p^{2} + 1\right)} p^{4} \right),
			\end{dmath*}
		\end{minipage}
		& p \equiv 2 \pmod{3} \end{cases}\\[2ex]
		\label{eq:mu_prime}
	\mu' &= \begin{cases}
		\begin{minipage}{11cm} \small
			\begin{dmath*}  
				\left( 45  p^{20} - 18  p^{19} + 27  p^{18} + 18  p^{17} - 36  p^{16} - 12  p^{15} + 12  p^{14} + 36  p^{12} - 27  p^{11} - 6  p^{10} + 5  p^{9} - 30  p^{8} + 69  p^{7} - 29  p^{6} - 39  p^{5} + 81  p^{4} - 120  p^{3} + 60  p^{2} + 108  p - 72 \right)
\Big/
72  p^{20}, 
			\end{dmath*}
		\end{minipage} & p \equiv 1 \pmod{3}\\
		& \\[-2ex]
		\begin{minipage}{11cm} \small
			\begin{dmath*} 
				\left( 5  p^{20} - 2  p^{19} + 3  p^{18} + 2  p^{17} - 4  p^{16} + 4  p^{15} - 4  p^{14} + 4  p^{12} - 3  p^{11} + 2  p^{10} - 3  p^{9} + 2  p^{8} + 5  p^{7} - 13  p^{6} + 9  p^{5} + 9  p^{4} - 24  p^{3} + 12  p^{2} + 12  p - 8 \right)
\Big/
 8  p^{20},
			\end{dmath*}
		\end{minipage}
		& p \equiv 2 \pmod{3} \end{cases}\\[2ex]
		\label{eq:lambda}
	\lambda &= \begin{minipage}{14cm} \small
			\begin{dmath*}  
				\left( 72  p^{28} + 144  p^{27} + 288  p^{26} + 504  p^{25} + 744  p^{24} + 888  p^{23} + 1068  p^{22} + 1092  p^{21} + 1116  p^{20} + 1089  p^{19} + 1104  p^{18} + 1088  p^{17} + 1126  p^{16} + 1149  p^{15} + 1089  p^{14} + 1122  p^{13} + 1262  p^{12} + 1354  p^{11} + 1368  p^{10} + 1665  p^{9} + 1566  p^{8} + 1346  p^{7} + 1144  p^{6} + 903  p^{5} + 417  p^{4} + 180  p^{3} + 96  p^{2} - 36  p - 72 \right)
\Big/
\left( 72  {\left(p^{8} - p^{6} + p^{4} - p^{2} + 1\right)} {\left(p^{4} + p^{3} + p^{2} + p + 1\right)}^{3} {\left(p^{4} - p^{3} + p^{2} - p + 1\right)} {\left(p^{2} + 1\right)} {\left(p + 1\right)} p^{2} \right)
			\end{dmath*}
		\end{minipage}\\[2ex]
		\label{eq:sigma_5_prime}
	\sigma_5' &= \begin{cases}
		\begin{minipage}{11cm} \tiny
			\begin{dmath*}  
				\left( 91  p^{30} + 246  p^{29} + 478  p^{28} + 850  p^{27} + 1262  p^{26} + 1680  p^{25} + 1902  p^{24} + 2202  p^{23} + 2242  p^{22} + 2271  p^{21} + 2243  p^{20} + 2270  p^{19} + 2214  p^{18} + 2185  p^{17} + 2299  p^{16} + 2142  p^{15} + 2228  p^{14} + 2570  p^{13} + 2512  p^{12} + 2701  p^{11} + 3300  p^{10} + 2984  p^{9} + 2348  p^{8} + 2323  p^{7} + 1363  p^{6} + 288  p^{5} + 186  p^{4} + 60  p^{3} - 264  p^{2} - 72  p + 144 \right)
\Big/
\left( 144  {\left(p^{8} - p^{6} + p^{4} - p^{2} + 1\right)} {\left(p^{4} + p^{3} + p^{2} + p + 1\right)}^{3} {\left(p^{4} - p^{3} + p^{2} - p + 1\right)} \denomtimes {\left(p^{2} + 1\right)} {\left(p + 1\right)} p^{3} \right), 
			\end{dmath*}
		\end{minipage} & p \equiv 1 \pmod{3}\\
		& \\[-2ex]
		\begin{minipage}{11cm} \tiny
			\begin{dmath*} 
				\left( 91  p^{30} + 246  p^{29} + 478  p^{28} + 850  p^{27} + 1294  p^{26} + 1792  p^{25} + 2206  p^{24} + 2410  p^{23} + 2578  p^{22} + 2671  p^{21} + 2635  p^{20} + 2574  p^{19} + 2590  p^{18} + 2769  p^{17} + 2667  p^{16} + 2286  p^{15} + 2580  p^{14} + 2826  p^{13} + 2160  p^{12} + 2781  p^{11} + 3852  p^{10} + 3096  p^{9} + 2628  p^{8} + 3195  p^{7} + 1827  p^{6} + 432  p^{5} + 522  p^{4} + 252  p^{3} - 360  p^{2} - 72  p + 144 \right)
\Big/
\left( 144  {\left(p^{8} - p^{6} + p^{4} - p^{2} + 1\right)} {\left(p^{4} + p^{3} + p^{2} + p + 1\right)}^{3} {\left(p^{4} - p^{3} + p^{2} - p + 1\right)} \denomtimes {\left(p^{2} + 1\right)} {\left(p + 1\right)} p^{3} \right),
			\end{dmath*}
		\end{minipage}
		& p \equiv 2 \pmod{3} \end{cases}\\[2ex]
		\label{eq:sigma_5_prime_prime}
	\sigma_5'' &= \begin{cases}
		\begin{minipage}{11cm} \tiny
			\begin{dmath*}  
				\left( 819  p^{43} + 2376  p^{42} + 4599  p^{41} + 7424  p^{40} + 11091  p^{39} + 14515  p^{38} + 16101  p^{37} + 19341  p^{36} + 19532  p^{35} + 19542  p^{34} + 20605  p^{33} + 21042  p^{32} + 21969  p^{31} + 25640  p^{30} + 27075  p^{29} + 25531  p^{28} + 26901  p^{27} + 24399  p^{26} + 18864  p^{25} + 19800  p^{24} + 18900  p^{23} + 16200  p^{22} + 14580  p^{21} + 14148  p^{20} + 8478  p^{19} + 6102  p^{18} + 3492  p^{17} + 1476  p^{16} + 378  p^{15} + 378  p^{14} - 324  p^{13} + 468  p^{12} + 180  p^{11} - 864  p^{10} + 594  p^{9} + 2052  p^{8} + 684  p^{7} + 3096  p^{6} + 4590  p^{5} + 1674  p^{4} + 648  p^{3} + 1080  p^{2} - 648  p - 1296 \right)
\Big/
\left( 1296  {\left(p^{8} - p^{6} + p^{4} - p^{2} + 1\right)} {\left(p^{4} + p^{3} + p^{2} + p + 1\right)}^{3} {\left(p^{4} - p^{3} + p^{2} - p + 1\right)} \denomtimes {\left(p^{2} + 1\right)} {\left(p + 1\right)} p^{16} \right), 
			\end{dmath*}
		\end{minipage} & p \equiv 1 \pmod{3}\\
		& \\[-2ex]
		\begin{minipage}{11cm} \tiny
			\begin{dmath*} 
				\left( 91  p^{43} + 300  p^{42} + 607  p^{41} + 1024  p^{40} + 1531  p^{39} + 1903  p^{38} + 2329  p^{37} + 2581  p^{36} + 2404  p^{35} + 2686  p^{34} + 2725  p^{33} + 2166  p^{32} + 2497  p^{31} + 3216  p^{30} + 2739  p^{29} + 2943  p^{28} + 3897  p^{27} + 3279  p^{26} + 2544  p^{25} + 2904  p^{24} + 2676  p^{23} + 1992  p^{22} + 1908  p^{21} + 1764  p^{20} + 1134  p^{19} + 630  p^{18} + 324  p^{17} + 180  p^{16} + 90  p^{15} - 54  p^{14} - 36  p^{13} + 180  p^{12} - 108  p^{11} - 288  p^{10} + 162  p^{9} + 180  p^{8} - 180  p^{7} + 360  p^{6} + 558  p^{5} + 90  p^{4} + 72  p^{3} + 216  p^{2} - 72  p - 144 \right)
\Big/
\left( 144  {\left(p^{8} - p^{6} + p^{4} - p^{2} + 1\right)} {\left(p^{4} + p^{3} + p^{2} + p + 1\right)}^{3} {\left(p^{4} - p^{3} + p^{2} - p + 1\right)} {\left(p^{2} + 1\right)}\denomtimes {\left(p + 1\right)} p^{16} \right),
			\end{dmath*}
		\end{minipage}
		& p \equiv 2 \pmod{3} \end{cases}
\end{align}

\subsection{\texorpdfstring{$\tau_i$}{tau} values (see Lemma \ref{lem:taus})}

\begin{align}
	\label{eq:tau_6}
	\tau_6 &= \begin{cases}
		\begin{minipage}{11cm} \small 
			\begin{dmath*}  
				 {\left(6  p^{4} + 3  p^{3} + 5  p^{2} + 4  p + 4\right)} {\left(3  p^{3} + p^{2} + 2  p + 2\right)} 
\Big/
 9  {\left(p^{4} + p^{3} + p^{2} + p + 1\right)}^{2} , 
			\end{dmath*}
		\end{minipage} & p \equiv 1 \pmod{3}\\
		& \\[-2ex]
		\begin{minipage}{11cm} \small 
			\begin{dmath*} 
				 {\left(6  p^{4} + 3  p^{3} + 3  p^{2} + 4  p + 4\right)} {\left(3  p^{2} + 2\right)} {\left(p + 1\right)}
\Big/
 9  {\left(p^{4} + p^{3} + p^{2} + p + 1\right)}^{2} ,
			\end{dmath*}
		\end{minipage}
		& p \equiv 2 \pmod{3} \end{cases}\\[2ex]
		\label{eq:tau_7}
	\tau_7 &= \begin{cases}
		\begin{minipage}{11cm} \small
			\begin{dmath*}  
				\left( 72  p^{16} - 48  p^{15} + 12  p^{14} + 36  p^{12} - 27  p^{11} - 6  p^{10} + 5  p^{9} - 30  p^{8} + 69  p^{7} - 29  p^{6} - 39  p^{5} + 81  p^{4} - 120  p^{3} + 60  p^{2} + 108  p - 72 \right)
\Big/
 72  p^{17} , 
			\end{dmath*}
		\end{minipage} & p \equiv 1 \pmod{3}\\
		& \\[-2ex]
		\begin{minipage}{11cm} \small
			\begin{dmath*} 
				\left( 8  p^{16} - 4  p^{14} + 4  p^{12} - 3  p^{11} + 2  p^{10} - 3  p^{9} + 2  p^{8} + 5  p^{7} - 13  p^{6} + 9  p^{5} + 9  p^{4} - 24  p^{3} + 12  p^{2} + 12  p - 8 \right)
\Big/
 8  p^{17} ,
			\end{dmath*}
		\end{minipage}
		& p \equiv 2 \pmod{3} \end{cases}\\[2ex]
		\label{eq:tau_8}
	\tau_8 &= \begin{cases}
		\begin{minipage}{11cm} \small
			\begin{dmath*}  
				\left( 144  p^{17} - 120  p^{16} + 60  p^{15} - 12  p^{14} + 36  p^{13} - 63  p^{12} + 21  p^{11} + 11  p^{10} - 35  p^{9} + 99  p^{8} - 98  p^{7} - 10  p^{6} + 120  p^{5} - 201  p^{4} + 180  p^{3} + 48  p^{2} - 180  p + 72 \right)
\Big/
72  p^{18} , 
			\end{dmath*}
		\end{minipage} & p \equiv 1 \pmod{3}\\
		& \\[-2ex]
		\begin{minipage}{11cm} \small
			\begin{dmath*} 
				\left( 16  p^{17} - 8  p^{16} - 4  p^{15} + 4  p^{14} + 4  p^{13} - 7  p^{12} + 5  p^{11} - 5  p^{10} + 5  p^{9} + 3  p^{8} - 18  p^{7} + 22  p^{6} - 33  p^{4} + 36  p^{3} - 20  p + 8 \right)
\Big/
8  p^{18} ,
			\end{dmath*}
		\end{minipage}
		& p \equiv 2 \pmod{3} \end{cases}\\[2ex]
		\label{eq:tau_9}
	\tau_9 &= \begin{cases}
		\begin{minipage}{11cm} \tiny
			\begin{dmath*}  
				\left( 144  p^{44} + 336  p^{43} + 600  p^{42} + 936  p^{41} + 1416  p^{40} + 1704  p^{39} + 1968  p^{38} + 2160  p^{37} + 2328  p^{36} + 2136  p^{35} + 2280  p^{34} + 2472  p^{33} + 2592  p^{32} + 2784  p^{31} + 3115  p^{30} + 3030  p^{29} + 2806  p^{28} + 2650  p^{27} + 2366  p^{26} + 2256  p^{25} + 1998  p^{24} + 1914  p^{23} + 1642  p^{22} + 1335  p^{21} + 827  p^{20} + 566  p^{19} + 246  p^{18} + 25  p^{17} - 29  p^{16} + 6  p^{15} - 52  p^{14} + 98  p^{13} - 80  p^{12} - 83  p^{11} + 276  p^{10} + 200  p^{9} + 20  p^{8} + 523  p^{7} + 259  p^{6} - 288  p^{5} - 54  p^{4} + 12  p^{3} - 264  p^{2} - 72  p + 144 \right)
\Big/
\left( 144  {\left(p^{8} - p^{6} + p^{4} - p^{2} + 1\right)} {\left(p^{4} + p^{3} + p^{2} + p + 1\right)}^{3}  {\left(p^{4} - p^{3} + p^{2} - p + 1\right)} \denomtimes {\left(p^{2} + 1\right)} {\left(p + 1\right)} p^{18} \right), 
			\end{dmath*}
		\end{minipage} & p \equiv 1 \pmod{3}\\
		& \\[-2ex]
		\begin{minipage}{11cm} \tiny
			\begin{dmath*} 
				\left( 144  p^{44} + 432  p^{43} + 792  p^{42} + 1224  p^{41} + 1800  p^{40} + 2184  p^{39} + 2352  p^{38} + 2640  p^{37} + 2712  p^{36} + 2424  p^{35} + 2472  p^{34} + 2664  p^{33} + 2592  p^{32} + 2880  p^{31} + 3403  p^{30} + 3414  p^{29} + 3286  p^{28} + 3226  p^{27} + 2878  p^{26} + 2656  p^{25} + 2494  p^{24} + 2122  p^{23} + 1786  p^{22} + 1447  p^{21} + 835  p^{20} + 390  p^{19} + 238  p^{18} + 129  p^{17} - 45  p^{16} - 138  p^{15} + 108  p^{14} + 162  p^{13} - 432  p^{12} - 99  p^{11} + 540  p^{10} - 72  p^{9} - 180  p^{8} + 819  p^{7} + 243  p^{6} - 432  p^{5} + 90  p^{4} + 108  p^{3} - 360  p^{2} - 72  p + 144 \right)
\Big/
\left( 144  {\left(p^{8} - p^{6} + p^{4} - p^{2} + 1\right)} {\left(p^{4} + p^{3} + p^{2} + p + 1\right)}^{3} {\left(p^{4} - p^{3} + p^{2} - p + 1\right)} \denomtimes {\left(p^{2} + 1\right)} {\left(p + 1\right)} p^{18} \right),
			\end{dmath*}
		\end{minipage}
		& p \equiv 2 \pmod{3} \end{cases}
\end{align}

\subsection{\texorpdfstring{$\theta_i$}{theta} values (see Lemma \ref{lem:thetas})} 

\begin{align}
	\label{eq:theta_3}
	\theta_3 &= \begin{cases}
		\begin{minipage}{11cm} \small
			\begin{dmath*}  
				 {\left(10  p^{3} - p^{2} + 3  p - 6\right)} {\left(2  p^{2} - 3  p + 3\right)} {\left(p + 2\right)} 
\Big/
36  p^{6} , 
			\end{dmath*}
		\end{minipage} & p \equiv 1 \pmod{3}\\
		& \\[-2ex]
		\begin{minipage}{11cm} \small
			\begin{dmath*} 
				 {\left(2  p^{3} + p^{2} + p - 2\right)} {\left(2  p^{2} - 3  p + 2\right)} {\left(p + 1\right)} 
\Big/
 4  p^{6} ,
			\end{dmath*}
		\end{minipage}
		& p \equiv 2 \pmod{3} \end{cases}\\[2ex]
		\label{eq:theta_4}
	\theta_4 &= \begin{cases}
		\begin{minipage}{11cm} \small
			\begin{dmath*}  
				{\left(76  p^{6} - 14  p^{5} + 43  p^{4} - 90  p^{3} + 21  p^{2} - 36  p + 36\right)} {\left(2  p^{2} - 3  p + 3\right)} {\left(p + 2\right)} 
\Big/
 216  p^{9} , 
			\end{dmath*}
		\end{minipage} & p \equiv 1 \pmod{3}\\
		& \\[-2ex]
		\begin{minipage}{11cm} \small
			\begin{dmath*} 
				 {\left(4  p^{6} + 2  p^{5} + 3  p^{4} - 2  p^{3} - 3  p^{2} - 4  p + 4\right)} {\left(2  p^{2} - 3  p + 2\right)} {\left(p + 1\right)}
\Big/
8  p^{9},
			\end{dmath*}
		\end{minipage}
		& p \equiv 2 \pmod{3} \end{cases}\\[2ex]
		\label{eq:theta_6}
	\theta_6 &= \begin{cases}
		\begin{minipage}{11cm} \small
			\begin{dmath*}  
				 {\left(5  p^{4} + 3  p^{3} + 6  p^{2} + 4  p + 4\right)} {\left(p^{4} + 3  p^{3} + 2  p + 2\right)} 
\Big/
 9  {\left(p^{4} + p^{3} + p^{2} + p + 1\right)}^{2} , 
			\end{dmath*}
		\end{minipage} & p \equiv 1 \pmod{3}\\
		& \\[-2ex]
		\begin{minipage}{11cm} \small
			\begin{dmath*} 
				{\left(3  p^{4} + 3  p^{3} + 6  p^{2} + 4  p + 4\right)} {\left(3  p^{3} + 2\right)} {\left(p + 1\right)}
\Big/
 9  {\left(p^{4} + p^{3} + p^{2} + p + 1\right)}^{2},
			\end{dmath*}
		\end{minipage}
		& p \equiv 2 \pmod{3} \end{cases}\\[2ex]
		\label{eq:theta_7}
	\theta_7 &= \begin{cases}
		\begin{minipage}{11cm} \small
			\begin{dmath*}  
				\left( 2  p^{8} + 4  p^{7} - 6  p^{6} + 3  p^{5} + 2  p^{4} - 4  p^{3} + 2  p^{2} + 9  p - 6 \right)
\Big/
 6  p^{8} , 
			\end{dmath*}
		\end{minipage} & p \equiv 1 \pmod{3}\\
		& \\[-2ex]
		\begin{minipage}{11cm} \small
			\begin{dmath*} 
				\left( 2  p^{8} - 2  p^{6} + p^{5} + 2  p^{4} - 4  p^{3} + 2  p^{2} + 3  p - 2 \right)
\Big/
2  p^{8} ,
			\end{dmath*}
		\end{minipage}
		& p \equiv 2 \pmod{3} \end{cases}\\[2ex]
		\label{eq:theta_8}
	\theta_8 &= \begin{cases}
		\begin{minipage}{11cm} \small
			\begin{dmath*}  
				\left( 20  p^{11} + 20  p^{10} - 40  p^{9} + 54  p^{8} - 37  p^{7} + 27  p^{6} + 10  p^{4} - 3  p^{3} + 21  p^{2} - 72  p + 36 \right)
\Big/
36  p^{11}, 
			\end{dmath*}
		\end{minipage} & p \equiv 1 \pmod{3}\\
		& \\[-2ex]
		\begin{minipage}{11cm} \small
			\begin{dmath*} 
				\left( 4  p^{11} - 2  p^{8} - p^{7} + 7  p^{6} - 4  p^{5} - 6  p^{4} + 13  p^{3} - 3  p^{2} - 8  p + 4 \right)
\Big/
4  p^{11} ,
			\end{dmath*}
		\end{minipage}
		& p \equiv 2 \pmod{3} \end{cases}\\[2ex]
		\label{eq:theta_9}
	\theta_9 &= \begin{cases}
		\begin{minipage}{11cm} \tiny
			\begin{dmath*}  
				\left( 432  p^{37} + 2160  p^{36} + 3888  p^{35} + 6264  p^{34} + 9720  p^{33} + 12528  p^{32} + 13392  p^{31} + 16848  p^{30} + 19440  p^{29} + 21168  p^{28} + 22842  p^{27} + 25920  p^{26} + 24948  p^{25} + 23004  p^{24} + 22356  p^{23} + 20907  p^{22} + 19548  p^{21} + 19179  p^{20} + 20276  p^{19} + 19569  p^{18} + 20185  p^{17} + 17433  p^{16} + 16929  p^{15} + 13646  p^{14} + 10200  p^{13} + 7753  p^{12} + 8118  p^{11} + 5301  p^{10} + 5336  p^{9} + 6501  p^{8} + 4741  p^{7} + 1665  p^{6} + 2547  p^{5} + 450  p^{4} - 882  p^{3} - 540  p^{2} + 108  p - 648 \right)
\Big/
\left( 1296  {\left(p^{8} - p^{6} + p^{4} - p^{2} + 1\right)} {\left(p^{4} + p^{3} + p^{2} + p + 1\right)}^{3} {\left(p^{4} - p^{3} + p^{2} - p + 1\right)} {\left(p^{2} + 1\right)} {\left(p + 1\right)} p^{10} \right), 
			\end{dmath*}
		\end{minipage} & p \equiv 1 \pmod{3}\\
		& \\[-2ex]
		\begin{minipage}{11cm} \tiny
			\begin{dmath*} 
				\left( 144  p^{37} + 432  p^{36} + 720  p^{35} + 1080  p^{34} + 1656  p^{33} + 1872  p^{32} + 1872  p^{31} + 2160  p^{30} + 2448  p^{29} + 2448  p^{28} + 2826  p^{27} + 3264  p^{26} + 3252  p^{25} + 3036  p^{24} + 2964  p^{23} + 2803  p^{22} + 2592  p^{21} + 2515  p^{20} + 2644  p^{19} + 2665  p^{18} + 2389  p^{17} + 2221  p^{16} + 2041  p^{15} + 1414  p^{14} + 976  p^{13} + 817  p^{12} + 474  p^{11} + 229  p^{10} + 480  p^{9} + 453  p^{8} + 297  p^{7} + 453  p^{6} + 387  p^{5} + 66  p^{4} + 30  p^{3} + 36  p^{2} - 84  p - 72 \right)
\Big/
\left( 144  {\left(p^{8} - p^{6} + p^{4} - p^{2} + 1\right)} {\left(p^{4} + p^{3} + p^{2} + p + 1\right)}^{3} {\left(p^{4} - p^{3} + p^{2} - p + 1\right)} {\left(p^{2} + 1\right)} {\left(p + 1\right)} p^{10} \right),
			\end{dmath*}
		\end{minipage}
		& p \equiv 2 \pmod{3} \end{cases}
\end{align}

\subsection{Small primes \texorpdfstring{$p$}{p} (see \S \ref{sec:small_primes})}

\begin{align}
	\label{eq:rho7} \rho(7) &= \frac{63104494755178622851603292623187277054743730183645677893972}{64083174787206696882429945655801281538844149896400159815375} \approx 0.98472 \\
	\label{eq:rho13} \rho(13) &= \frac{7877728357244577414025901931296747409682076255666526984515273526822853}{7890643570620106747776737292792780623510727026420779539893772399701475} \approx 0.99836\\
	\label{eq:rho19} \rho(19) &= \scalebox{0.85}{$ \dfrac{3122673715489206150449285868243361150392235799365815266879438393279346795671}{3123410013311365155035964479837966797560851333614271490136481337080636454180}$} \approx 0.99976\\
	\label{eq:rho31} \rho(31) &= \scalebox{0.8}{$ \dfrac{9196796457678318869139089936786462146535210039832850454297877482020635073857159758299}{9196865061587843544830989041473808798913128587425995645857828572610918436035833907250}$} \approx 0.999992 \\
	\label{eq:rho37} \rho(37) &= \scalebox{0.75}{$ \dfrac{171128647900820194784458101787952920169924464886519055453844647154184805036447476640345735119}{171128889636157060536894474187017088464271236509977199491208939449738127658679723715588944500}$} \approx 0.999998 \\
	\label{eq:rho43} \rho(43) &= \scalebox{0.7}{$ \dfrac{84000121343283090388653356431804100707331364779290664490547105768867844862712134447832720508750281}{84000151671513555191647712567596101710800846209116830568013729377404991150901973105093039939237500}$} \approx 0.9999996
\end{align}

\bibliographystyle{alpha}
\bibliography{sec_densities}

\end{document}